\documentclass[11pt]{article}

\usepackage{layout}
\usepackage{amssymb}
\usepackage{epsfig}
\usepackage{graphicx}
\usepackage{color}
\usepackage{tcolorbox}
\usepackage{latexsym}

\usepackage[left=2cm,right=2cm,top=2cm,bottom=2cm]{geometry}

\usepackage{amssymb}
\usepackage{algorithmicx}
\usepackage[ruled]{algorithm}
\usepackage{algpseudocode}
\usepackage{algpascal}
\usepackage{algc}

\alglanguage{pseudocode}
\usepackage{graphicx}
\usepackage{amsmath,epsfig,graphics}

\newenvironment{proof}[1][Proof]{\textbf{#1.} }{\ \rule{0.5em}{0.5em}}
\newcommand{\R}{I\!\!R}

\newtheorem{theorem}{Theorem}[section]

\newtheorem{proposition}[theorem]{Proposition}
\newtheorem{remark}[theorem]{Remark}

\newcommand{\Frac}[2] {\frac{\textstyle #1} {\textstyle #2}}

\newcommand{\Min}  {\mathop{\rm Min}}

\usepackage{collectbox}

\makeatletter
\newcommand{\sqbox}{%
    \collectbox{%
        \@tempdima=\dimexpr\width-\totalheight\relax
        \ifdim\@tempdima<\z@
            \fbox{\hbox{\hspace{-.5\@tempdima}\BOXCONTENT\hspace{-.5\@tempdima}}}%
        \else
            \ht\collectedbox=\dimexpr\ht\collectedbox+.5\@tempdima\relax
            \dp\collectedbox=\dimexpr\dp\collectedbox+.5\@tempdima\relax
            \fbox{\BOXCONTENT}%
        \fi
    }%
}
\makeatother

 \usepackage{tikz}
\usepackage[framemethod=default]{mdframed}

\begin{document}

\title{Fast and Stable Schemes for Phase Fields Models}
\maketitle
\centerline{\scshape Matthieu Brachet$^{1}$ and Jean-Paul Chehab$^2$ 
}
\medskip
\centerline{$^{1}${\footnotesize Laboratoire LJK ({\small UMR} CNRS 5224) and {\small INRIA} Project {\small AIRSEA}- B\^atiment IMAG, Universit\'e Grenoble Alpes  }}
  \centerline{{\footnotesize 700 Avenue Centrale,
  Campus de Saint Martin d'H\`eres }}
\centerline{{\footnotesize 38401 Domaine Universitaire de Saint-Martin-d'H\`eres }}
\centerline{{\footnotesize {\tt matthieu.brachet@inria.fr}}}
\centerline{$^{2}${\footnotesize Laboratoire LAMFA ({\small UMR} CNRS 7352),
 Universit\'e de Picardie Jules Verne}}
  \centerline{{\footnotesize 33 rue Saint Leu, 80039 Amiens C\'edex, France}}
    \centerline{{\footnotesize {\tt Jean-Paul.Chehab@u-picardie.fr}}}


\maketitle

\begin{abstract}
We propose and analyse new stabilized time marching schemes for Phase Fields model such as Allen-Cahn and Cahn-Hillard equations, when discretized in 
space with high order finite differences compact schemes.
The stabilization applies to semi-implicit schemes  for which the linear part is simplified using sparse pre-conditioners. The new methods allow 
to significant obtain a gain of CPU time. The numerical illustrations we give concern applications on pattern dynamics and on image processing (inpainting, segmentation) in two and three dimension cases.
\end{abstract}
{\small
}\\

\section{Introduction}
Diffuse interface dynamics governed by Phase fields equations, such as Allen-Cahn's or Cahn-Hilliard's, play an important role in
a large number of applications: let us cite
 \cite{AllenCahn1,AllenCahn2,Emmerich,Provatras} in material science, 
\cite{Benes,Bertozzi1,Bertozzi2,Fakih,Lee,LiLee,LiJeongChoiLeeKim} in image processing, \cite{LiLee,MPierreARougirel}  in chemistry 
\cite{ChehabFrancoMammeri,JiangShi} or in ecology and in medicine, the list being non-exhaustive. In addition, the interest for these models in the mathematical analysis point of view is considerably developed since the last three decades, notably in the study of the long time behavior of the solutions,  see e.g. \cite{Elliott,ElliottStuart,TemamBook}, making the simulation of Phase fields models a key  issue. \\

The numerical integration of such reaction-diffusion equations can be a delicate task: it needs to recover at the discrete level intrinsic properties of the solution (Energy diminishing, maximum principle) and the presence of small parameter $\epsilon >0$ (typically, the interphase length) can generate practical difficulties in the iterations processes with a  hard time step restriction, even for fully-implicit schemes ; this is due on the way the fixed points problems are solved at each iteration.\\

The construction of a robust (stable) and efficient (fast) scheme lies on the balance between the advantages and the drawbacks of  implicit (stable but costly) and of explicit (fast but with often hard stability condition) times-marching schemes. For instance,  the simple Forward Euler's, can be used  only for small time steps; this restriction can be very important, e.g., when considering heat-equation, the basic linear part of reaction-diffusion equations. This restriction allows to prevent the expansion of high mode components, the ones that lead to the divergence of the scheme. A way to enhance the stability region while solving a relatively simple linear system is to  introduce a proper approximation to an unconditionally stable scheme. Consider, e.g., Backward Euler's applied to the discretized Heat equation:
\begin{eqnarray}\label{BE_HEAT}
\Frac{u^{(k+1)}-u^{(k)}}{\Delta t}+Au^{(k+1)}=0,
\end{eqnarray}
where $A$ is the stiffness matrix, $\Delta t>0$, the time step; here $u^{(k)}$ is the approximation of the solution at time $t=k\Delta t$ in the spatial approximation space. To simplify the linear system that must be solved at each step,  one replaces $Au^{(k+1)}$ by
$\tau B (u^{(k+1)}-u^{(k)}) +Au^{(k)}$, where $\tau\ge 0$ and where $B$ is a pre-conditioner of $A$; the new scheme reads as
\begin{eqnarray}\label{RSS_HEAT}
\Frac{u^{(k+1)}-u^{(k)}}{\Delta t}+\tau B (u^{(k+1)}-u^{(k)}) +Au^{(k)}=0.
\end{eqnarray}
This stabilization procedure, also called RSS scheme (Residual Smoothing Scheme), was introduced independently by \cite{AverbuchCohenIsraeli} and \cite{CostaDettoriGottliebTemam} (in the multilevel case), see also \cite{BrachetChehabJSC} for recent developments. It allows to take large time steps while simplifying the linear problem to solve at each step: in that way the stability is enhanced and in addition a save of computation time can be obtained as respect to the classical backward Euler's scheme. Of course there exist many different stabilization procedures that can be applied to a large variety of schemes used for reaction-diffusion equations, see, e.g. \cite{Dutykh1,Dutykh2}, particularly those based on hyperbolic perturbations that we will not consider here. The stabilized scheme for a reaction-diffusion equation writes as
\begin{eqnarray}\label{RSS_REAC_DIFF}
\Frac{u^{(k+1)}-u^{(k)}}{\Delta t}+\tau B (u^{(k+1)}-u^{(k)}) +Au^{(k)} +f(u^{(k)})=0.
\end{eqnarray}
It corresponds to a stabilized semi-implicit Euler scheme for, e.g., Allen-Cahn equations; in the same way,  using the stabilization procedure,  we can consider coupled systems as
\begin{eqnarray}\label{RSS_CH}
\Frac{u^{(k+1)}-u^{(k)}}{\Delta t}+\tau B (\mu^{(k+1)}-\mu^{(k)}) +A\mu^{(k)} =0,\\
\mu^{(k+1)}=\tau B (u^{(k+1)}-u^{(k)}) +Au^{(k)} +f(u^{(k)})=0.
\end{eqnarray}
The technique can then applied to high order or coupled problems such as Cahn-Hilliard's. It must be noticed that this stabilization procedure allows to recover the same steady states as the original scheme, this is an important property when considering, e.g.,  inpainting or image segmentation problems. \\

The aim of this article is to propose and analyze fast finite differences schemes for phase fields with a focus on  Allen-Cahn and Cahn-Hilliard equations, when the space discretization is realized with finite differences compact schemes. The new methods combine high order compact finite differences schemes for the discretization in space together with a stabilization of explicit time schemes implemented by using low coast pre-conditioners of the linear term. 

The article is organized as follows: in Section 2 we consider the linear case, we recall the principle of the stabilization (RSS- scheme) and derive stability results for a  number of time schemes that will be used in the non linear case. After that, in Section 3, then in Section 4, we introduce and study new stabilized schemes for Allen-Cahn's (then Cahn-Hilliard's) equation. We give in particular conditions to obtain energy diminishing schemes. In Section 5 we present numerical illustrations on pattern dynamics, image segmentation and inpainting. 
%
%
\section{Stabilized schemes in the linear case}
We first give here stability results for stabilized-Schemes derived from time marching method in the linear case; these schemes will be used to build new methods for solving nonlinear time dependent problems, as presented in Sections 3 and 4.
\subsection{Explicit Schemes and stabilization}
Consider the Heat equation
\begin{eqnarray}\label{HEAT_EQUATION}
\Frac{\partial u}{\partial t} -\Delta u =0, & x\in\Omega, \ t >0,\\
u=0 & x \in \partial \Omega , t>0,\\
u(x,0)=u_0(x) & x \in \Omega.
\end{eqnarray}
They are several ways to express the stability of a scheme, depending on the norm one considers to measure the boundedness of the sequence of time approximations of the solution. We will focus on two following stability notions
\begin{itemize}
\item Stability in Energy (a consequence of energy time-diminishing): 
$$
\displaystyle{\sum_{|\alpha|= 0}^m\aleph_\alpha\|D^{\alpha}u(t)\|^2_{L^2{\Omega)}}}\le\displaystyle{\sum_{|\alpha|= 0}^m\aleph_\alpha\|D^{\alpha}u(t')\|^2_{L^2{\Omega)}}}, \ \forall t >t' ,
$$
with $\aleph_{\alpha} \ge 0, \ \sum \aleph_{\alpha} >0$.
\item $L^{\infty}$ Stability (a consequence of the maximum principle):
$$
\exists L> 0 / \|u(t)\|_{L^{\infty}}\le L, \forall t \ge 0.
$$
\end{itemize}
The space discretization of (\ref{HEAT_EQUATION}) leads to the differential system
\begin{eqnarray}\label{DISCR_HEAT_EQUATION}
\Frac{d u}{d t} +Au  =0, & x\in\Omega, \ t >0,\\
u(0)=u_0. & 
\end{eqnarray}
Here $A$ is the stiffness matrix. The time numerical integration of  (\ref{DISCR_HEAT_EQUATION}) produces a sequence of vectors $u^k\simeq u(k\Delta t)$, and we define the stability of the schemes as 
\begin{itemize}
\item Stability in Energy (discrete Energy diminishing), e.g.,
$$
\Frac{1}{2}<Au^{k+1},u^{k+1}>\le\Frac{1}{2}<Au^k,u^k>, \forall k\ge 0,
$$
\item $L^{\infty}$ Stability
$$
\exists L> 0, / \max_{i}|u_i^k|\le L, \forall k.
$$
\end{itemize}
These notions of stability will be used also in the nonlinear case, especially for Allen-Cahn's equation.\\

Let us first recall a simple but useful result, \cite{BrachetChehabJSC}.
Let $A$ and $B$ be two $n\times n$ Symmetric Semi-Positive Definite matrices (SSPD). We now define the following hypothesis ${\cal H}$ that will be used from now on:
\begin{tcolorbox}
Hypothesis ${\cal H}$: 
\begin{itemize}
\item[i.] Assume that $Ker(A)=Ker(B)=W$
\item[ii.] Assume that there exist two strictly positive constants $\alpha$ and $\beta$ such that
\begin{eqnarray}
\label{Hyp_H}
\alpha <Bu,u> \le <A u, u> \le \beta <Bu,u>, \ \forall  u\in \R^n
\end{eqnarray}
\end{itemize}
\end{tcolorbox}
\begin{remark}
Condition $ii.$ is obviously satisfied for every $u\in W$ and can be alternatively expressed as
$$
\alpha <Bu,u> \le <A u, u> \le \beta <Bu,u>, \ \forall  u\in W^T.
$$
The coefficients $\alpha$ and $\beta$ can depend on the dimension $n$.
Also both $A$ and $B$ are SDP on $W^T$.
\end{remark}
The simplest stabilized scheme is obtained from Backward Euler's as\begin{eqnarray}\label{RSS_HEAT}
\Frac{u^{(k+1)}-u^{(k)}}{\Delta t}+\tau B (u^{(k+1)}-u^{(k)}) +Au^{(k)}=0,
\end{eqnarray}
where $\tau \ge 0$; Forward Euler is recovered for $\tau=0$ while Backward Euler's is obtained for $\tau=1$ and $B=A$. We have the stability conditions 
\begin{proposition}
Assume that $A$ and $B$ are two symmetric semi-positive definite matrices and that hypothesis (\ref{Hyp_H}) holds. We set $W=Ker(A)=Ker(B)$.
Then, we have the following stability conditions:
\begin{itemize}
\item If $\tau\ge \Frac{\beta}{2}$, the schemes (\ref{RSS_ADI1}) and (\ref{RSS_ADI2}) are unconditionally stable (i.e. stable $\forall \ \Delta t >0$),
\item If $\tau < \Frac{\beta}{2}$, then the scheme is stable for
$
0<\Delta t < \Frac{2}{\left(1-\Frac{2\tau}{\beta}\right)\rho(A)}.
$
\end{itemize} 
Moreover $u^{(k+1)}-u^{(k)}\in W^\perp, \forall k\ge 0$.
\label{RSS_Stab_lin}
\end{proposition}
\begin{proof}
The last assertion is obtained directly by using the symmetry of $A$ and $B$ and taking the scalar product with any element of $W$.
Using this property, the rest of the proof is then  similar to the one given  in \cite{BrachetChehabJSC} when $W=\{0\}$. 
\end{proof}
\\
In \cite{BrachetChehabJSC} was considered Homogeneous Dirichlet Boundary conditions ; the double inequality $$\alpha <Bu,u> \le <A u, u> \le \beta <Bu,u>, \ \forall  u\in W^\perp$$ in hypothesis (\ref{Hyp_H})
allows to consider also periodic or homogeneous Neumann Boundary conditions which are of interest for Phase Fields models, in that case $W=\{(1,1,\cdots, 1)^T\}$. We will note in the sequel of the paper
 ${\bf 1}=(1,1,\cdots, 1)$; ${\bf 1}$ is a line vector while ${\bf 1}^T$ is a column one.

Of course, instead of Euler's method,  we can consider second order schemes such Crank Nicolson's and apply the same stabilization procedure:

\begin{proposition}
Consider the Stabilized Crank Nicolson Scheme
\begin{eqnarray*}\label{RSS_CN}
u^{(k+1)}-u^{(k)}
+\tau\Frac{\Delta t}{2} B(u^{(k+1)}-u^{(k)})+\Delta t Au^k=0
\end{eqnarray*}
Assume that $A$ and $B$ are two symmetric semi-positive definite matrices and that hypothesis (\ref{Hyp_H}) holds. We set $W=Ker(A)=Ker(B)$.
\begin{itemize}
\item If $\tau\ge \beta$, the scheme (\ref{RSS_CN}) is unconditionally stable (i.e. stable $\forall \ \Delta t >0$)
\item If $\tau < \beta$, then the scheme is stable for
$
0<\Delta t < \Frac{2}{\left(1-\Frac{\tau}{\beta}\right)\rho(A)}.
$
\end{itemize} 
Moreover $u^{(k+1)}-u^{(k)}\in W^\perp, \forall k\ge 0$.
\label{RSS_Stab_CN}
\end{proposition}
\begin{proof} It suffices to replace $\tau$ by $\Frac{\tau}{2}$ in Proposition (\ref{RSS_Stab_lin}).
\end{proof}
\begin{remark}
Similar result can be derived for stabilized Gear's scheme: $$\Frac{1}{2\Delta t}(3u^{(k+1)}-4u^{(k)}+u^{(k-1)})
+\tau B(u^{(k+1)}-u^{(k)})+Au^k=0.$$
\end{remark}
\subsection{Discrete Maximum Principle}
Another important  stability property, crucial in a number of applications, is the $L^{\infty}$ stability guaranteed by a maximum principle. We have the
\begin{proposition}
Assume that $f\ge 0$ and $u^{(0)} \ge 0$.
Assume in addition that  $Id+\Delta t B$ is a $M$-matrix for all $\Delta t >0$. Set $D=\tau B-A$ and $I=\{i\in\{1,\cdots,n\} / D_{ii}< 0\}$. 
If
$$
D_{i,j} \ge 0, \forall i,j=1,\cdots, N, i\neq j \mbox{ and } 0< \Delta t < \Frac{1}{\max_{i\in I}\mid D_{ii}\mid}, \ i \in I,
$$
then $u^{(k)}\ge 0, \forall k$.
\label{Heat_Max_Princip}
\end{proposition}
\begin{proof}
Let $k$ be fixed and assume that $u^{(k)} \ge 0$. We have
$$
(Id+\tau \Delta t B)u^{(k+1)}=(Id+\Delta t (\tau  B-A))u^{(k)} +\Delta t f.
$$
The matrix $Id+\tau \Delta t B$ is a M-matrix since $B$ is also one, as a direct consequence and a sufficient condition to have $u^{(k+1)}\ge 0$ is 
$(Id+\Delta t (\tau  B-A))u^{(k)} \ge 0$, this is guaranteed when the matrix 
$R=Id+\Delta t (\tau  B-A)=Id +\Delta t D$ has all positive entries, say
$$
\left\{
\begin{array}{ll}
1+\Delta t (\tau  B_{ii}-A_{ii}) \ge 0 & i=1,\cdots, n,\\
 &\\
 \tau  B_{ij}-A_{ij} \ge 0 & i=1,\cdots, n\, i\neq j .\\
\end{array}
\right.
$$
Hence the result by a simple induction.
\end{proof}
\\

It is usual that the discrete Maximum principle is satisfied for small values of $\Delta t$; we recover particularly the conditions that must be satisfied for the Crank-Nicolson scheme taking $\tau=\Frac{1}{2}$ and $A=B$, see \cite{MortonMayers}.
\subsection{ADI Stabilized Scheme}
A important issue for a fast simulation of parabolic equations is the use of splitting methods. We give here stabilized versions of classical ADI schemes.
Consider the linear differential system
$$
\Frac{dU}{dt}+AU=0,
$$
with $A=A_1+A_2$. Let $B_1$ and $B_2$ be pre-conditioners of $A_1$ and $A_2$ respectively and $\tau_1, \ \tau_2$ two positive real numbers. All the matrices are supposed to be symmetric positive definite.
We introduce the stabilized ADI-schemes

\begin{eqnarray}
\Frac{u^{(k+1/2)}-u^{(k)}}{\Delta t} +\tau_1 B_1 (u^{(k+1/2)}-u^{(k)}) = -A_1 u^{(k)},\\
\Frac{u^{(k+1)}-u^{(k+1/2)}}{\Delta t} +\tau_2 B_2 (u^{(k+1)}-u^{(k+1/2)}) = -A_2 u^{(k+1/2)},
\label{RSS_ADI1}
\end{eqnarray}
and the Strang's Splitting
\begin{eqnarray}
\Frac{u^{(k+1/3)}-u^{(k)}}{\Delta t/2} +\tau_1 B_1 (u^{(k+1/3)}-u^{(k)}) = -A_1 u^{(k)},\\
\Frac{u^{(k+2/3)}-u^{(k+1/3)}}{\Delta t} +\tau_2 B_2 (u^{(k+2/3)}-u^{(k+1/3)}) = -A_2 u^{(k+1/3)},\\
\Frac{u^{(k+1)}-u^{(k+2/3)}}{\Delta t/2} +\tau_1 B_1 (u^{(k+1)}-u^{(k+2/3)}) = -A_1 u^{(k+2/3)}.
\label{RSS_ADI2}
\end{eqnarray}
Consider now the general decompositions $A=\displaystyle{\sum_{i=1}^m A_i}$ and $B =\displaystyle{\sum_{i=1}^m B_i}$. We can  associate the following RSS-splitting scheme:
\begin{eqnarray}
\Frac{u^{(k+i/m)}-u^{(k+(i-1)/m)}}{\Delta t} +\tau_i B_i(u^{(k+i/m)}-u^{(k+(i-1)/m)}) = -A_i u^{(k+(i-1)/m)}.
\label{RSS_ADIm}
\end{eqnarray}
As a direct consequence of proposition \ref{RSS_Stab_lin}, we can prove the following result:
\begin{proposition}
Under hypothesis (\ref{Hyp_H}) applied to each pair $(A_i,B_i)$, and that $W_i=Ker(A_i)=ker(B_i)=W=Ker(A)=Ker(B)$. Scheme (\ref{RSS_ADIm}) is stable under the following conditions:
\begin{itemize}
\item If $\tau_i\ge \Frac{\beta_i}{2}$,$i=1,\cdots, m$ the scheme (\ref{RSS_ADI1}) is unconditionally stable (i.e. stable $\forall \ \Delta t >0$),
\item If $\tau_i < \Frac{\beta_i}{2}$, $i=1,\cdots, m$, then the scheme is stable for
$$
0<\Delta t < \displaystyle{\min_{1\le i\le m}\left(\Frac{2}{\left(1-\Frac{2\tau_i}{\beta_i}\right)\rho(A_i)}\right)}.
$$
Morevover  $u^{(k+1)}-u^{(k)}\in W^\perp, \ \forall k\ge 0$.
\end{itemize} 
\label{RSS_Stab_ADI}
\end{proposition}
\begin{proof}
The last assertion is obtained directly by using the symmetry of $A_i$ and $B_i$ and taking the scalar product with any element of $W$.
The rest of the proof is obtained applying proposition \ref{RSS_Stab_lin} to each system.  
\end{proof}
\subsection{Discretization in space}
Before presenting the stabilized schemes for phase fields models, we give hereafter some numerical
 illustrations on linear problems when discretized in space by finite differences compact schemes, focusing on Neumann boundary conditions. We propose as in \cite{BrachetChehabJSC} to use a (lower) second order discretization matrix for preconditioning the underlining matrices. We first consider the Elliptic
 Neumann problem then the Heat equation.
\subsubsection{Finite difference Preconditioning for compact schemes and the Neumann problem}
We use finite differences compact schemes for the discretization in space: in two words these schemes are nonlocal (they have an implicit part), and they allows to reach an accuracy comparable to the spectral one; we refer to the classical book of Collatz \cite{Collatz} and to the seminal paper of Lele \cite{Lele}.\\

We recall briefly the stricture of the  compact schemes: let $U=(U_1,\cdots,U_n)^T$ denotes a vector whose the components are the approximations of a regular function $u$ at (regularly spaced) grid points $x_i=ih$, $i=1,\cdots, n$.  We compute approximations of $V_i={\cal L}(u)(x_i)$ as solution of a system
$$
P . V= Q U,
$$
so the approximation matrix is formally $B=P^{-1}Q$.

  There are several ways to build the compact difference scheme associated to a differential operator, especially when non-periodic boundary conditions are present.
To obtain a high accuracy at the boundary points while preserving the implicit par of the scheme, an extrapolation scheme is used, see \cite{Lele} and  \cite{BrachetChehabJSC}  for the second derivative associated to Neumann Boundary conditions with a fourth order of accuracy:
Now applying the same approach, we can consider fourth order compact schemes for the second derivative with associated homogeneous Neumann Boundary conditions

$$ P= \begin{pmatrix}
1 & \frac{1}{10} &   &   &   \\ 
\frac{1}{10} & 1 & \frac{1}{10} &   &   \\ 
  & \ddots & \ddots & \ddots &   \\ 
  &  & \frac{1}{10} & 1 & \frac{1}{10} \\ 
 &  &  & \frac{1}{10} & 1
\end{pmatrix}, $$ 

and $$ Q = \frac{1}{h^2} \begin{pmatrix}
a_1 & a_2 & a_3 & a_4 & a_5 &   &   \\ 
-\frac{6}{5} & \frac{12}{5} & -\frac{6}{5} &   &   &   &   \\ 
  & -\frac{6}{5} & \frac{12}{5} & -\frac{6}{5} &   &   &   \\ 
  &   & \ddots & \ddots & \ddots &   &   \\ 
  &   &   & -\frac{6}{5} & \frac{12}{5} & -\frac{6}{5} &   \\ 
  &   &   &   & -\frac{6}{5} & \frac{12}{5} & -\frac{6}{5} \\ 
  &   & a_{N-4} & a_{N-3} & a_{N-2} & a_{N-1} & a_N
\end{pmatrix}, $$

with

$$ \left\lbrace \begin{array}{rcccl}
a_1&=&a_{N}&=&\frac{2681}{480},   \\
a_2&=&a_{N-1}&=&-\frac{32}{3},   \\
a_3&=&a_{N-2}&=&\frac{113}{40},   \\
a_4&=&a_{N-3}&=&-\frac{13}{15},   \\
a_5&=&a_{N-4}&=&\frac{59}{480}.   \\
\end{array}  \right.  $$
\\

However this presents a drawback: the conservation of the mean value for the discrete heat resolvant is not formally satisfied and instabilities can occur when considering problems in which the mass must be conserved as for the Heat equation (se hereafter) or the Pattern dynamics for Cahn-Hilliard (but this is not the case for inpainting or the segmentation) as shown in Section 4: in these situations the mean value conservation  must be forced using a projection step; a modified RSS scheme is then proposed, see hereafter. 

The second approach to derive boundary formula is to leave unchanged the explicit part. The following scheme (labelled CS2) is second order accurate at the boundary and fourth order accurate at  the interior points:
$$ P_{CS2}= \begin{pmatrix}
\frac{2}{5}& \frac{1}{5} &   &   &   \\ 
\frac{1}{10} & 1 & \frac{1}{10} &   &   \\ 
  & \ddots & \ddots & \ddots &   \\ 
  &  & \frac{1}{10} & 1 & \frac{1}{10} \\ 
 &  &  & \frac{2}{5} & \frac{2}{5}
\end{pmatrix}
\mbox{ and  } Q_{CS2} = \frac{1}{h^2} \begin{pmatrix}
-\frac{6}{5} & \frac{6}{5} & &   &   &   &   \\ 
-\frac{6}{5} & \frac{12}{5} & -\frac{6}{5} &   &   &   &   \\ 
  & -\frac{6}{5} & \frac{12}{5} & -\frac{6}{5} &   &   &   \\ 
  &   & \ddots & \ddots & \ddots &   &   \\ 
  &   &   & -\frac{6}{5} & \frac{12}{5} & -\frac{6}{5} &   \\ 
  &   &   &   & -\frac{6}{5} & \frac{12}{5} & -\frac{6}{5} \\ 
  &   & & & & \frac{6}{5} & -\frac{6}{5}
\end{pmatrix}. $$

We represent in Figure \ref{COMP_CS} (left) the comparison of the spectrum for each of the discretization schemes (second order, Lele's and CS2). We observe the better approximation of the compact schemes to the exact eigenvalues of the operator; particularly the eigenvalues of CS2 and Lele's schemes are very close expect for the last two ones. The lower eigenvalue of CS2 matrix is $0$ while the one of Lele's matrix is $-4.5418e-11$.  In Figure \ref{COMP_CS} (right) we displayed the $L^{\infty}$ error for each discretization scheme when considering the function $u(x)=\cos(x(1-x))$, for different values of the step-size $h$. We observe a good accordance to the order of accuracy looking to the slopes: $2.0116$ (second order scheme and CS2) and $3.8480$ for Lele's compact scheme.

\begin{figure}[h!!]
\begin{center}
\includegraphics[height=5.5cm]{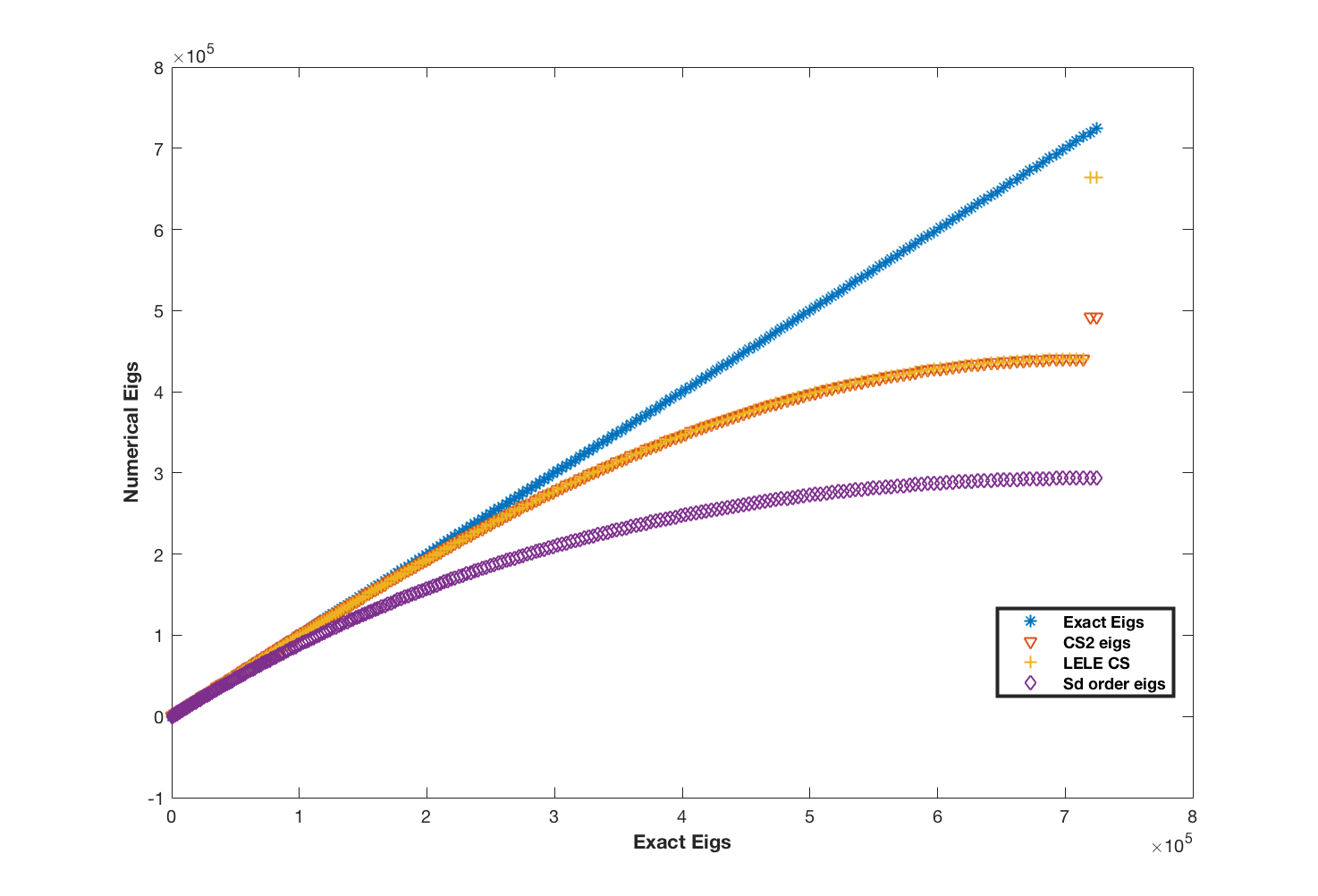}
\includegraphics[height=5.5cm]{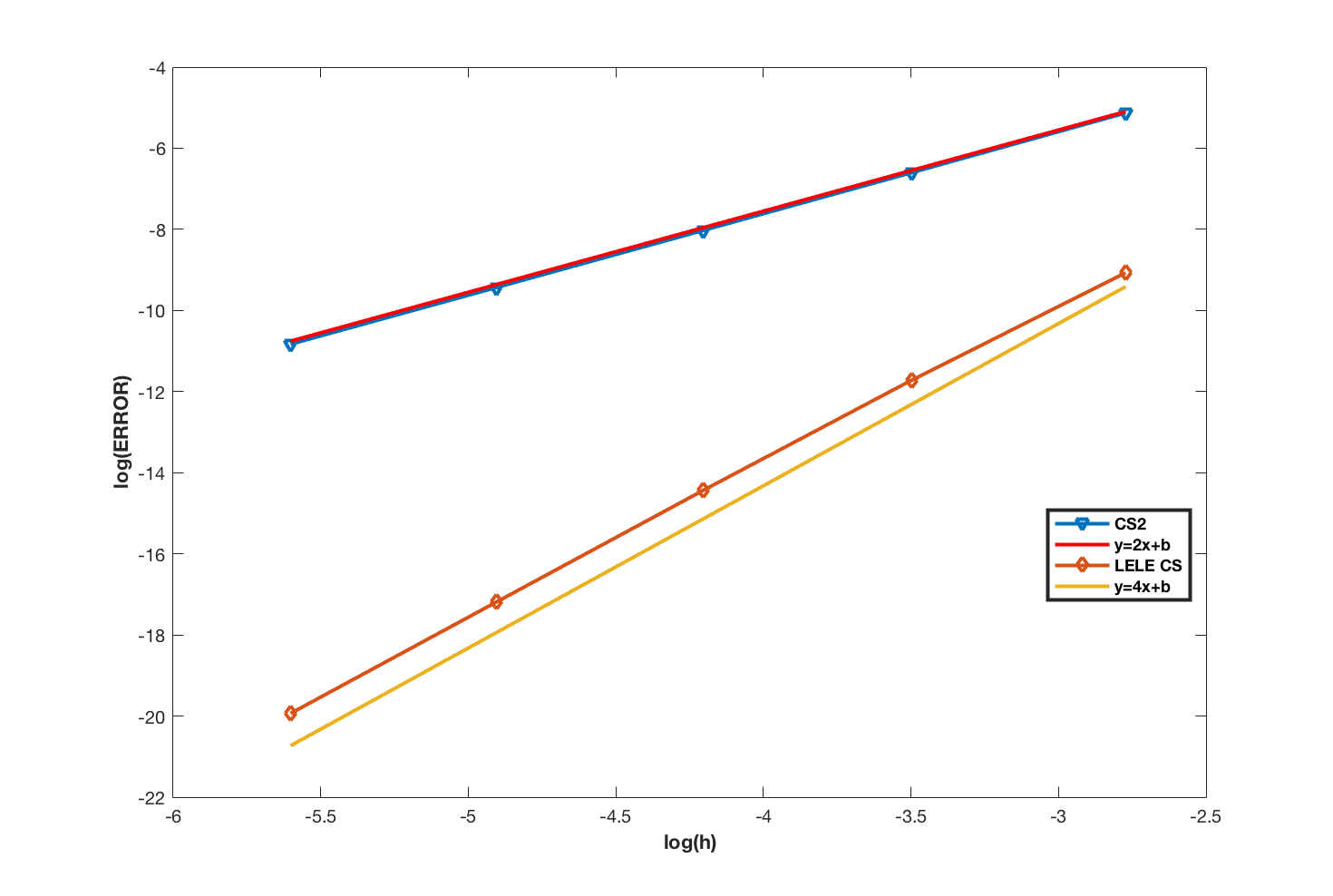}
\end{center}
\caption{Spectrum of the laplacian matrix discretized with second order and fourth order compact schemes (Lele's and CS2) compared to the exact spectrum of the operator (left); 
$L^{\infty}$ error to the function $u(x)=\cos(x(1-x))$ when varying the step-size $h$, for the 3 schemes - comparison with lines of slope $2$ and $4$}
\label{COMP_CS}
\end{figure}

The compact schemes of the second order derivative in space dimension 2 and 3, are simply obtained by using
the previous schemes and to expand them tensorially.\\

The 2D and 3D reaction-diffusion problems we will simulate (Allen-Cahn's or Cahn-Hilliard's equations) are completed with Neumann Boundary Conditions;  IMEX schemes (IMplicit for the linear terms and EXpicit for the nonlinear ones) need to solve the basic linear problem :
\begin{eqnarray}
\alpha u -\Delta u =f & \mbox{ in } \Omega = ]0,1[^{2,3},\\
\Frac{\partial u}{\partial n}=0 & \mbox{ on } \partial \Omega,
\end{eqnarray}
when discretized by second order (RSS) or fourth order compact schemes.  If $n$ is the number of discretization points in each direction of the domain $ \Omega$, the stiffness matrices are then of respective sizes $n^2\times n^2$ (2D problem) and
$n^3\times n^3$ (3D problem). The preconditioning systems can be solved using the cosine FFT.
\subsubsection{Heat Equation}
When $\int_{\Omega}fdx=0$, the solution $u(x,t)$ of the heat equation
\begin{eqnarray}
\Frac{\partial u}{\partial t} -\Delta u =f, x\in \Omega, t>0,\\
\Frac{\partial u}{\partial n}=0 &x\in \partial \Omega,\\
u(x,0)=u_0(x) & x\in \Omega,
\end{eqnarray}
satisfies $\int_{\Omega}u(x,t)dx=\int_{\Omega}u(x,0)dx, \ \forall t >0$. This important property is satisfied at the discrete level when  the stiffness matrix $A$ enjoys of the property
$I_n(Av)=0$, for all $v\in \R^n$, where $I_n$ is a given proper numerical quadrature; for example, when considering the classical second order finite differences  laplacian matrix and $I_n(v)={\bf 1}^Tv/n=Kv$ with
$\{{\bf 1}^T\}=Ker(A)$, the condition on $A$ writes as
$$
I_n(Av)=<Av,{\bf 1}/n>={\bf 1}^TAv/n=\Frac{1}{n}\sum_{i=1}^n\sum_{j=1}^nA_{i,j}v_j=0, \forall v\in \R^n,
$$
say
$$
\sum_{i=1}^nA_{i,j}=0, \,  \forall j=1,\cdots,n.
$$
In this case, the Backward Euler's Scheme 
$$
(Id +\Delta t  A)u^{(k+1)}=u^{(k)} +\Delta t f^{(k)},
$$
allows to reproduce the property act each iteration.
Let $A_L$ (resp. $A_{CS2}$) be the matrix produced by Lele's (resp. by CS2) compact schemes and $B$ be the matrix corresponding to the  classical 3 points schemes. We observe that 
$Ker(B)=Ker(A_{CS2})=\{{\bf 1}^T\}$ while $Ker(A_L)\neq \{{\bf 1}^T\}$ even if $A_L {\bf 1}^T \simeq 1.e-11$. Both compact schemes matrices are not symmetric. However their anti-symmetric part is small as compared to the symmetric one, and this has no incidence in practice when considering Dirichlet Boundary Conditions, se also \cite{BrachetChehabJSC}. When the boundary conditions are periodic or homogeneous Neumann the conservation of the mean value is crucial since the accumulation of errors can deteriorate the solution. The null-mean property is important to recover at the discret level since we will consider splitting schemes in which the linear part consists in solving a time step of a heat equation.
For these reasons, and for a sake of simplicity, we propose to impose the mean null property. This can be done by introducing a Lagrange multiplier or by using a splitting scheme with a projection step on the null-mean space, and give rise to modified RSS schemes. We first briefly describe these two procedures for the general case.
\subsubsection{Classical scheme}
The Backward Euler scheme can be classically interpreted in terms of minimization problem as
$$
u^{(k+1)}= \displaystyle{arg\min_{u\in \R^N}\Frac{1}{2}\|F-Mu\|^2},
$$
where $F=u^{(k)} +\Delta t f^{(k)}$ and $M=Id +\Delta t A$ and $u^{(k+1)}$ satisfies $M^T M u^{(k+1)}=M^TF$ or $Mu^{(k+1)}=F$ since $M$ is invertible. We can impose the null-mean condition in two ways:

First, defining $u^{(k+1)}$ as
$$
u^{(k+1)}= \displaystyle{arg\min_{u\in \R^N I_n(u)=0}\Frac{1}{2}\|F-Mu\|^2},
$$
say
$$
(Id +\Delta A)^T \left((Id +\Delta A)u^{(k+1)}-F\right)+\lambda \frac{1}{n}{\bf 1}^T=0,  \ \  \frac{1}{n}{\bf 1}u=0.
$$

We find formally  $\lambda=\Frac{<g,F>}{<g,g>}$ so $u^{(k+1)}=M^{-1}(F-g)$, where we have set $g=M^{-T} \frac{1}{n}{\bf 1}^T$ .\\

We can also use a projection-like approach which consists in centering the solution at each time step, say
$$
u^{(k+1)}= M^{-1}F -\Frac{<{\bf 1},M^{-1}F>}{<{\bf 1},{\bf 1}>}{\bf 1}.
$$
When $A$ is symmetric, we define $u^{(k+1)}$ as $\displaystyle{arg\min_{u\in \R^N, <{\bf 1},u>=0}\Frac{1}{2}<Mu,u> - <F,u>}$ and we derive similar formulae.
\subsubsection{RSS modified schemes}
We now adapt the previous procedures in the context of the RSS schemes.\\
\\
{\bf Lagrangian approach}\\
The matrix here $B$ is assumed to be Symmetric Semi-Definite Positive.
We can define $\delta^{(k+1)}=u^{(k+1)}-u^{(k)}$ as minimizing
$$
J(\delta)= \Frac{1}{2}<M\delta,\delta>  +\Delta t <-Au^{(k)},\delta>, 
$$
where $M$ is the SPD matrix $Id+\tau \Delta t B$.
Hence, to impose the condition $I_n(\delta)=0$, or equivalently $\delta \perp W=Ker(B)$, we consider rather, e.g., the Lagrangian
$$
{\cal L}(\delta,\lambda)=J(\delta) + \lambda I_n(\delta)=J(\delta) + \lambda <\frac{1}{n}{\bf 1}^T,\delta> .
$$
 The solution of $\inf_{\delta, I_n(\delta)=0}J(\delta)$ is then classically given by the system
 \begin{eqnarray}
 M\delta +\lambda \frac{1}{n}{\bf 1}^T=-\Delta t Au^{(k)},\\
 \frac{1}{n}<{\bf 1}^T,\delta> = 0,
 \end{eqnarray}
 so, formally, letting $f=-\Delta t Au^{(k)}$ and $g= \frac{1}{n}M^{-1}{\bf 1}^T$,
$ \lambda=\frac{<g,f>}{<\frac{1}{n}{\bf 1},g>}$ and
$$
\delta=M^{(-1)}f-\lambda g.
$$
{\bf Projection approach}\\
The second approach consists in  a splitting based on a projection step:
\begin{eqnarray}
\delta^{(k+1)}+\tau \Delta t B\delta^{(k+1)}=-\Delta t A u^{(k)},\\
\delta^{(k+1)}:=\delta^{(k+1)} - \Frac{<{\bf 1}^T\delta^{(k+1)}>}{n}{\bf 1}^T,\\
u^{(k+1)}=u^{(k)}+\delta^{(k+1)}.
 \end{eqnarray}
 This last method is slightly simple than the previous\\
 
 We give hereafter numerical illustration in which we compare the simple RSS and the modified  RSS scheme with splitting in 2D 
 and in 3D; both the Lagrangian and the  projection schemes give similar results. 
 We simulate the exact solution 
 \begin{equation}
u(x,y,t)=\cos(\pi x) \cos (  \pi y ) \exp \left[ \sin (t) \right]
\end{equation}
and represent in Figures \ref{RSS_HEAT_STAB1} (for $N=32$) and \ref{RSS_HEAT_STAB2} (for $N=64$), the time evolution of the mean value and of the error computed in $L^2$ norm.
We observe that imposing the condition $<{\bf 1},u^{(k)}>=0$ allows the modified RSS and classical Backward Euler to obtain a normal level of accuracy in time while the error are strongly cumulated
when this condition is not satisfied. Similar results are found when considering CS2 compact schemes for the spatial discretization, see Figure \ref{RSS_HEAT_STABCS2}. A CPU time reduction si obtained with RSS schemes since the matrix to invert at each iteration is spare and fast solvers are available.
\begin{figure}[h!!] \begin{center}
\includegraphics[height=5.5cm]{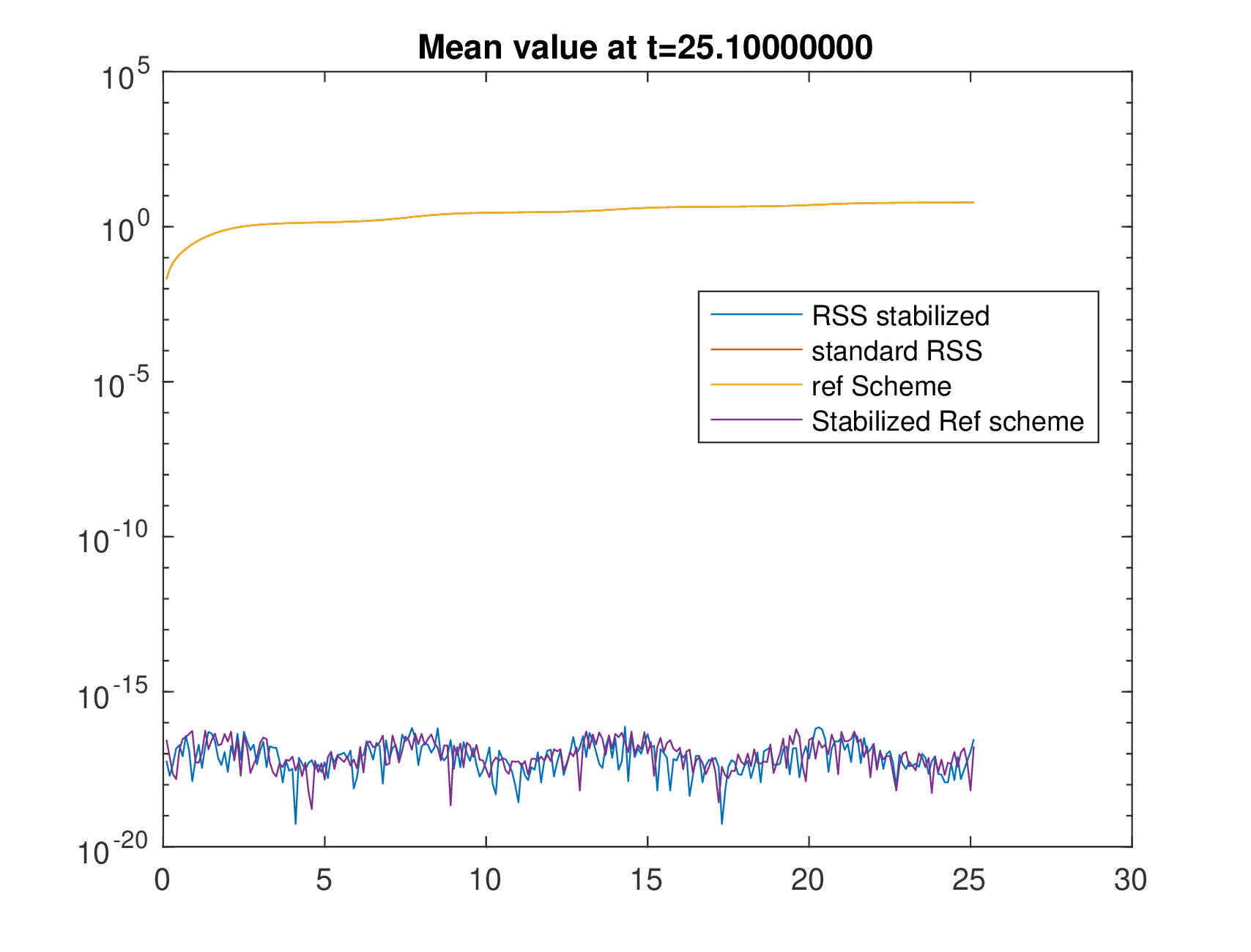}
\includegraphics[height=5.5cm]{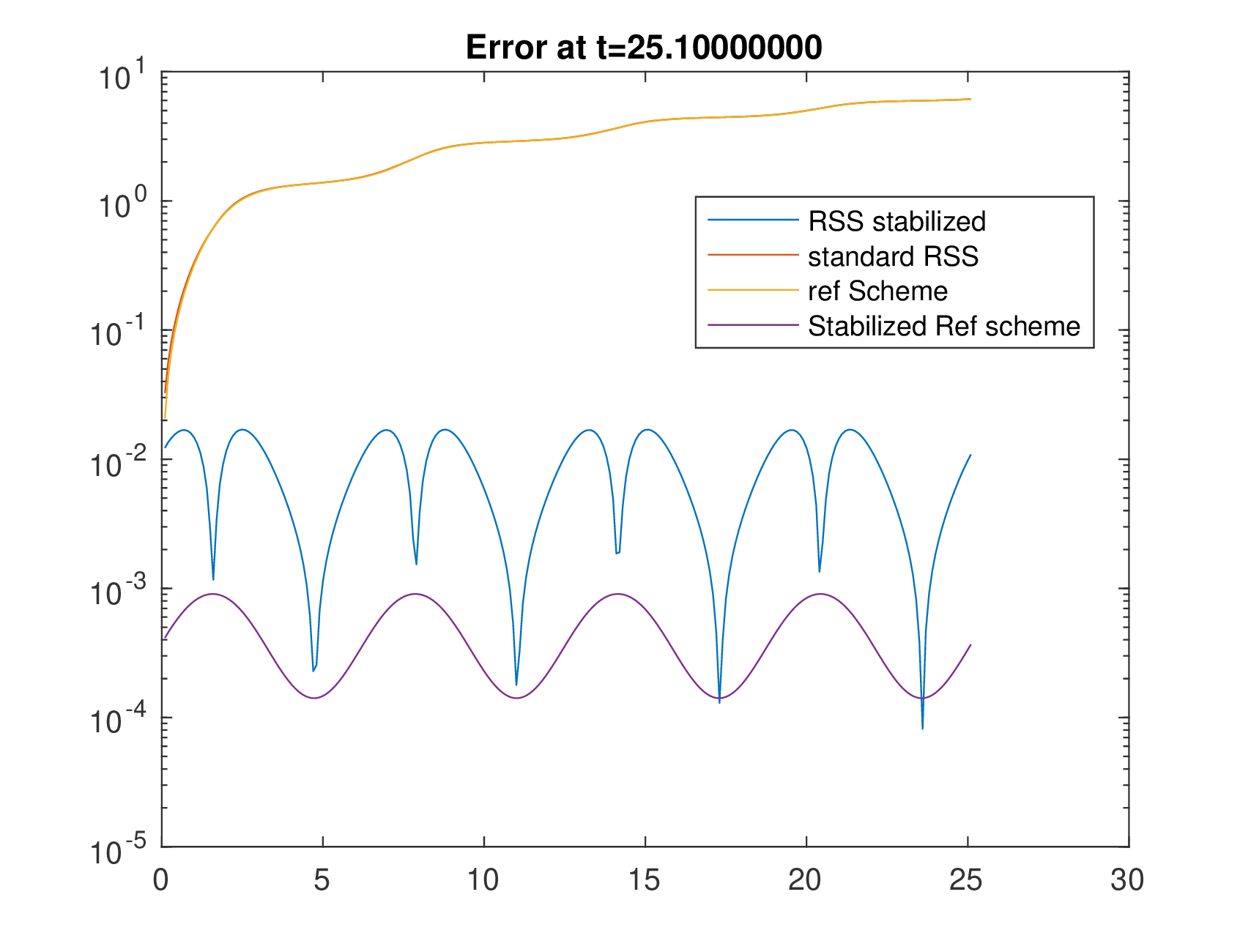}
\end{center}
\caption{2D Heat Equation. $N=32$, $\Delta t = 0.01$, $\tau=2$. Spatial discretization with Lele's Compact Scheme}
\label{RSS_HEAT_STAB1}
\end{figure}
\begin{figure}[h!!]
\begin{center}
\includegraphics[height=5.5cm]{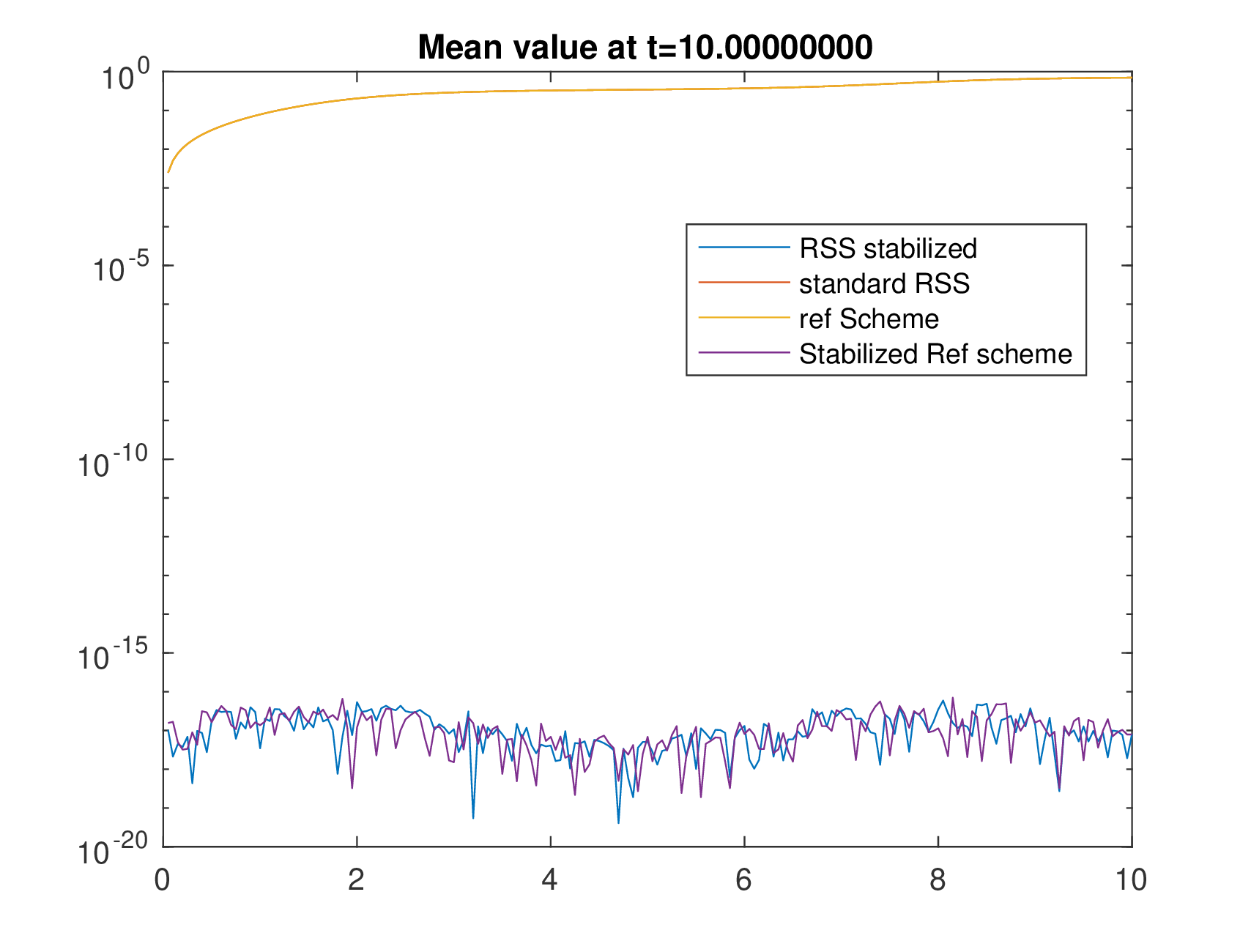}
\includegraphics[height=5.5cm]{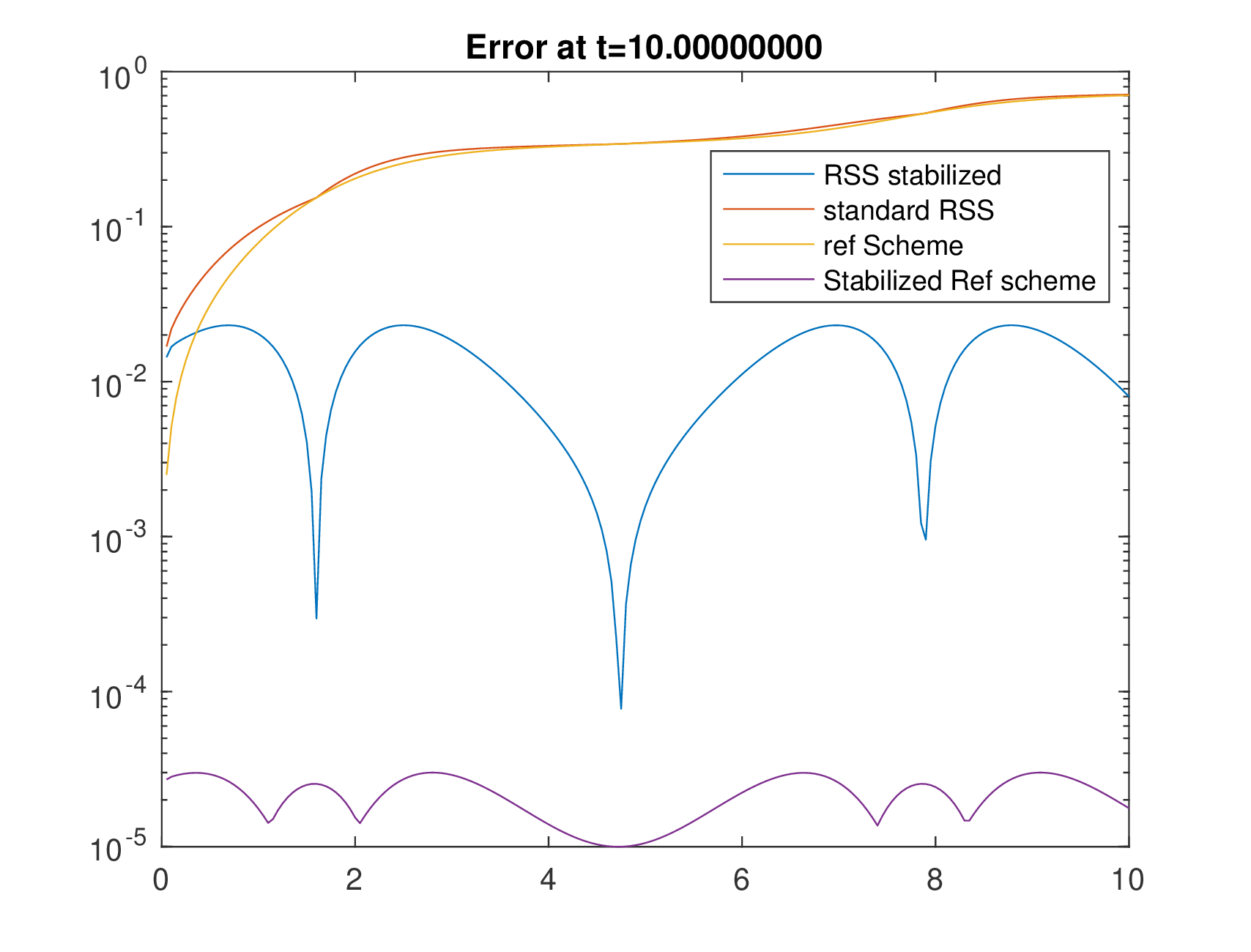}
\end{center}
\caption{2D Heat Equation. $N=64$, $\Delta t = 0.005$, $\tau=4$. Spatial discretization with Lele's Compact Scheme}
\label{RSS_HEAT_STAB2}
\end{figure}
\begin{figure}[h!!]
\begin{center}
\includegraphics[height=5.5cm]{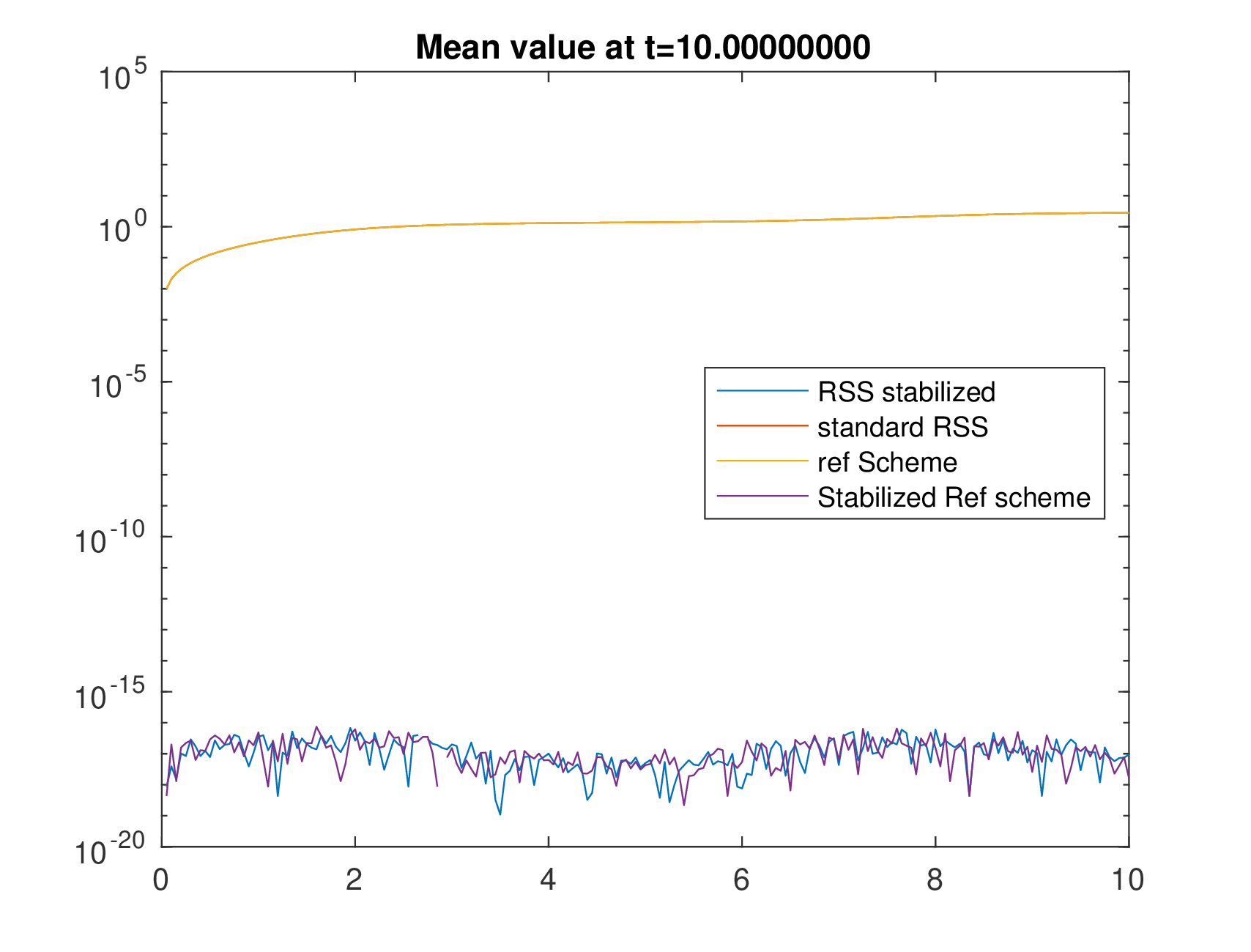}
\includegraphics[height=5.5cm]{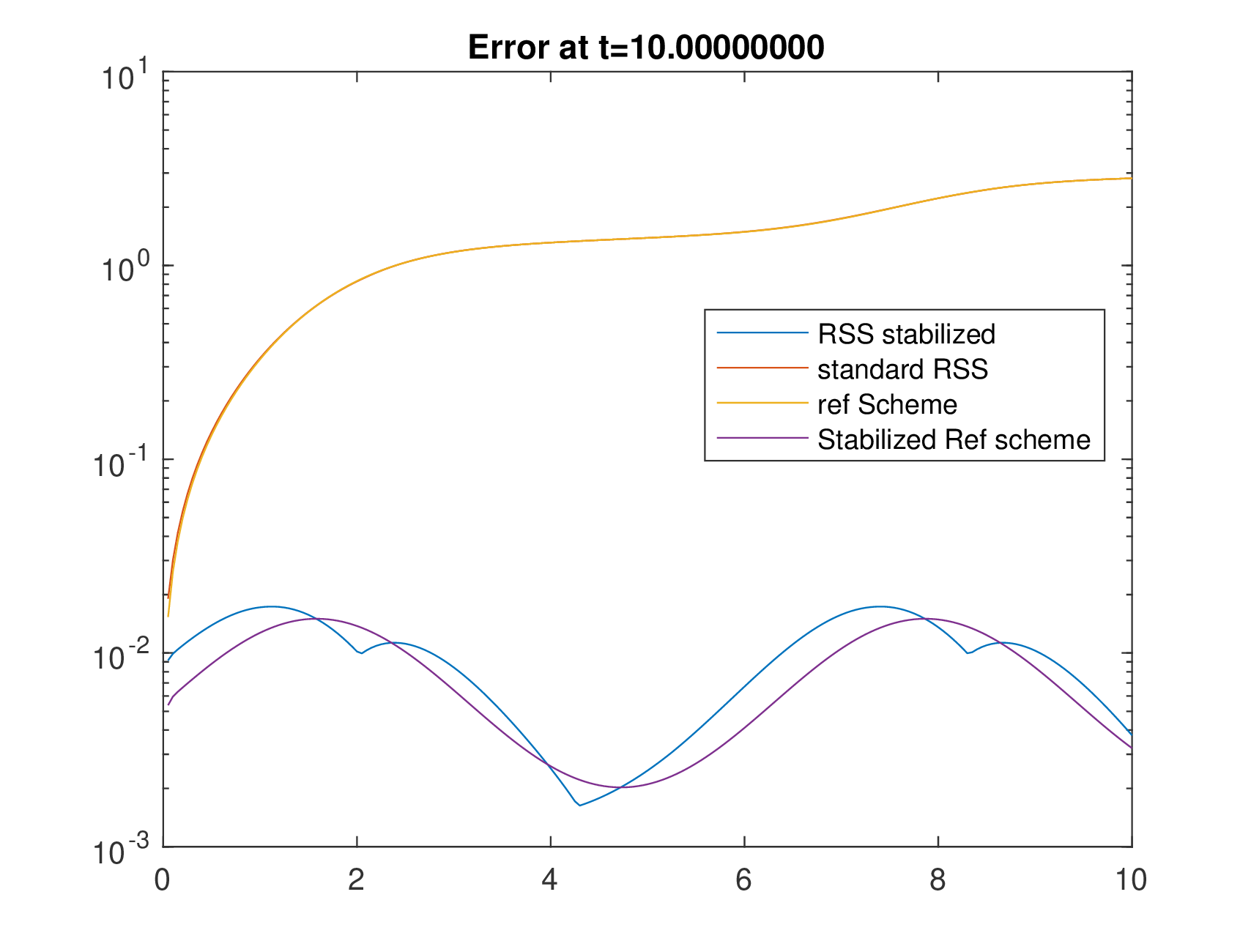}
\end{center}
\caption{2D Heat Equation. $N=32$, $\Delta t = 0.005$, $\tau=2$. Spatial discretization with CS2 Compact Scheme}
\label{RSS_HEAT_STABCS2}
\end{figure}
%
%
\section{Allen-Cahn's equation}
\subsection{Phase transition models}
\subsubsection{Pattern dynamics}
Let $\Omega \subset \R^n, n=2,3$ a regular open bounded set. We here consider the simple  Allen-Cahn equation
\begin{eqnarray}
\Frac{\partial u}{\partial t} +{\cal M}(-\Delta u +\Frac{1}{\epsilon^2}f(u))=0, & x\in \Omega, t >0,\\
\Frac{\partial u}{\partial n}=0 & t>0,\\
u(0,x)=u_0(x), & x\in \Omega.
\end{eqnarray}
which describes the process of phase separation in iron alloys \cite{AllenCahn1,AllenCahn2}, including order-disorder transitions: 
${\cal M}$ is the mobility (taken to be 1 for simplicity),
$F=\displaystyle{\int_{-\infty}^u f(v)dv}$ is the free energy, $u$ is the (non-conserved) order parameter, $\epsilon$ is the interface length. The homogenous Neumann boundary conditions imply that there is not a loss of mass outside the domain $\Omega$. It is important to note that here is a competition between the  potential term and the diffusion term: regularization in phase transition. Two important properties are satisfied by the solution and must be captured by the numerical scheme (intrinsically or numerically):
\begin{itemize}
\item the energy diminishing: Allen-Cahn equation is a gradient flow for the energy 
$$E(u)=\displaystyle{\Frac{1}{2}\int_{\Omega}\|\nabla u\|^2dx +\Frac{1}{\epsilon^2}
F(u)dx},$$
so $E(u(t))\le  E(u(t')), \forall t \ge t'$,
\item the maximum principle: $|u(.,t)|_{L^{\infty}} \le L , \forall t >0$.
\end{itemize}
\subsubsection{Image Segmentation}
 Image Segmentation consists in labelling pixels in digital images in such a way the image becomes easier to analyse, in particular it allows to locate objects and boundaries. There are many different numerical methods to label data such as KNN (K- Nearest Neighbors) algorithms or Orthogonal Neighborhood Preserving Projection (ONPP)
 that uses statistical and linear numerical algebra tools respectively, see \cite{KokiopoulouSaad}.
 Allen-Cahn equations was also proposed as a tool in image segmentation in \cite{LiLee} where the following model is considered:
\begin{eqnarray}
 \Frac{\partial \phi}{\partial t}-\Delta \phi +\Frac{F'(\phi)}{\epsilon^2}+\lambda \left((1+\phi)(f_0-c_1)^2-(1-\Phi)(f_0-c_2)^2\right)
,& x \in \Omega,\\
\Frac{\partial \Phi}{\partial n}=0,& \partial \Omega.
\end{eqnarray}
If $C$ is the segmenting curve, then the phase $\phi$ corresponds to the situations
$$
\phi(x)=\left\{
\begin{array}{ll}
>0 & \mbox{ if $x$ is inside $C$},\\
=0 & \mbox{ if $x \in C$},\\
< 0 & \mbox{ if $x$ is outside $C$}.\\
\end{array}
\right.
$$

Here  $\epsilon >0$, $F'(\phi)=\phi(\phi^2-1)$, $\lambda$ is a nonnegative parameter, $f_0$ is the given image.The terms $c_1$ and $c_2$ are the averages of $f_0$ in the 
regions $(\phi \ge 0)$ and $(\phi < 0)$, say
$$
c_1=\Frac{\int_{\Omega}f_0(x)(1+\phi(x))dx}{\int_{\Omega}(1+\phi(x))dx}
\mbox { and }
c_2=\Frac{\int_{\Omega}f_0(x)(1-\phi(x))dx}{\int_{\Omega}(1-\phi(x))dx}.
$$
\subsection{Energy diminishing schemes}
We set $E(u)=\Frac{1}{2}<Au,u>+\Frac{1}{\epsilon^2}<F(u),{\bf 1}^T>$, where $F$ is a primitive of $f$ that we choose such that $F(0)=0$. We say that the scheme is energy decreasing if
 $$
 E(u^{(k+1)}) < E(u^{(k)}).
 $$

We here first recall somme schemes and their stability conditions, see also \cite{BrachetChehabJSC}.\\

The  semi-implicit Scheme applied to the pattern evolution Allen-Cahn equation:
\begin{eqnarray}
 \Frac{u^{(k+1)}-u^{(k)}}{\Delta t} +Au^{(k+1)}+ \Frac{1}{\epsilon^2}f(u^{(k)}) =0,
 \label{CLASS_AC}
 \end{eqnarray}
 is energy diminishing under the time step restriction
 $$
 0 < \Delta t < \Frac{\epsilon^2 L}{2},
 $$
 where $\mid f'\mid_{\infty}\le L$, see \cite{JShenACCH}.
 
The stabilization of its linear part leads to the scheme
\begin{center}
\begin{minipage}[H]{12cm}
  \begin{algorithm}[H]
    \caption{: RSS-IMEX for Allen Cahn } \label{RSSAC}
    \begin{algorithmic}[1]
     \For{$k=0,1, \cdots$}
          \State {\bf Solve} $ (Id +\tau \Delta t  B)\delta =-\Delta t(A u^{(k)}+\Frac{1}{\epsilon^2}f(u^{(k)}))$
                      \State {\bf Set } $ u^{(k+1)}=u^{(k)}+\delta$
                                  \EndFor
    \end{algorithmic}
    \end{algorithm}
\end{minipage}
\end{center}
 We have the stability result which extend the one given in \cite{BrachetChehabJSC}:
 \begin{theorem}
Assume that hypothesis (\ref{Hyp_H}) holds and that $f$ is ${\cal C}^1$ and $\mid f'\mid_{\infty}\le L$. 
If $W=\{0\}$ we have then the following stability conditions
\begin{itemize}
\item If $\tau\ge \Frac{\beta}{2}$ 
\begin{itemize}
\item if $\left(\Frac{\tau}{\beta}-\Frac{1}{2}\right)\lambda_{min}(A) -\Frac{L}{2\epsilon^2}\ge 0$ then the scheme is unconditionally stable,
\item if $\left(\Frac{\tau}{\beta}-\Frac{1}{2}\right)\lambda_{min}(A) -\Frac{L}{2\epsilon^2}< 0$ then the scheme is stable for
$$
0<\Delta t <\Frac{1}{\Frac{L}{2\epsilon^2} -\left(\Frac{\tau}{\beta}-\Frac{1}{2}\right)\lambda_{min}(A)},
$$
\end{itemize}
\item If $\tau < \Frac{\beta}{2}$ then the scheme is stable for
$$
0<\Delta t <\Frac{1}{\Frac{L}{2\epsilon^2} -\left(\Frac{\tau}{\beta}-\Frac{1}{2}\right)\rho(A)}.
$$
\end{itemize} 
Here $\lambda_{min}(A)$ denotes the lower strictly positive eigenvalue of $A$.\\
If A $W\neq\{0\}$ the stability conditions above apply whenever we assume in addition that  $0<\Delta t< \Frac{2\epsilon^2}{L}$.
\label{Stab_AC}
\end{theorem}
\begin{proof}
For every $k\ge 0$, we decompose every vector $u$ as
$u=u_1+u_2$ with $u_1\in W^T$ and $u_2\in W$.
We have then
$$
\begin{array}{l}
\lambda_{min}(A)\|u_1\|_2^2\le <Au,u>=<Au_1,u_1>\le \rho(A)\|u_1\|_2^2,
\mbox{ and similarly, }\\
 \lambda_{min}(B)\|u_1\|_2^2\le <Bu,u>=<Bu_1,u_1>\le \rho(B)\|u_1\|_2^2.
 \end{array}
 $$
Taking the scalar product of the equation with $u^{(k+1)}-u^{(k)}$, we find after the usual simplifications (parallelogram identity)
$$
\begin{array}{l}
\|u^{(k+1)}-u^{(k)}\|^2+\tau \Delta t <B(u^{(k+1)}-u^{(k)},u^{(k+1)}-u^{(k)}> - \Frac{\Delta t}{2} <A(u^{(k+1)}-u^{(k)}),u^{(k+1)}-u^{(k)}>\\
+\Frac{\Delta t}{2}\left(<Au^{(k+1)},u^{(k+1)}> - <Au^{(k)},u^{(k)}>\right)
+\Frac{\Delta t}{\epsilon^2}<f(u^{(k)},u^{(k+1)}-u^{(k)}>=0.
\end{array}
$$
Following  \cite{JShenACCH}, we write 
$$
<F(u^{(k+1)})-F(u^{(k)}),{\bf 1}>=<f(u^{(k)}),u^{(k+1)}-u^{(k)}> +\Frac{1}{2}<f'(\xi_k)(u^{(k+1)}-u^{(k)}),u^{(k+1)}-u^{(k)}>.
$$
Therefore, using hypothesis \ref{Hyp_H},
$$
\begin{array}{ll}
\|u^{(k+1)}-u^{(k)}\|^2+ \Delta t \left( \Frac{\tau}{\beta}-\Frac{1}{2}\right)<A(u^{(k+1)}-u^{(k)}),(u^{(k+1)}-u^{(k)})> &\\
&\le \Frac{L \Delta t}{2\epsilon^2}\|u^{(k+1)}-u^{(k)}\|^2,\\
+\Frac{\Delta t}{2}\left(<Au^{(k+1)},u^{(k+1)}> - <Au^{(k)},u^{(k)}>\right)
+\Frac{\Delta t}{\epsilon^2}<F(u^{(k+1)})-F(u^{(k)}),{\bf 1}>. &
\end{array}
$$
Finally,
$$
(1-\Frac{L\Delta t}{2\epsilon^2}) \|u^{(k+1)}-u^{(k)}\|^2+ \Delta t \left( \Frac{\tau}{\beta}-\Frac{1}{2}\right)<A(u^{(k+1)}-u^{(k)}),u^{(k+1)}-u^{(k)}> 
+\Delta t\left( E(u^{(k+1)})-E(u^{(k)})\right)\le 0,
$$
If $W=\{0\}$, then $\lambda_{min}(A) \|u^{(k+1)}-u^{(k)}\|^2\le <A(u^{(k+1)}-u^{(k)}),u^{(k+1)}-u^{(k)}> \le \rho(A)\|u^{(k+1)}-u^{(k)}\|^2$ so
$$
(1-\Frac{L\Delta t}{2\epsilon^2} + \Delta t \left( \Frac{\tau}{\beta}-\Frac{1}{2}\right)\lambda_{min}(A) ) \|u^{(k+1)}-u^{(k)}\|^2 +\Delta t\left( E(u^{(k+1)})-E(u^{(k)})\right)\le 0,
$$
if $\Frac{\tau}{\beta}-\Frac{1}{2}\ge 0$, and
$$
(1-\Frac{L\Delta t}{2\epsilon^2} + \Delta t \left( \Frac{\tau}{\beta}-\Frac{1}{2}\right) \rho(A)) \|u^{(k+1)}-u^{(k)}\|^2 +\Delta t\left( E(u^{(k+1)})-E(u^{(k)})\right)\le 0,
$$
if $\Frac{\tau}{\beta}-\Frac{1}{2}\le 0$.\\

When $W\neq\{0\}$, using the relation $\|u^{(k+1)}-u^{(k)}\|^2=\|u_1^{(k+1)}-u_1^{(k)}\|^2+\|u_2^{(k+1)}-u_2^{(k)}\|^2$, we can write
$$
\begin{array}{ll}
(1-\Frac{L\Delta t}{2\epsilon^2}) \|u_2^{(k+1)}-u_2^{(k)}\|^2&\\
+ (1-\Frac{L\Delta t}{2\epsilon^2}) \|u_1^{(k+1)}-u_1^{(k)}\|^2+\Delta t \left( \Frac{\tau}{\beta}-\Frac{1}{2}\right)<A(u_1^{(k+1)}-u_1^{(k)}),u_1^{(k+1)}-u_1^{(k)}> &\le 0,\\
+\Delta t\left( E(u^{(k+1)})-E(u^{(k)})\right)&
\end{array}
$$
hence the result.
\end{proof}
\begin{remark}
We recover the result given in \cite{BrachetChehabJSC} when $W=\{0\}$ and get enhanced stability as compared to classical IMEX scheme. When $W\neq \{0\}$ the stability conditions are comparable and the advantage of RSS-IMEX is to solve simplified linear part that can be solved fastly.
\end{remark}
A more stable  way to overcome the stability restriction is to consider
directly Allen-Cahn equation as a gradient system with a natural diminishing energy property.
A first unconditionally stable scheme is (\cite{Elliott,ElliottStuart})
\begin{eqnarray}
 \Frac{u^{(k+1)}-u^{(k)}}{\Delta t} +Au^{(k+1)}+\Frac{1}{\epsilon^2}DF(u^{(k)},u^{(k+1)})=0,
 \label{ISAC1}
 \end{eqnarray}
where
$$
DF(u,v)=\left\{
\begin{array}{ll}
\Frac{F(u)-F(v)}{u-v} & \mbox{ if } u\neq v ,\\
f(u) & \mbox{ if } u= v.\\
\end{array}
\right.
$$ 
In \cite{BrachetChehabJSC} it was introduced the RSS-scheme
\begin{eqnarray}\label{RSSNLG}
\Frac{u^{(k+1)}-u^{(k)}}{\Delta t} +\tau B(u^{(k+1)}-u^{(k)})+DF(u^{(k+1)},u^{(k)})=
-Au^{(k)},
\end{eqnarray}
which enjoys of the following stability condition, see \cite{BrachetChehabJSC} for a similar proof.
\begin{proposition}
Under hypothesis ${\cal H}$
\begin{itemize}
\item if $\tau \ge \Frac{\beta}{2}$, the nonlinear RSS scheme (\ref{RSSNLG}) is unconditionally stable,
\item if $\tau < \Frac{\beta}{2}$, the nonlinear RSS scheme (\ref{RSSNLG}) is stable under condition
$$
0<\Delta t < \Frac{\beta}{\rho(A)(\frac{\beta}{2}-\tau)}.
$$
\end{itemize}
\end{proposition}
Finally,  unconditionally stable scheme is to use the so-called convex splitting, \cite {Eyre,Diegel}.
These schemes are based on a proper splitting of the free energy term
\begin{itemize}
\item[i.] 
$$
F(u)=F_c(u)-F_e(u),
$$
where $F_*\in {\cal C}^2(\R^n,\R)$, $*=c$ or $*=e$.
\item[ii.] 
$F_*$ is strictly convex in $\R^n$, $*=c$ or $*=e$.
\item[iii.] $<[\nabla F_e(u)]u,u>\ge -\lambda$, $\forall u \in \R^n.$
\end{itemize}
The scheme reads as
\begin{eqnarray}\label{CLASS_CONV_SSSPLIT}
\Frac{u^{(k+1)}-u^{(k)}}{\Delta t} +Au^{(k+1)} +\nabla F_c(u^{(k+1)})=+\nabla F_e(u^{(k)}),
\end{eqnarray}
Its sabilization reads as
\begin{eqnarray}\label{NLGRSSSPLIT}
\Frac{u^{(k+1)}-u^{(k)}}{\Delta t} +\tau B(u^{(k+1)}-u^{(k)})+\nabla F_c(u^{(k+1)})=-Au^{(k)}+\nabla F_e(u^{(k)}),
\end{eqnarray}
and we can prove the following result (see  \cite{BrachetChehabJSC}):
\begin{theorem}
Assume that hypothesis (\ref{Hyp_H}) holds.
We have the following sufficient stability conditions:
\begin{itemize}
\item If $(\tau \beta -\Frac{1}{2}\rho(A) +\left( {\hat \lambda}- |\lambda| \right) >0$ then the scheme is unconditionally stable 
\item Else it is stable under condition
$$
0<\Delta t<\Frac{1}{(\frac{1}{2}-\tau \beta)\rho(A) +|\lambda| -{\hat \lambda}}.
$$
\end{itemize}
\end{theorem}
Here ${\hat \lambda}$ is the upper bound of the interval $I$ where $<\nabla F_e(u)-\nabla F_e(v),u-v>\ge c \|u-v\|^2$ holds $\forall c \in I$, see \cite{BrachetChehabJSC}.
\subsection{Splitting schemes}
We follow \cite{LiLee} who proposed for the so-called double well potential case  ($F(u)=\frac{1}{4}(1-u^2)^2$) the following splitting scheme:
\begin{eqnarray}
\Frac{u^{*}-u^{(k)}}{\Delta t}+Au^*=0,\\
\Frac{u^{(k+1)}-u^*}{\Delta t}=\Frac{u^{(k+1)}-(u^{(k+1)})^3}{\epsilon^2}.
\end{eqnarray}
The last equation can be simplified since it correspond to a one-step approximation by backward Euler's 
to the differential equation
\begin{eqnarray}
\Frac{du}{dt}=\Frac{u-u^3}{\epsilon^3},
\end{eqnarray}
whose the solution is
$$
u(t)=\Frac{u(0)}{\sqrt{e^{-2\frac{t}{\epsilon^2}} +u(0)^2(1-e^{-2\frac{ t}{\epsilon^2}}) }}.
$$
Hence the simplified scheme is obtained
\begin{eqnarray}
\Frac{u^{*}-u^{(k)}}{\Delta t}+Au^*=0,\\
u^{(k+1)}=\Frac{u^*}{\sqrt{e^{-2\frac{\Delta t}{\epsilon^2}} +(u^*)^2(1-e^{-2\frac{\Delta t}{\epsilon^2}}) }}.
\end{eqnarray}
We now give here a simple stability result:
\begin{theorem}\label{th_split}
Assume that $Ker(A)=\{{\bf 1}^T\}$ and that $Id+\Delta tA$ enjoys of the discrete maximum principle.
Assume that $|u^{(0)}_i|\le 1, i=1,\cdots N$. Then the sequence $u^{(k)}$ defined by
\begin{eqnarray}
\Frac{u^{*}-u^{(k)}}{\Delta t}+Au^*=0,\\
u^{(k+1)}=\Frac{u^*}{\sqrt{e^{-2\frac{\Delta t}{\epsilon^2}} +(u^*)^2(1-e^{-2\frac{\Delta t}{\epsilon^2}}) }},
\end{eqnarray}
satisfies $|u^{(k)}_i|\le 1, i=1,\cdots N$.
\end{theorem}
\begin{proof}
We proceed by induction. First of all we show that if $|u^{(k)}|\le 1$ then $|u^{(*)}|\le 1$. 
We have
$$
(Id +\Delta tA) (u^*-{\bf 1})=(u^{(k)}-{\bf 1}).
$$
Hence, by the maximum principle, if $u^{(k)}-{\bf 1}^T \le 0$ then $u^{(*)}-{\bf 1}^T \le 0$. Replacing
$(u^{(k)}-{\bf 1})$ (resp. $(u^*-{\bf 1})$) by $(u^{(k)}+{\bf 1}^T)$ (resp. $(u^*+{\bf 1}^T)$) we find that
$u^{(*)}+{\bf 1}^T \ge 0$ and conclude that $-1\le u^*\le 1$.\\
Now, to conclude, it suffices to  show that 
$$
|\Frac{x}{\sqrt{e^{-2\frac{\Delta t}{\epsilon^2}} +(x)^2(1-e^{-2\frac{\Delta t}{\epsilon^2}}) }}| \le 1
\ \forall x \in [-1,1], \ \forall \Delta t >0,  \ \forall \epsilon^2>0.
$$
We set for convenience $\gamma=e^{-2\frac{\Delta t}{\epsilon^2}} \in [0,1]$. We start from
$$
x^2 \le 1
\iff
\gamma x^2 \le \gamma
\iff
x^2\le (1-\gamma)x^2+\gamma
\iff
\Frac{x^2}{(1-\gamma)x^2+\gamma}\le 1.
$$
The result is obtained by taking the square-root of this last expression.
\end{proof}\\
At this point, we can define a stabilized version of this splitting scheme as
\begin{eqnarray}\label{RSS_split1}
\Frac{u^{*}-u^{(k)}}{\Delta t}+\tau B(u^*-u^{(k)})=-Au^{(k)},\\ \label{RSS_split2}
u^{(k+1)}=\Frac{u^*}{\sqrt{e^{-2\frac{\Delta t}{\epsilon^2}} +(u^*)^2(1-e^{-2\frac{\Delta t}{\epsilon^2}}) }}.
\end{eqnarray}
To implement RSS-like version of this splitting scheme it then suffices to replace the first step by a RSS-CN scheme as proposed in section 2. We then obtain the RSS-splitting scheme
\begin{center}
\begin{minipage}[H]{12cm}
  \begin{algorithm}[H]
    \caption{: RSS-splitting for Allen Cahn }\label{splitt_AC}
    \begin{algorithmic}[1]
     \For{$k=0,1, \cdots$}
          \State {\bf Solve} $ (Id +\tau \Delta t  B)\delta =-\Delta tA u^{(k)}$
                      \State {\bf Set } $ u^{(*)}=u^{(k)}+\delta$
                       \State {\bf Set }
                       $u^{(k+1)}=\Frac{u^*}{\sqrt{e^{-2\frac{\Delta t}{\epsilon^2}} +(u^*)^2(1-e^{-2\frac{\Delta t}{\epsilon^2}}) }}
$                      
            \EndFor
    \end{algorithmic}
    \end{algorithm}
\end{minipage}
\end{center}
The proof for the $L^{\infty}$-stability of the classical $\theta$-scheme with a second order FD matrix
is given in (\cite{MortonMayers}), page 33, but is based on a pointwise analysis.
\begin{theorem}
We assume that the assumptions of Proposition \ref{Heat_Max_Princip} on $A$ and $B$ hold and, in addition, that
$Ker(A)=Ker(B)=\{{\bf 1}\}$ and that $|u^{(0)}_i|\le 1, i=1,\cdots N$. Then the sequence $u^{(k)}$ defined by
(\ref{RSS_split1})-(\ref{RSS_split2})
satisfies $|u^{(k)}_i|\le 1, i=1,\cdots N$
\end{theorem}
\begin{proof}
According to the proof of Theorem \ref{th_split}, it suffices to show that if $|u^{(k)}|\le 1$ then $|u^{(*)}|\le 1$. This is automatically provided using Proposition  \ref{Heat_Max_Princip}.
\end{proof}
\begin{remark}
We gave here sufficient conditions to ensure the stabilized  scheme to satisfy a discrete maximum principle. However, this doesn't implies that the scheme is energy decreasing for any $\Delta t >0$: this is observed for small values of $\Delta t$ and moderate values of $\tau$ ; at the contrary large values of $\tau$ allows to take large time steps but the energy becomes no longer decreasing.
\end{remark}
\subsection{Projection scheme}
As pointed out in the previous section, we do not have generally $A{\bf 1}^T=0$ so the mean of  solution of the linear step is not conserved. To overcome this problem we project it on the mean-mull vector space
and obtain the scheme
\begin{center}
\begin{minipage}[H]{12cm}
  \begin{algorithm}[H]
    \caption{: RSS-splitting for Allen Cahn  with projection}\label{splitt_AC}
    \begin{algorithmic}[1]
     \For{$k=0,1, \cdots$}
          \State {\bf Solve} $ (Id +\tau \Delta t  B)\delta =-\Delta tA u^{(k)}$
           \State {\bf Set } $\delta = \delta -\Frac{<{\bf 1}^T,\delta>}{n}{\bf 1}^T$
                      \State {\bf Set } $ u^{(*)}=u^{(k)}+\delta$                      
                       \State {\bf Set }
                       $u^{(k+1)}=\Frac{u^*}{\sqrt{e^{-2\frac{\Delta t}{\epsilon^2}} +(u^*)^2(1-e^{-2\frac{\Delta t}{\epsilon^2}}) }}
$                      
            \EndFor
    \end{algorithmic}
    \end{algorithm}
\end{minipage}
\end{center}
%
%
\section{Cahn-Hilliard equation}
We here present briefly Cahn-Hilliard equations used for Phase transition and for image inpainting. We introduce new stabilized schemes and establish stability properties.
\subsection{The models}
\subsubsection{Cahn-Hilliard and Patterns dynamics}
The Cahn-Hilliard  equation describes the process of phase separation, by which the two components of a binary fluid spontaneously separate and form domains pure in each component. It writes as
\begin{eqnarray}
\Frac{\partial u}{\partial t} -\Delta( -\Delta u +\Frac{1}{\epsilon^2}f(u))=0,\\
\Frac{\partial u}{\partial n}=0,\\
\Frac{\partial }{\partial n}\left(\Delta u-\Frac{1}{\epsilon^2}f(u)\right)=0,\\
u(0,x)=u_0(x).
\end{eqnarray}
This equation enjoys of the following  properties
\begin{itemize}
\item Conservation of the mass: ${\bar u}=\displaystyle{\int_{\Omega}u(x,t)dx}=\displaystyle{\int_{\Omega}u_0(x)dx}$,
\item Decay of the energy in time
$$
\Frac{\partial E(u)}{\partial t}=
-\displaystyle{\int_{\Omega}|\nabla( -\Delta u +\Frac{1}{\epsilon^2}f(u))|^2dx}\le 0.
$$
\end{itemize}
A classical way to study and to simulate Cahn-Hilliard model is to decouple the equation as follows:
\begin{eqnarray}
\Frac{\partial u}{\partial t} -\Delta\mu=0, & \mbox{ in }\Omega , t >0,\\
\mu= -\Delta u +\Frac{1}{\epsilon^2}f(u),&\mbox{ in }\Omega , t>0,\\
\Frac{\partial u}{\partial n}=0,\Frac{\partial \mu}{\partial n}=0,&\mbox{ on }\partial \Omega , t>0,\\
u(0,x)=u_0(x)& \mbox{ in } \Omega.
\end{eqnarray}
\subsubsection{The inpainting problem}\label{Impainting:model}
Cahn-Hilliard equations allow here to in paint a tagged picture.
Let $g$ be the original image and $D\subset \Omega$  the region of $\Omega$ in which the image is deterred. The idea is to add a penalty term that forces the image to remain unchanged in $\Omega\setminus D$ and to reconnect the fields of $g$ inside $D$, see e.g. \cite{Bertozzi1,Bertozzi2}. Let $\lambda >>1$. We have
\begin{eqnarray}
\Frac{\partial u}{\partial t} -\Delta( -\epsilon \Delta u +\Frac{1}{\epsilon}f(u))&+\lambda\chi_{\Omega\setminus D}(x)(u-g)=0,\\
\underbrace{\mbox{Cahn-Hilliard equation} }&\underbrace{\mbox{Fidelity term} }\\
\Frac{\partial u}{\partial n}=0&\Frac{\partial }{\partial n}\left(\Delta u-\Frac{1}{\epsilon^2}f(u)\right)=0,\\
u(0,x)=u_0(x),& x \in \Omega.
\end{eqnarray}
Here $\chi_{\Omega\setminus D}(x)=\left\{\begin{array}{ll}1 &\mbox{ if } x\in \Omega\setminus D,\\ 0 &\mbox{ else.}\end{array}\right.$\\
 The presence of the penalization  term $\lambda\chi_{\Omega\setminus D}(x)(u-g)$ forces the solution to be close to $g$ in $\Omega\setminus D$ when $\lambda>>1$;
the Cahn-Hilliard flow has as effect to connect the fields inside $D$.
Here $\epsilon$ will play the role of the "contrast". A post-processing is possible using the thresholding procedure consisting in replacing the dominant phase  by 1 at every point
of $\Omega$ and the other phases (colors) by 0 to obtain the final inpainting result with a sharp contrast, see also \cite{Fakih,CherfilsFakihMiranville1} and the references therein.
\subsection{The Stabilized Scheme}
The semi-implicit scheme
\begin{eqnarray}\label{CHC1}
\Frac{u^{(k+1)}-u^{(k)}}{\Delta t} +A\mu^{(k+1)}=0,\\
\label{CHC2}
\mu^{(k+1)}=\epsilon Au^{(k+1)}+\Frac{1}{\epsilon}f(u^{(k)}),
\end{eqnarray}
suffers from a hard time step restriction, its energy stability is guaranteed for
$$
0<\Delta t< C \epsilon^2,
$$
where $C$ is a constant depending on $f$, see \cite{JShenACCH}; the stability condition becomes $0<\Delta t< C \epsilon^4$ when the second equation is
$\mu^{(k+1)}= Au^{(k+1)}+\Frac{1}{\epsilon^2}f(u^{(k)})$ which do not change the extrema of the energy.\\
We derive the Stabilized-Scheme from the backward Euler's (\ref{CHC1})-(\ref{CHC2})  by
replacing $Az^{(k+1)}$ by $\tau B(z^{(k+1)}-z^{(k)})+Az^{(k)}$ for $z=u$ or $z=\mu$. We obtain
\begin{eqnarray}
\label{RSS_CH1}
\Frac{u^{(k+1)}-u^{(k)}}{\Delta t}+\tau B(\mu^{(k+1)}-\mu^{(k)}) +A\mu^{(k)}=0,\\
\label{RSS_CH2}
\mu^{(k+1)}=\epsilon \tau B (u^{(k+1)}-u^{(k)})+\epsilon Au^{(k)}+\Frac{1}{\epsilon} f(u^{(k)}).
\end{eqnarray}
We remark that this scheme preserves the steady state. We now address a stability analysis.
\begin{theorem}\label{theo_CH1}
Assume that $A$ and $B$ satisfy the hypothesis (\ref{Hyp_H}), we note $W=Ker(A)=Ker(B)$.
We have the following stability conditions in the linear and in the nonlinear case:
\begin{itemize}
\item Linear case $f \equiv 0$:
If $\tau > \beta $, then the scheme (\ref{RSS_CH1})-(\ref{RSS_CH2}) 
is unconditionally stable.
\item Nonlinear case:
If $\tau \ge max(\beta,\Frac{L}{2\epsilon^2 \lambda_{min}(B)}+\Frac{\beta}{2})$,
then the scheme (\ref{RSS_CH1})-(\ref{RSS_CH2}) 
is unconditionally stable. Here $\lambda_{min}(B)>0$ is the smallest strictly positive eigenvalue of $B$, i.e.\\ $\lambda_{min}(B)=\Min_{x\in W^T, \|x\|=1}<Bx,x>$.
\end{itemize}
In addition $u^{(k+1)}-u^{(k)}\in W^\perp$, $\forall k\ge 0$, in particular, if $W=\{{\bf 1}^T\}$, the mean value of $u^{(k)}$ is conserved.
\end{theorem}
\begin{proof}
We first prove directly that $u^{(k+1)}-u^{(k)}\in W^\perp, \forall k\ge 0$ by taking the scalar product with any element of $W=Ker(A)=Ker(B)$ in the first system.
We begin by considering the linear case $(f \equiv 0)$.\\
We take the scalar product of (\ref{RSS_CH1}) with $\mu^{(k+1)}$ and of
(\ref{RSS_CH2}) with $u^{(k+1)}-u^{(k)}$. After the use of the parallelogram identity and usual simplifications, we obtain, on the one hand
$$
\begin{array}{l}
<u^{(k+1)}-u^{(k)},\mu^{(k+1)}>+\Frac{\Delta t \tau}{2}\left(<B\mu^{(k+1)},\mu^{(k+1)}>
-<B\mu^{(k)},\mu^{(k)}>\right.\\
+\left.<B(\mu^{(k+1)}-\mu^{(k)}),\mu^{(k+1)}-\mu^{(k)}> \right)
+\Frac{\Delta t}{2}\left(<A\mu^{(k+1)},\mu^{(k+1)}>
+<A\mu^{(k)},\mu^{(k)}>\right.\\
-\left.<A(\mu^{(k+1)}-\mu^{(k)}),\mu^{(k+1)}-\mu^{(k)}> \right)=0,
\end{array}
$$
and on the other hand
$$
\begin{array}{ll}
<u^{(k+1)}-u^{(k)},\mu^{(k+1)}>&=\tau \epsilon <B(u^{(k+1)}-u^{(k)}),u^{(k+1)}-u^{(k)}>\\
&+\Frac{1}{2}\epsilon \left(<Au^{(k+1)},u^{(k+1)}>
-<Au^{(k)},u^{(k)}>\right)\\
&-\Frac{1}{2}\epsilon<A(u^{(k+1)}-u^{(k)}),u^{(k+1)}-u^{(k)}>.
\end{array}
$$
Taking the difference of the last two identities, we obtain 
$$
\begin{array}{l}
\epsilon \{ \tau<B(u^{(k+1)}-u^{(k)}),u^{(k+1)}-u^{(k)}>-\Frac{1}{2}<A(u^{(k+1)}-u^{(k)}),u^{(k+1)}-u^{(k)}>\}\\
+\Delta t \{\Frac{ \tau}{2}<B(\mu^{(k+1)}-\mu^{(k)}),\mu^{(k+1)}-\mu^{(k)}>-\Frac{1}{2}<A(\mu^{(k+1)}-\mu^{(k)}),\mu^{(k+1)}-\mu^{(k)}>\}\\
+\Frac{\Delta t}{2}\left(\tau <B\mu^{(k)},\mu^{(k)}> - <A\mu^{(k)},\mu^{(k)}>\right)\\
+\Frac{\Delta t}{2}\left(<A\mu^{(k+1)},\mu^{(k+1)}>+\tau <B\mu^{(k+1)},\mu^{(k+1)}>\right)
+R^{k+1}-R^{k}=0,
\end{array}
$$
where
$$
R^{k+1}=\Frac{1}{2}\epsilon <Au^{(k+1)},u^{(k+1)}>.
$$
The scheme is then stable when $R^{k+1}<R^k$.
Now using (\ref{Hyp_H}), we obtain
$$
\begin{array}{l}
\epsilon \{ \tau<B(u^{(k+1)}-u^{(k)}),u^{(k+1)}-u^{(k)}>-\Frac{1}{2}<A(u^{(k+1)}-u^{(k)}),u^{(k+1)}-u^{(k)}>\}\\
+\Delta t \{\Frac{ \tau}{2}<B(\mu^{(k+1)}-\mu^{(k)}),\mu^{(k+1)}-\mu^{(k)}>-\Frac{1}{2}<A(\mu^{(k+1)}-\mu^{(k)}),\mu^{(k+1)}-\mu^{(k)}>\}\\
+\Frac{\Delta t}{2}\left(\tau <B\mu^{(k)},\mu^{(k)}> - <A\mu^{(k)},\mu^{(k)}>\right)\\
\ge\\
\epsilon (\tau-\Frac{\beta}{2})B(u^{(k+1)}-u^{(k)}),u^{(k+1)}-u^{(k)}>\\
+\Frac{\Delta t}{2}(\tau-\beta)<B(\mu^{(k+1)}-\mu^{(k)}),\mu^{(k+1)}-\mu^{(k)}>\\
+\Frac{\Delta t}{2}(\tau-\beta)<B\mu^{(k)},\mu^{(k)}>.
\end{array}
$$
Hence the sufficient stability conditions.\\
Consider now the general case $f\neq 0$.  First of all, as in \cite{JShenACCH}, we take the Taylor expansion of the term
$\frac{1}{\epsilon}f(u^k)=\Frac{1}{\epsilon}<F(u^{(k+1)})-F(u^{(k)}),{\bf 1}^T>+ \Frac{1}{2\epsilon}<f'(\xi^{(k)})(u^{(k+1)}-u^{(k)}),
u^{(k+1)}-u^{(k)}>.$
 Hence, we deduce from the previous inequality
$$
\begin{array}{ll}
\epsilon (\tau-\Frac{\beta}{2})<B(u^{(k+1)}-u^{(k)}),u^{(k+1)}-u^{(k)}>&\\
+\Frac{\Delta t}{2}(\tau-\beta)<B(\mu^{(k+1)}-\mu^{(k)}),\mu^{(k+1)}-\mu^{(k)}>&\\
+\Frac{\Delta t}{2}(\tau-\beta)<B\mu^{(k)},\mu^{(k)}>& \le \Frac{L}{2\epsilon}\|u^{(k+1)}-u^{(k)}\|^2,\\
+\Frac{\Delta t}{2}\left(<A\mu^{(k+1)},\mu^{(k+1)}>+\tau <B\mu^{(k+1)},\mu^{(k+1)}>\right)
+E^{k+1}-E^{k}&
\end{array}
$$
where $L=\|f(u)\|_{\infty}$ and $E^{k+1}=\Frac{1}{2}\epsilon <Au^{(k+1)},u^{(k+1)}>+\Frac{1}{\epsilon}
<F(u^{(k+1)}-F(u^{(k)},{\bf 1}^T>$.
%
At this point, we use the bounds of the Rayleigh quotient of matrix $B$:
$$
\lambda_{min}(B)\|u\|^2 \le <Bu,u> \le \lambda_{max}(B)\|u\|^2, \forall u \in W^\perp,
$$
so, if $\tau > \beta$, then
$$
\begin{array}{ll}
\left(\lambda_{min}(B)\epsilon (\tau-\Frac{\beta}{2}) - \Frac{L}{2\epsilon}\right) \|u^{(k+1)}-u^{(k)}\|^2&\\
+\Frac{\Delta t}{2}(\tau-\beta)<B(\mu^{(k+1)}-\mu^{(k)}),\mu^{(k+1)}-\mu^{(k)}>&\\
+\Frac{\Delta t}{2}(\tau-\beta)<B\mu^{(k)},\mu^{(k)}>& \le 0.\\
+\Frac{\Delta t}{2}\left(<A\mu^{(k+1)},\mu^{(k+1)}>+\tau <B\mu^{(k+1)},\mu^{(k+1)}>\right)
+E^{k+1}-E^{k}&
\end{array}
$$
We find the unconditional stability condition
$$
\tau \ge max(\beta,\Frac{L}{2\epsilon^2 \lambda_{min}(B)}+\Frac{\beta}{2}).
$$
\end{proof}
\\
We can now give other stability results for nonlinear RSS schemes. We first consider 
\begin{eqnarray}
\label{RSS_NLCH1}
\Frac{u^{(k+1)}-u^{(k)}}{\Delta t}+\tau B(\mu^{(k+1)}-\mu^{(k)}) +A\mu^{(k)}=0,\\
\label{RSS_NLCH2}
\mu^{(k+1)}=\epsilon \tau B (u^{(k+1)}-u^{(k)})+\epsilon Au^{(k)}+\Frac{1}{\epsilon} DF(u^{(k+1)},u^{(k)}).
\end{eqnarray}
where $DF(u^{(k+1)},u^{(k)})$ is defined in Section 3.2.
\begin{theorem}\label{theo_NLCH1}
If $\tau > \beta $, then the scheme (\ref{RSS_NLCH1})-(\ref{RSS_NLCH2}) 
is unconditionally stable and
$<u^{(k+1)}-u^{(k)},{\bf 1}^T>=0$, $\forall k\ge 0$.
\end{theorem}
\begin{proof}
We proceed exactly as in Theorem \ref{theo_CH1}. We obtain after the usual simplifications
$$
\begin{array}{l}
\epsilon (\tau-\Frac{\beta}{2})<B(u^{(k+1)}-u^{(k)}),u^{(k+1)}-u^{(k)}>\\
+\Frac{\Delta t}{2}(\tau-\beta)<B(\mu^{(k+1)}-\mu^{(k)}),\mu^{(k+1)}-\mu^{(k)}>\\
+\Frac{\Delta t}{2}(\tau-\beta)<B\mu^{(k)},\mu^{(k)}>
+\Frac{\Delta t}{2}\left(<A\mu^{(k+1)},\mu^{(k+1)}>+\tau <B\mu^{(k+1)},\mu^{(k+1)}>\right)
+E^{k+1}-E^{k}\le 0,
\end{array}
$$
where 
$$
E^{k+1}=\Frac{1}{\epsilon}<F(u^{(k+1)}),{\bf 1}^T>+\Frac{1}{2}\epsilon <Au^{(k+1)},u^{(k+1)}>.
$$
The scheme is then stable when $R^{k+1}<R^k$.
Hence the sufficient stability conditions.
\end{proof}
\begin{remark}
The stabilization procedure differs from the one used in \cite{JShenACCH} which applied her gives rise to the modified scheme:
\begin{eqnarray}
\label{JSHEN_CH1}
\Frac{u^{(k+1)}-u^{(k)}}{\Delta t}+A\mu^{(k+1)}=0,\\
\label{JSHEN_CH2}
\mu^{(k+1)}= \epsilon Au^{(k+1)}+\Frac{S}{\epsilon} (u^{(k+1)}-u^{(k)})+\Frac{1}{\epsilon} f(u^{(k)}).
\end{eqnarray}
The parameter $S$ is then tuned to obtain a more stable scheme, however for large values  of $\tau'=\Frac{S}{\epsilon}$ it slows down the dynamics, particularly the energy decreases at a lower rate than that of a reference scheme, see also \cite{AbboudKosseifiChehab, BrachetChehabJSC}  for Allen-Cahn equation.
\end{remark}
We now describe the practical solution. We can write
$$
\left(
\begin{array}{ll}
Id & \tau \Delta t B\\
-\epsilon \tau B & Id\\
\end{array}
\right)
\left(
\begin{array}{l}
u^{(k+1)}-u^{(k)}\\\
\mu^{(k+1)}
\end{array}
\right)
=
\left(
\begin{array}{l}
-\Delta t  A\mu^{(k)}\\
\epsilon A u^{(k)}+\Frac{1}{\epsilon}f(u^{(k)})
\end{array}
\right).
$$
The matrix of the system can be factorized as Block LU
$$
M=\left(
\begin{array}{ll}
Id & \tau \Delta t B\\
-\epsilon \tau B & Id\\
\end{array}
\right)
=
\left(
\begin{array}{ll}
Id & 0\\
-\epsilon \tau B & Id\\
\end{array}
\right)
\left(
\begin{array}{ll}
Id & \tau \Delta t B\\
0& S\\
\end{array}
\right),
$$
where $S=Id+\tau^2\Delta t \epsilon B^2$ is the Schur complement. We have to solve the coupled linear system
$$
\left\{
\begin{array}{ll}
 X_1+\tau \Delta t B X_2=F1,\\
 -\tau \epsilon B X_1+X_2=F_2.
\end{array}
\right.
$$
Hence
$$
(Id +\tau^2 \Delta t \epsilon B^2)X_2=F_2+\epsilon \tau B F1 .
$$
Then,
$$
X_1=F_1-\tau \Delta t B X_2 .
$$
We can resume the implementation of (\ref{RSS_CH1})-(\ref{RSS_CH2}) which gives rise to the RSS-IMEX  scheme, reads as:
\begin{center}
\begin{minipage}[H]{12cm}
  \begin{algorithm}[H]
    \caption{: RSS-IMEX Cahn-Hilliard for Pattern Dynamics}\label{CH_RSS}
    \begin{algorithmic}[1]
     \For{$k=0,1, \cdots$until convergence}
           \State {\bf Set } $F_1=-\Delta t A\mu^{(k)}$ and $F_2=-\mu^{(k)}+\epsilon Au^{(k)}+\frac{1}{\epsilon}f(u^{(k)})$
             \State {\bf Solve} $ (Id +\tau^2 \Delta t \epsilon B^2)\delta\mu =F_2+\tau\epsilon B F_1$
              \State {\bf Set} $\mu^{(k+1)}=\mu^{(k)}+\delta\mu$
               \State {\bf Set } $ \delta u = F_1 - \tau \Delta t B \delta\mu$ 
               \State {\bf Set } $ \delta u = \delta u - \frac{<{\bf 1}^T,\delta u>}{n}{\bf 1}^T$ 
               \State {\bf Set } $u^{(k+1)} = u^{(k)} + \delta u$                 
            \EndFor
    \end{algorithmic}
    \end{algorithm}
\end{minipage}
\end{center}
The nonlinear RSS scheme, say the implementation of (\ref{RSS_NLCH1})-(\ref{RSS_NLCH2}) which gives rise to the NLRSS, can be obtained with inner fixed point iterations as following:
\begin{center}
\begin{minipage}[H]{12cm}
  \begin{algorithm}[H]
    \caption{: NLRSS Cahn-Hilliard for Pattern Dynamics}\label{CH_NLRSS}
    \begin{algorithmic}[1]
     \For{$k=0,1, \cdots$}
       \State {\bf Set } $u^{(k,0)}=u^{(k)}$
         \For{$m=0,1, \cdots$until convergence}
           \State {\bf Set } $F_1=-\Delta t A\mu^{(k)}$
           \State {\bf Set } $F_2=-\mu^{(k)}+\epsilon Au^{(k)}+\frac{1}{\epsilon}DFf(u^{(k,m)},u^{(k)})$
             \State {\bf Solve} $ (Id +\tau^2 \Delta t \epsilon B^2)\delta\mu =F_2+\tau\epsilon B F_1$
              \State {\bf Set} $\mu^{(k,m+1)}=\mu^{(k)}+\delta\mu$
              \State {\bf Set } $ \delta u = F_1 - \tau \Delta t B \delta\mu$ 
              \State {\bf Set} $\delta u = \delta u - \frac{<{\bf 1}^T,\delta u>}{n}{\bf 1}^T$
               \State {\bf Set } $ u^{(k,m+1)}=u^{(k)}+\delta u$  
               \EndFor                       
                \State {\bf Set} $\mu^{(k+1)}=\mu^{(k,m+1)}$
               \State {\bf Set } $ u^{(k+1)}=u^{(k,m+1)}$
            \EndFor
    \end{algorithmic}
    \end{algorithm}
\end{minipage}
\end{center}
When considering the inpainting model, the RSS-IMEX scheme can be written as
\begin{eqnarray}
\label{RSS_IMP_CH2}
\Frac{u^{(k+1)}-u^{(k)}}{\Delta t}+\tau B(\mu^{(k+1)}-\mu^{(k)}) +A\mu^{(k)} +\lambda_0D(u^{(k+1)}-g)=0,\\
\label{RSS_IMP_CH2}
\mu^{(k+1)}=\epsilon \tau B (u^{(k+1)}-u^{(k)})+\epsilon Au^{(k)}+\Frac{1}{\epsilon} f(u^{(k)}).
\end{eqnarray}
say in the matricial form
$$
\left(
\begin{array}{ll}
Id +\Delta t \lambda_0 D& \tau \Delta t B\\
-\epsilon \tau B & Id\\
\end{array}
\right)
\left(
\begin{array}{l}
u^{(k+1)}-u^{(k)}\\\
\mu^{(k+1)}-\mu^{(k)}
\end{array}
\right)
=
\left(
\begin{array}{l}
\Delta t  (\lambda_0D(g-u^{(k)})-A\mu^{(k)})\\
\epsilon A u^{(k)}+\Frac{1}{\epsilon}f(u^{(k)})-\mu^{(k)}
\end{array}
\right).
$$
The implementation of the scheme reads as
\begin{center}
\begin{minipage}[H]{12cm}
  \begin{algorithm}[H]
    \caption{: RSS-IMEX Cahn-Hilliard for inpainting }\label{CH_INPAINTING_RSS}
    \begin{algorithmic}[1]
     \For{$k=0,1, \cdots$ until convergence}
           \State {\bf Set } $F_1=\Delta t  (\lambda_0D(g-u^{(k)})-A\mu^{(k)})$ 
            \State {\bf Set }  $F_2=-\mu^{(k)}+\epsilon Au^{(k)}+\Frac{1}{\epsilon}f(u^{(k)})$
             \State {\bf Solve} $ (Id + \Delta t \lambda_0 D +\tau^2 \Delta t \epsilon B^2)\delta u =F_1-\tau \Delta t BF_2$
              \State{\bf Set} $\delta\mu=F_2+\epsilon\tau B\delta$
               \State {\bf Set } $ u^{(k+1)}=u^{(k)}+ \delta u$             
              \State {\bf Set} $\mu^{(k+1)}=\mu^{(k)}+\delta\mu$                         
            \EndFor
    \end{algorithmic}
    \end{algorithm}
\end{minipage}
\end{center}
\begin{remark}
Stabilization of semi-implicit scheme for Cahn-Hilliard equations have been considered, e.g. in \cite{Bertozzi1,Bertozzi2,Fakih} for inpainting problems, 
\begin{eqnarray}
\label{Bertozzi_CH1}
\Frac{u^{(k+1)}-u^{(k)}}{\Delta t}+A\mu^{(k+1)}+ c_1A(u^{(k+1)}-u^{(k)})+c_2(u^{(k+1)}-u^{(k)})+\lambda_0 D(u^{(k)}-g)=0,\\
\label{Bertozzi_CH2}
\mu^{(k+1)}=\epsilon Au^{(k+1)}+\Frac{1}{\epsilon} f(u^{(k)}).
\end{eqnarray}
Here $c_1$ and $c_2$ are positive constants, they play the role of stabilization parameters. Large values of $c_1$ and $c_2$ allow to take large time step, however it damages the dynamics. Our approach is here different.
\end{remark}
%
%
\section{Numerical Results}
\subsection{Implementation}
The applications we are interested with are Allen-Cahn and Cahn-Hilliard equations to which homogeneous Neumann boundary conditions are associated.
We proceed as in \cite{BrachetChehabJSC} and we first discretize in space the equation with high order finite difference compact schemes; the matrix $A$ corresponds then to the laplacean with Homogenous Neumann BC (HNBC).  Matrix $B$ is the (sparse) second order laplacean matrix with Neumann Homogeneous Boundary conditions.
For a fast solution of linear systems in the RSS, we will use the cosine-FFT to solve the Neumann problems with matrix $Id +\tau \Delta t B$.
Also, the mean null property of the linear part of Allen-Cahn equations (for splitting schemes) and for Cahn-Hilliard equations is imposed by the projection described in Section 2.4.4.\\
All the computations have been realized using Matlab  \textregistered.
\subsection{Allen-Cahn equation}
\subsubsection{Pattern Dynamics}
We here give results on the simulation of phase separation, with Allen-Cahn's equation in 2D ans in 3D. We use a splitting RSS-scheme:
We take  $\varepsilon=0.01$.\\
\\
We used here CS2 compact schemes for the space discretization and used the 2 dimensional cosine FFT for the solution of the linear systems in the RSS schemes; similar results are obtained using Lele's compact scheme. \\
In 2D, we chose $u_0(x,y)=\cos ( \pi x ) \cos ( 2 \pi y )$ and in 3D $u_0(x,y)=\cos ( \pi x ) \cos ( 2 \pi y ) \cos (6 z)$.
\\
The 2D results are given in Table \ref{tab: AC 2D} the 3D ones in Table \ref{tab: AC 3D}.
\\
\begin{table}[htbp!]
\begin{center}
\begin{tabular}{|c|ccc|}
\hline
                & IMEX &  RSS ($\tau = 2)$ & RSS-ADI ($\tau = 2$) \\
\hline
\hline
CPU time $N=16$ & $2.3809  $ & $0.5199 $  & $0.61618 $ \\
CPU time $N=32$ & $19.2752 $ & $0.72124$  & $0.7424  $ \\
CPU time $N=64$ & $426.7139$ & $3.5574 $  & $3.4405  $ \\
CPU time $N=128$& too long   & $21.6112$  & $21.6212 $ \\
\hline           
\end{tabular}
\caption{2D Allen-Cahn equation. Final time $t_{\max} = 0.01$, $\Delta t = 0.0001$, $\varepsilon = 0.01$. CPU time in seconds.}
\label{tab: AC 2D} 
\end{center}
\end{table}
\begin{table}[htbp!]
\begin{center}
\begin{tabular}{|c|ccc|}
\hline
                & IMEX &  RSS ($\tau = 2)$ & RSS-ADI ($\tau = 2$) \\
\hline
\hline
CPU time $N=8$  & $5.1924  $ & $0.7246 $  & $0.8054  $ \\
CPU time $N=16$ & $313.8788$ & $2.1995 $  & $2.0523  $ \\
CPU time $N=32$ & too long.  & $19.3381$  & $18.9964 $ \\
CPU time $N=64$ & too long.  & $353.88 $  & $363.0721$ \\
\hline           
\end{tabular}
\caption{3D Allen-Cahn equation. Final time $t_{\max} = 0.01$, $\Delta t = 0.0001$, $\varepsilon = 0.01$. CPU time in seconds.}
\label{tab: AC 3D} 
\end{center}
\end{table}
Clearly, the fast solution of the linear part of the problems allows to obtain an important reduction of the CPU time and, naturally, the gain in time computing increases when the dimension of the problem increases while the solutions are comparable; also, the IMEX-RSS produces comparable solution to IMEX scheme for moderate values of $\tau$ (this is not the case with RSS-ADI for which we observe that large values of $\tau$ of of $\Delta t$ do not make the energy decreasing in time), see illustrations below.
\\
\begin{figure}[htbp!]
\label{fig: AC 2D} 
\begin{center}
\includegraphics[height=4cm]{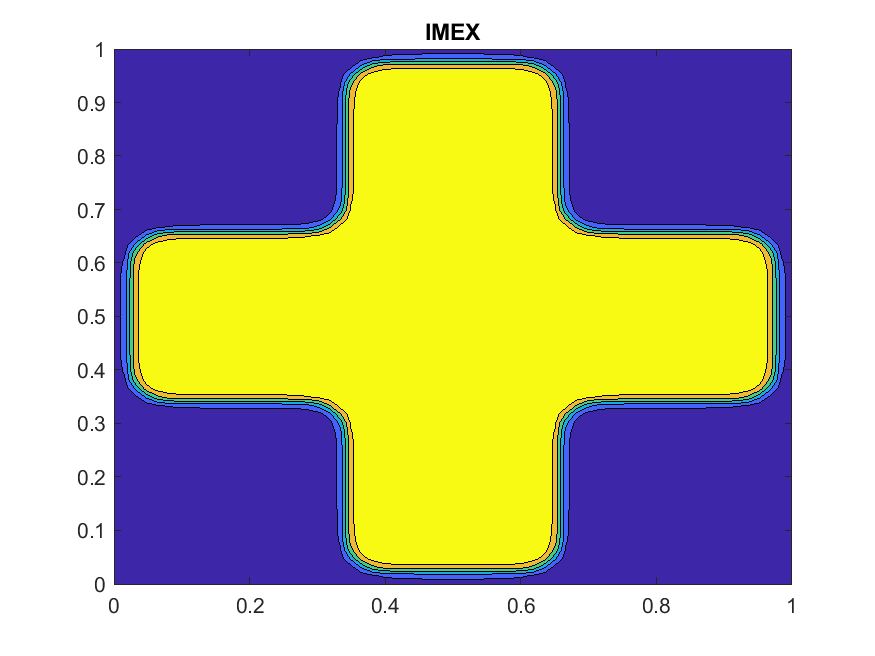}
\includegraphics[height=4cm]{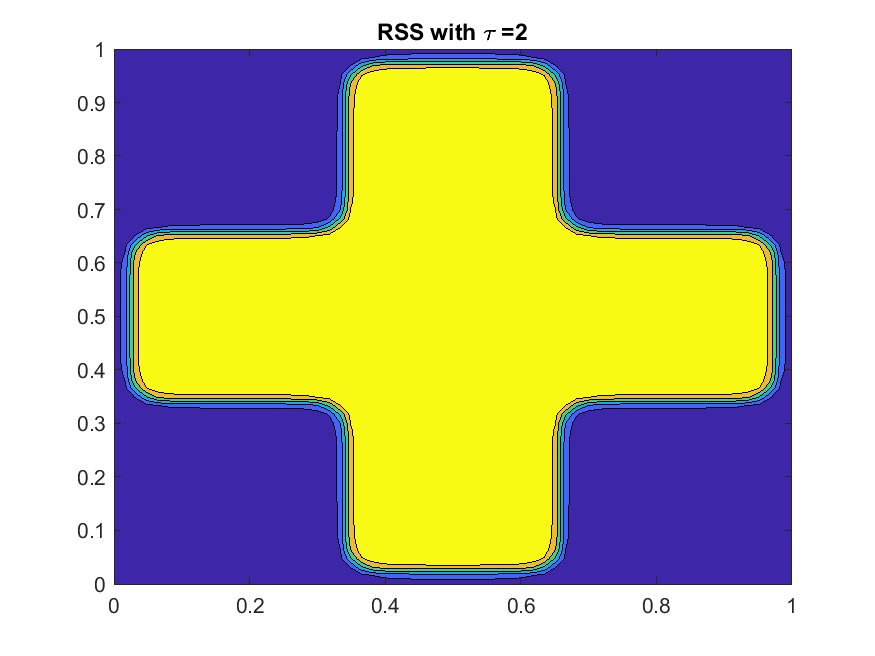}
\includegraphics[height=4cm]{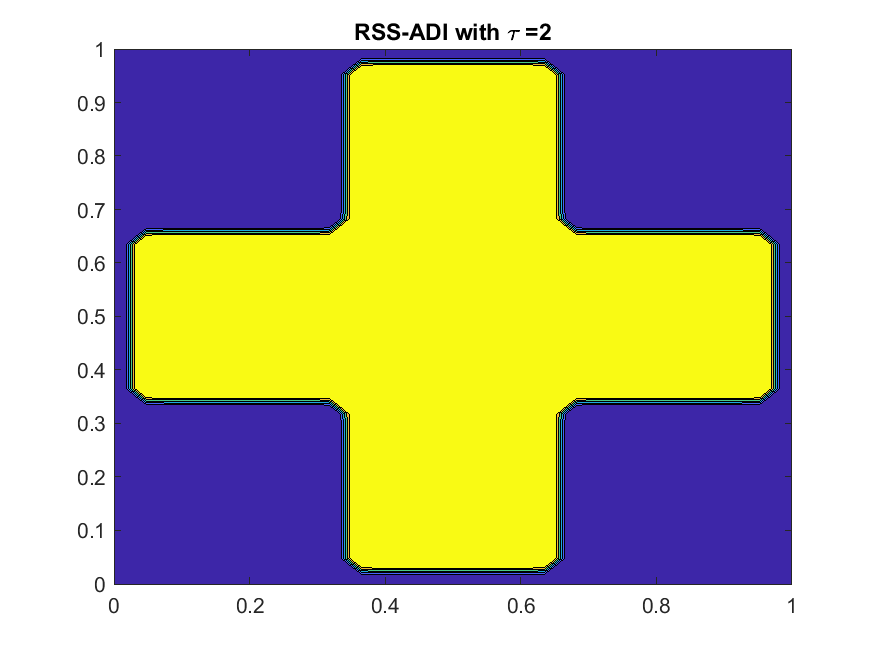}\\
\includegraphics[height=4cm]{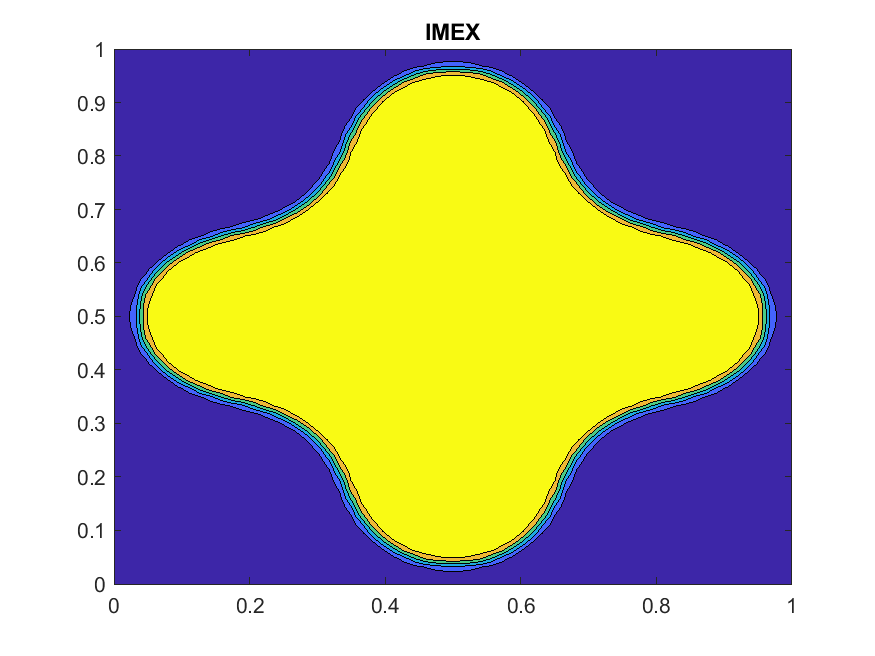}
\includegraphics[height=4cm]{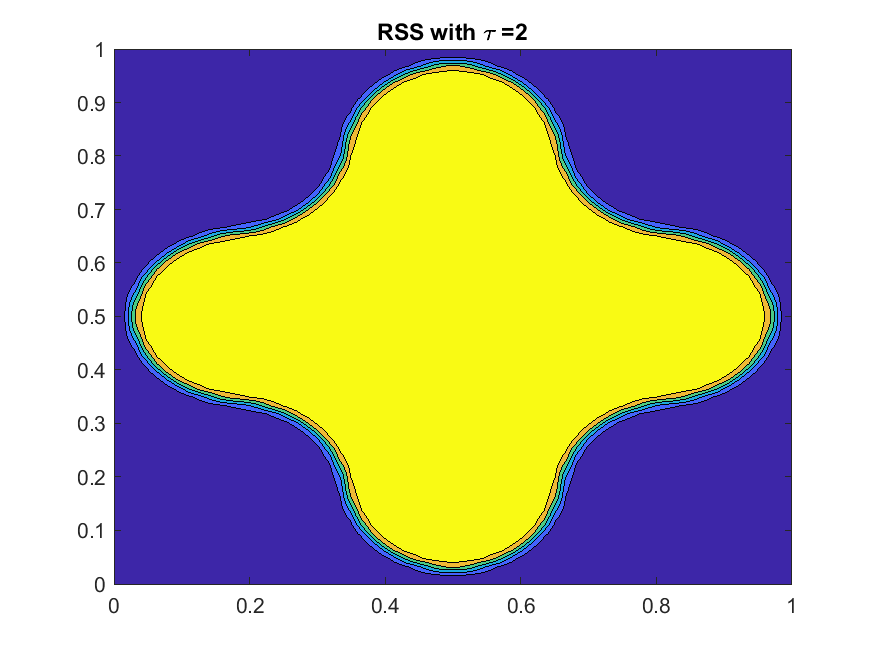}
\includegraphics[height=4cm]{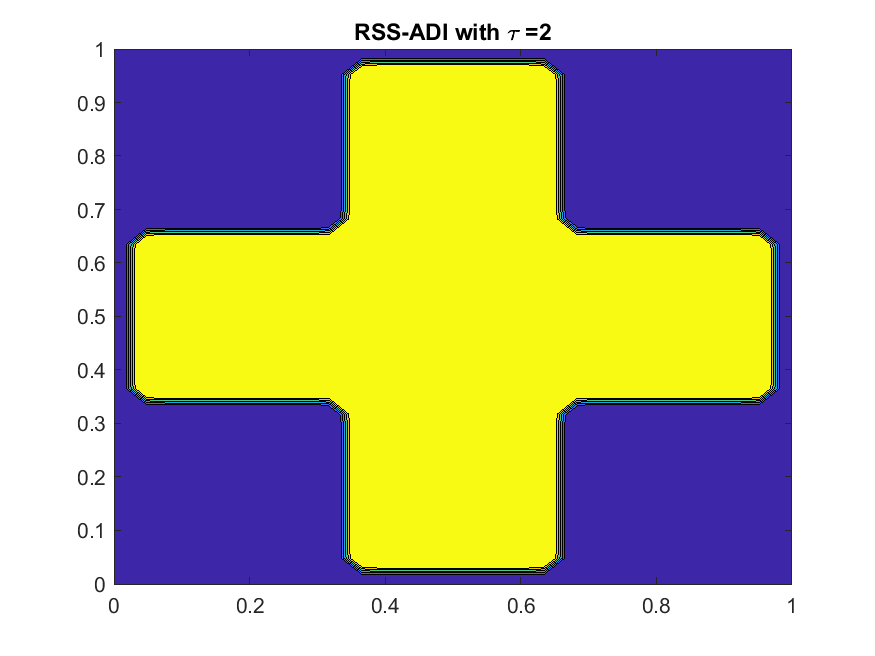}\\
\includegraphics[height=4cm]{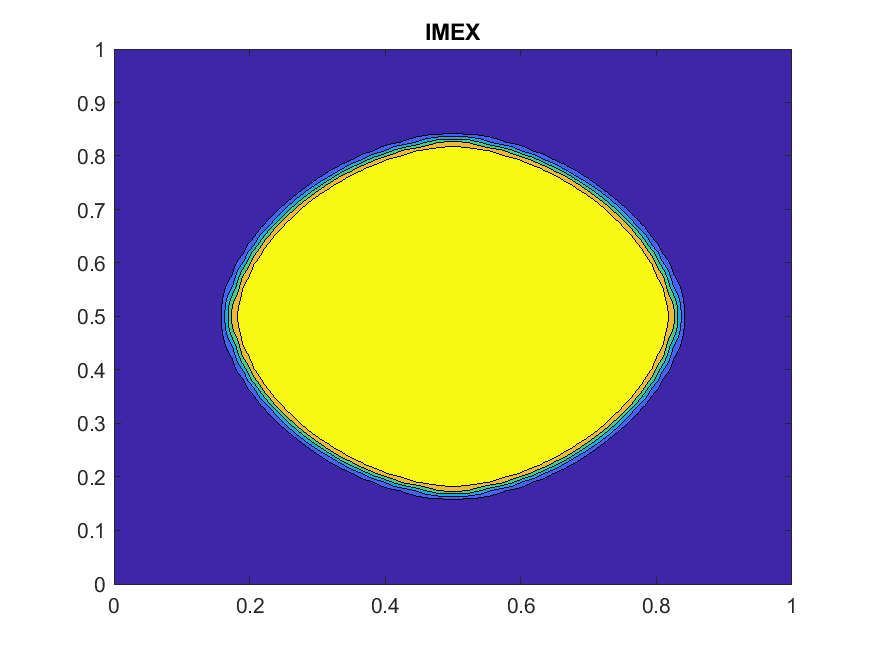}
\includegraphics[height=4cm]{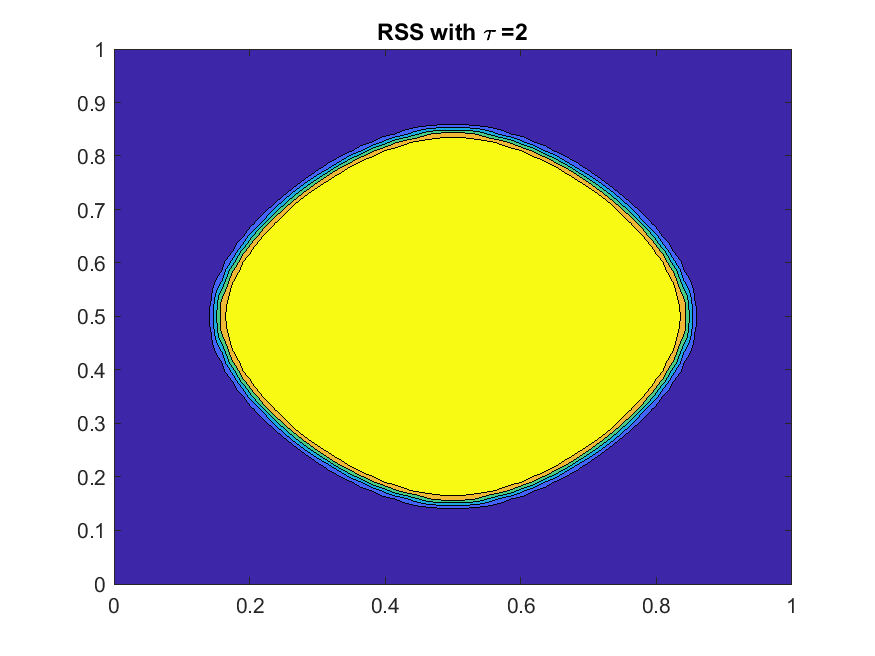}
\includegraphics[height=4cm]{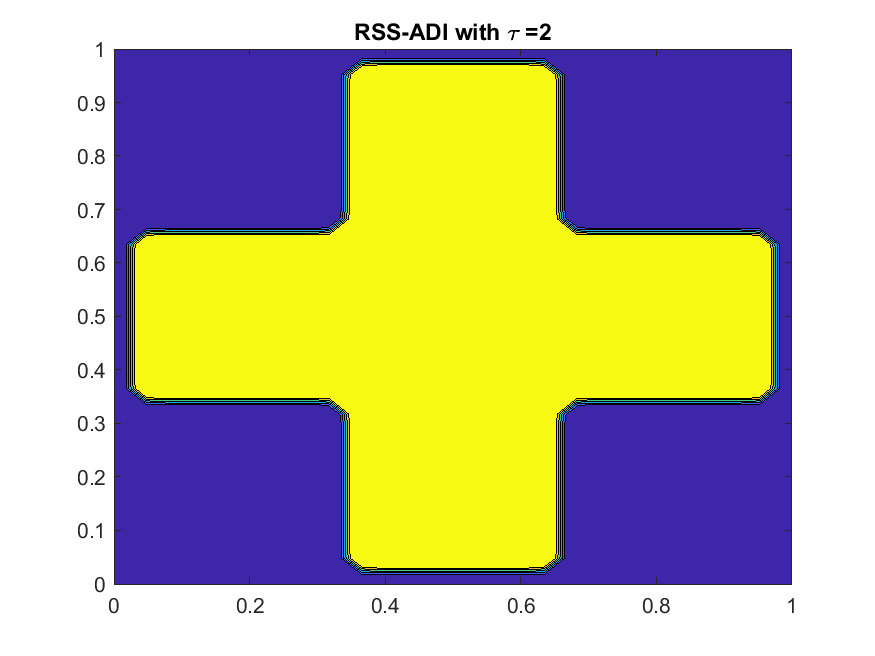}\\
\end{center}
\caption{Solution of the 2D Allen-Cahn equation with different time schemes. The initial condition is a given cross. Line by line, the numerical solution are at time $t=10^{-3}$, $t=10^{-2}$ and $t=10^{-1}$. The parameters are $\varepsilon = 10^{-2}$, $N=64$, $\Delta t = 10^{-4}$ and $\tau = 2$.}
\end{figure}
\\
\begin{figure}
\label{fig: AC 2D ENER}
\begin{center}
\includegraphics[height=5cm]{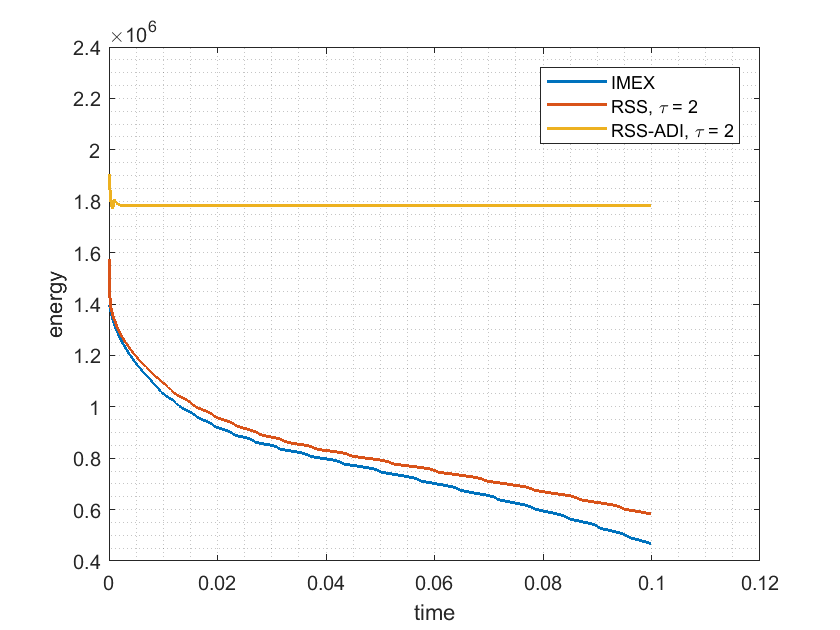}
\includegraphics[height=5cm]{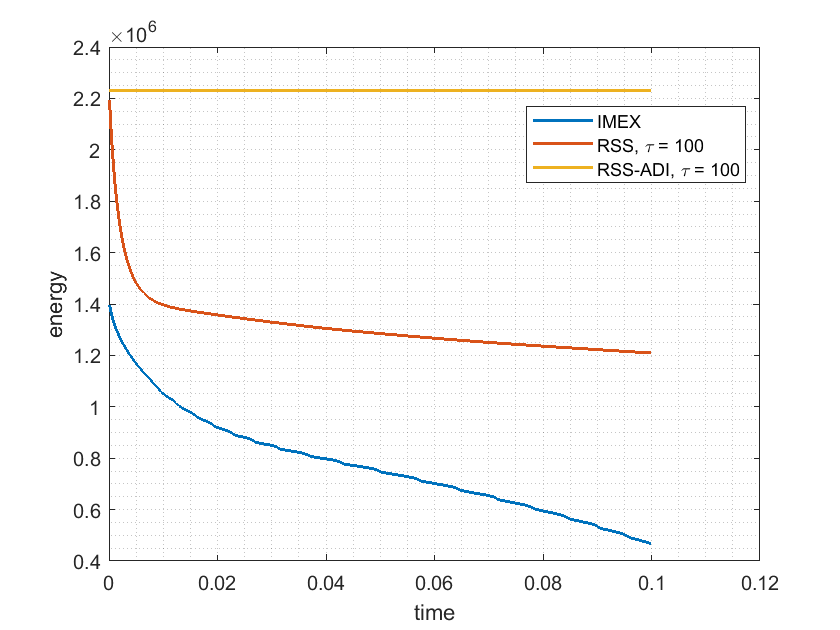}
\end{center}
\caption{History of the numerical energy for the 2D Allen-Cahn equation with different time schemes. The initial condition is is a given cross. The parameters are $\varepsilon = 10^{-2}$, $N=64$ and $\Delta t = 10^{-4}$.}
\end{figure}
\\
\begin{figure}[htbp!]
\label{fig: AC 3D} 
\begin{center}
\includegraphics[height=4cm]{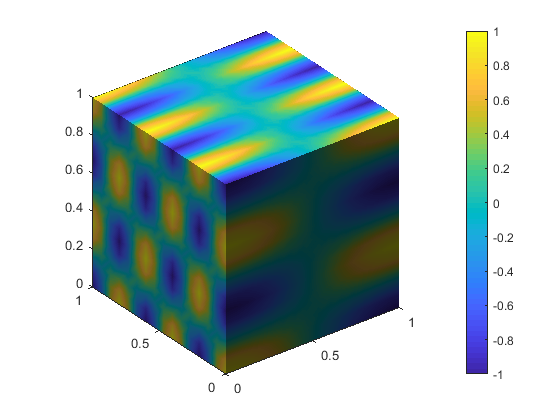}\\
\includegraphics[height=4cm]{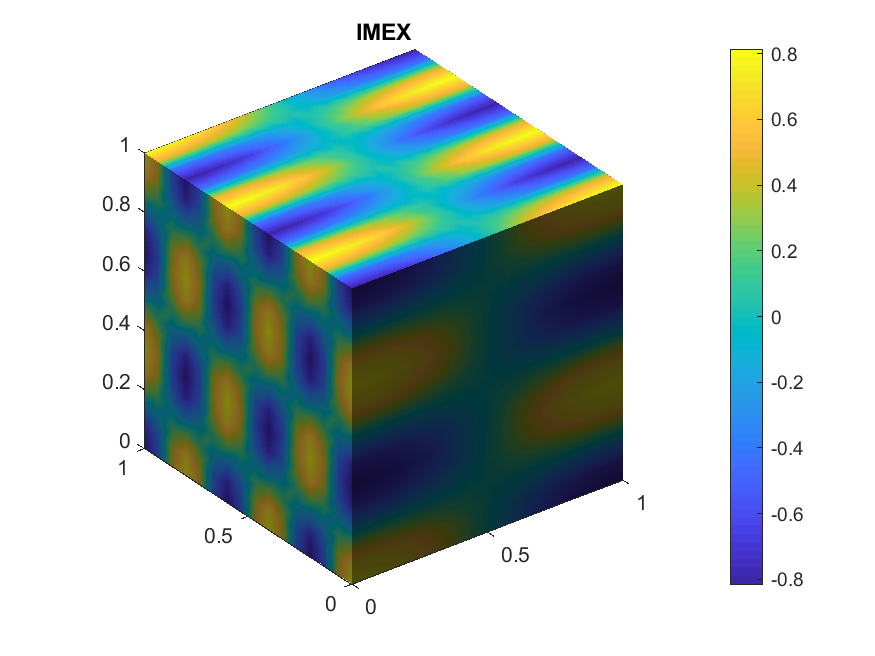}
\includegraphics[height=4cm]{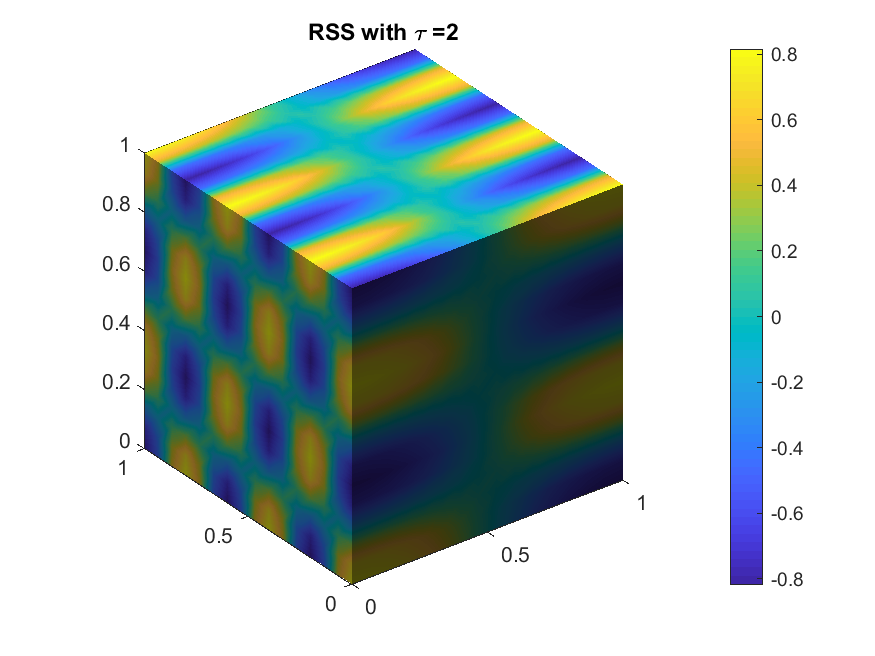}
\includegraphics[height=4cm]{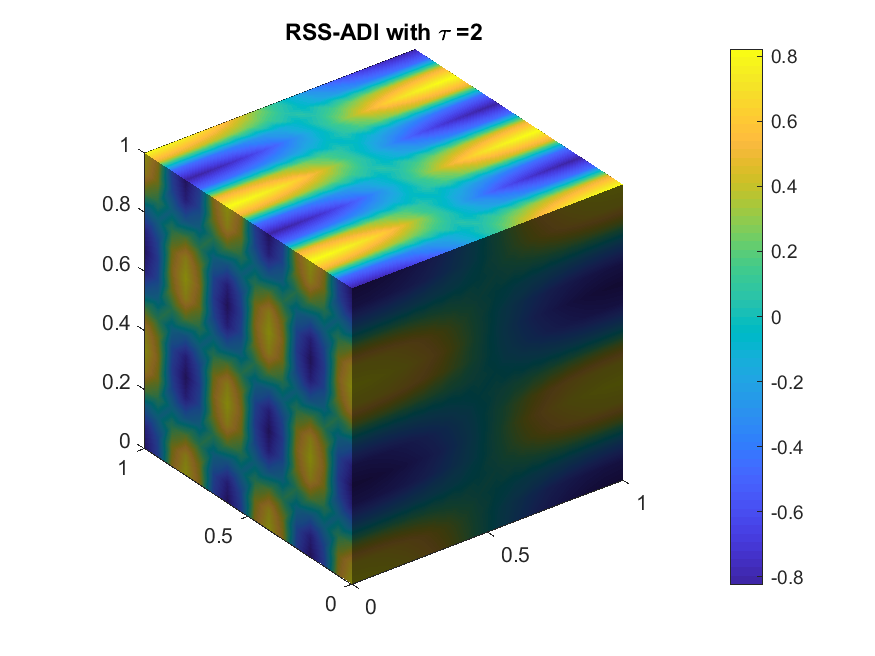}\\
\includegraphics[height=4cm]{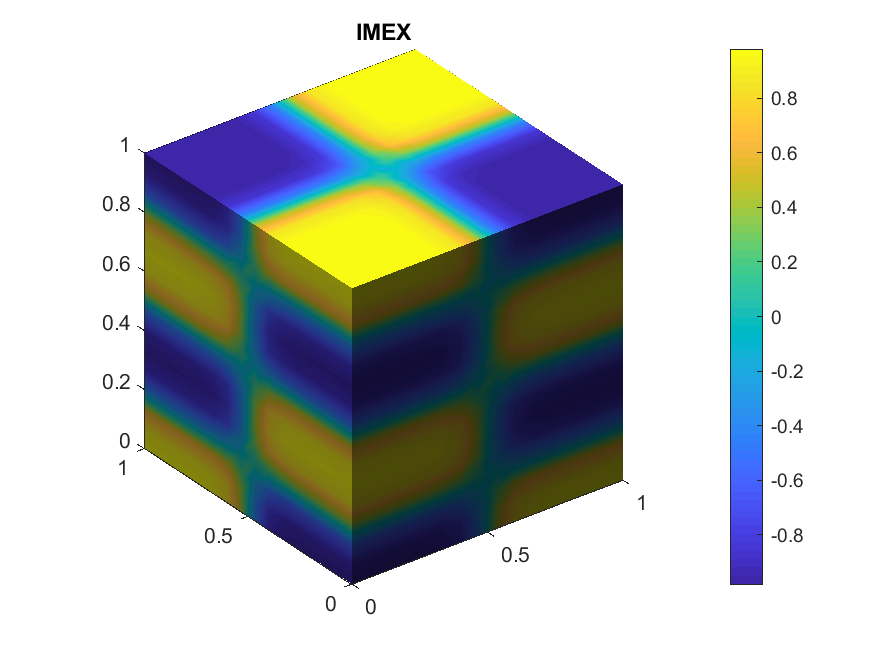}
\includegraphics[height=4cm]{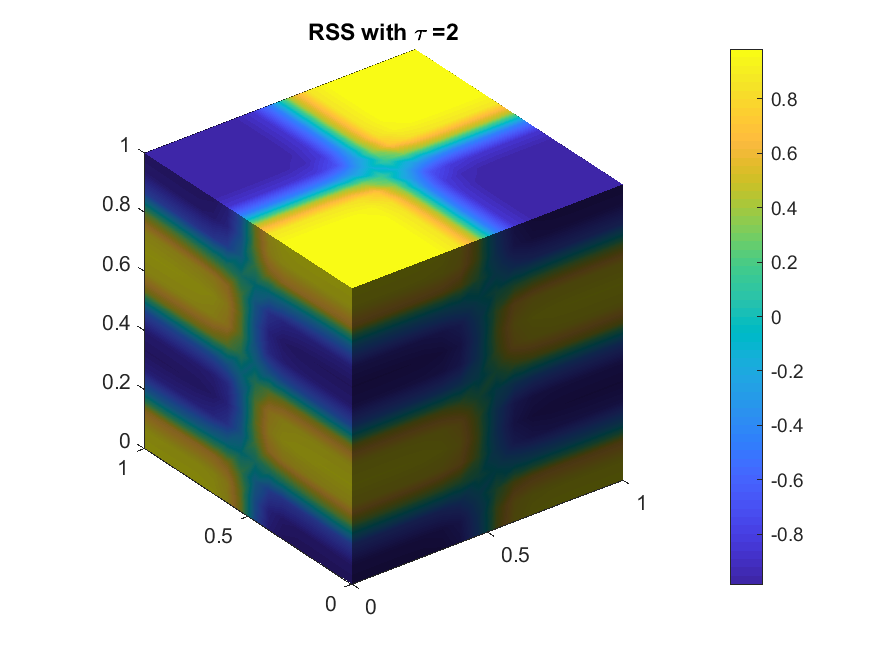}
\includegraphics[height=4cm]{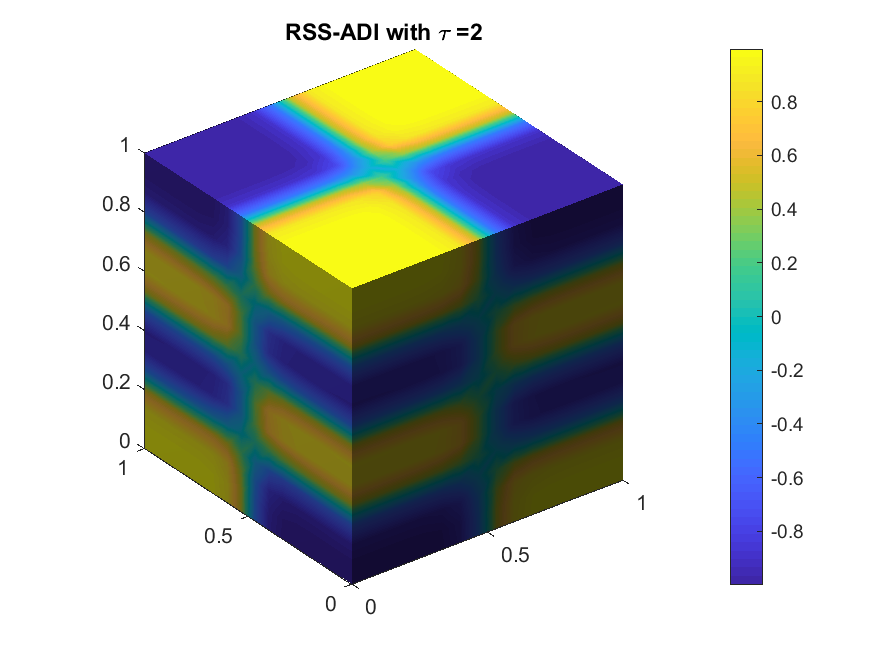}\\
\includegraphics[height=4cm]{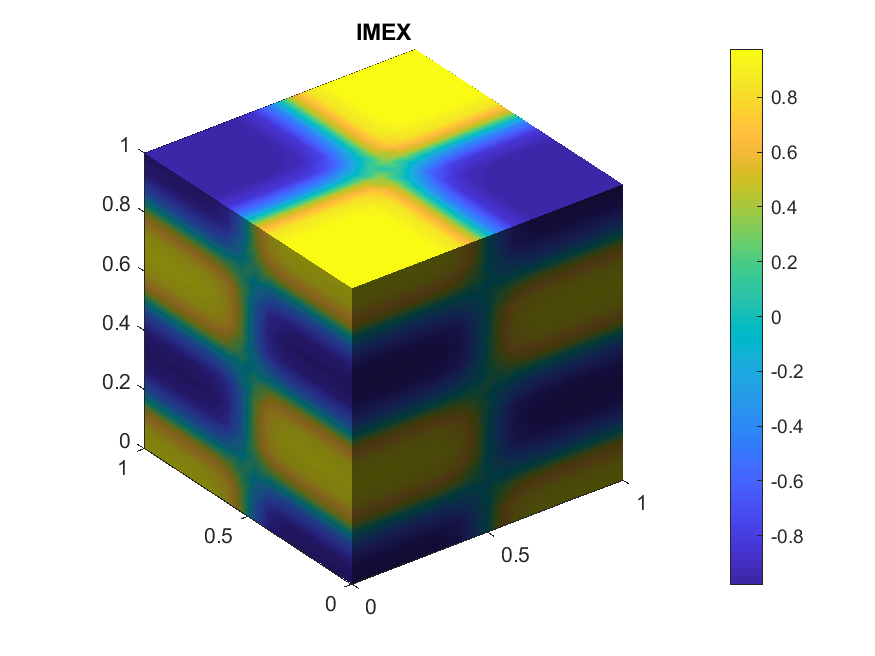}
\includegraphics[height=4cm]{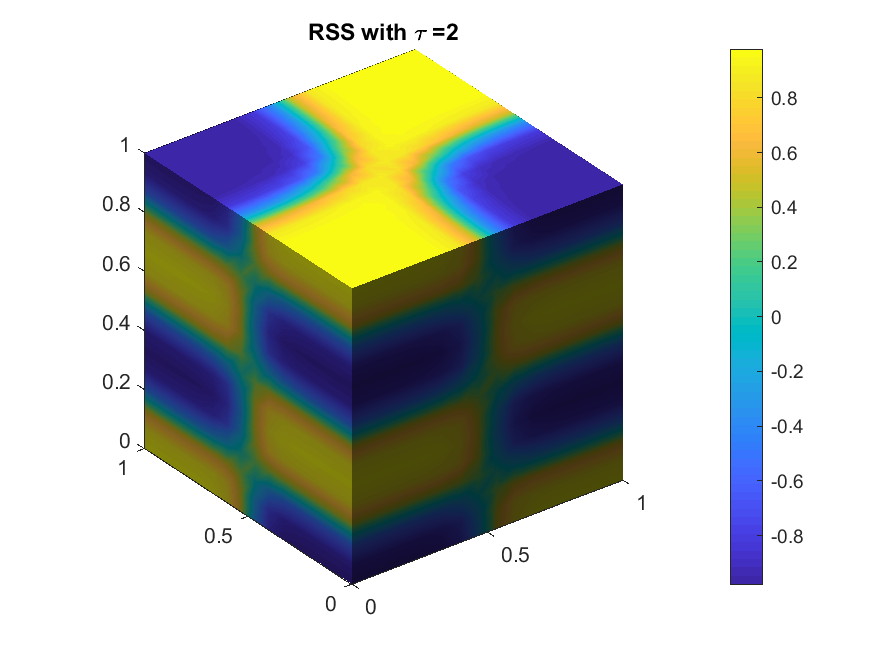}
\includegraphics[height=4cm]{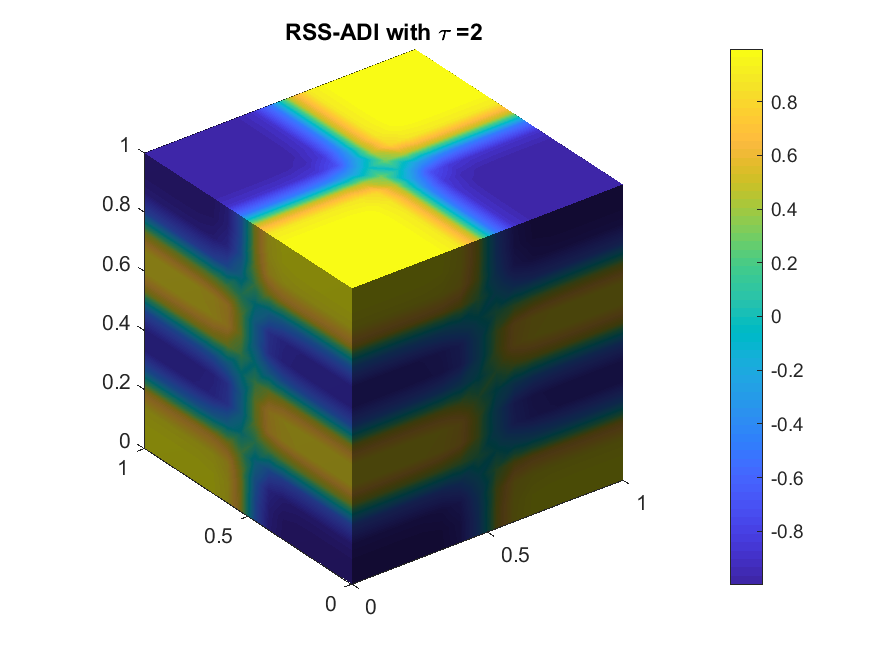}\\
\includegraphics[height=4cm]{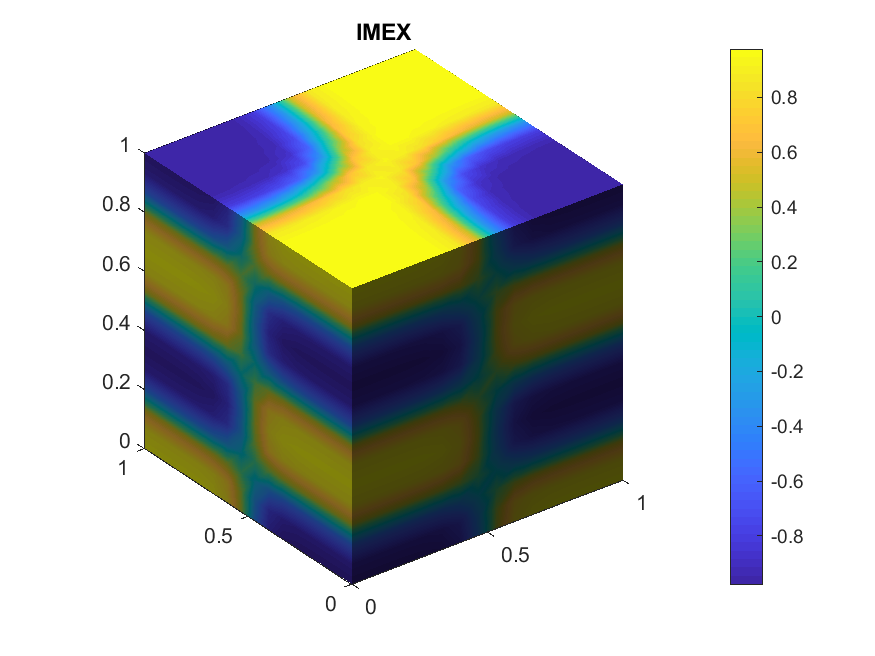}
\includegraphics[height=4cm]{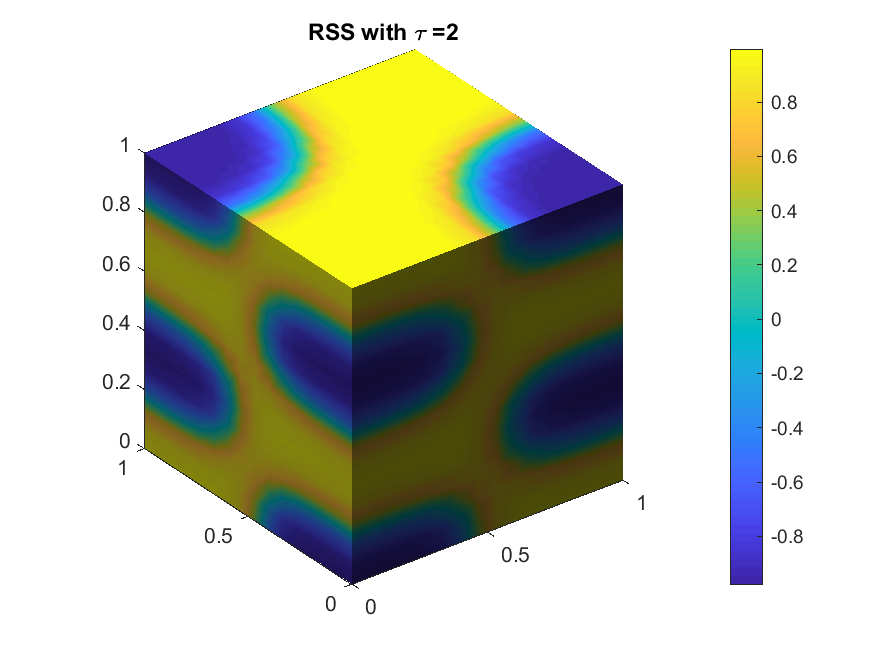}
\includegraphics[height=4cm]{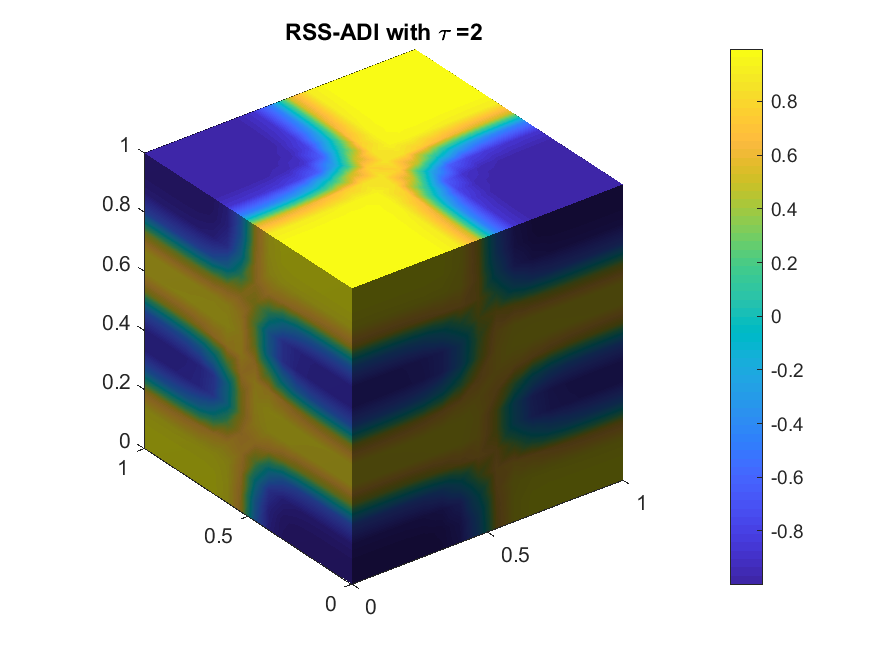}\\
\end{center}
\caption{Solution of the 3D Allen-Cahn equation with different time schemes. The initial condition is $u_0(x,y,z) = \cos ( \pi x ) \cos (5 \pi y ) \cos (3 \pi z)$. Line by line, the numerical solution are at time $t=0$, $t=0.001$, $t=0.1$, $t=0.15$ and $t=0.16$. The parameters are $\varepsilon = 0.05$, $N=16$, $\Delta t = 10^{-4}$ and $\tau = 2$.}
\end{figure}
\begin{figure}
\label{fig: AC 3D ENER}
\begin{center}
\includegraphics[height=5cm]{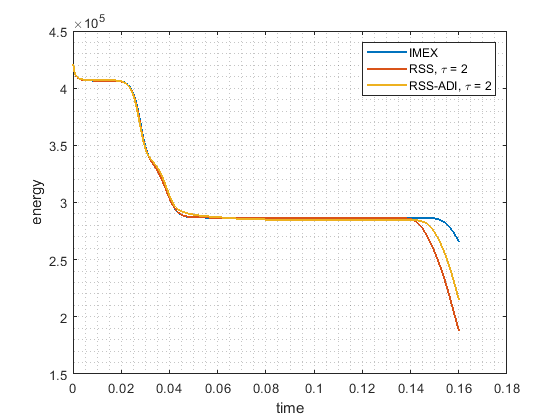}
\includegraphics[height=5cm]{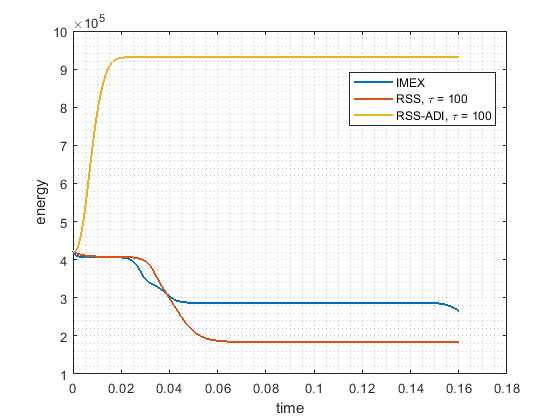}
\end{center}
\caption{History of the numerical energy for the 3D Allen-Cahn equation with different time schemes. The initial condition is $u_0(x,y,z) = \cos ( \pi x ) \cos (5 \pi y ) \cos (3 \pi z)$. The parameters are $\varepsilon = 0.05$, $N=16$, $\Delta t = 10^{-4}$ and $\tau = 2$.}
\end{figure}
\\
As shown above, the splitting scheme can enjoy of a discrete maximum property, for sufficiently small $\Delta t$, however there is no guarantee of a decreasing of the energy if $\Delta t$ is not small enough.
We observe here, for the same numerical and physical data that the RSS-IMEX scheme allows to capture the dynamics while this is not the case with the Splitting RSS for which the energy exhibits oscillations in time.
\clearpage
\subsubsection{Image Segmentation}
The RSS method is derived form the splitting scheme proposed in \cite{LiLee} (that we recover in the case $\tau=1$ and $B=A$) and reads as
\begin{center}
\begin{minipage}[H]{12cm}
  \begin{algorithm}[H]
    \caption{: RSS-splitting for Image segmentation with  Allen Cahn }\label{splitt_AC_IG}
    \begin{algorithmic}[1]
     \For{$k=0,1, \cdots$}
     \State {\bf Set  \ }$$\begin{array}{l}c^{(k)}_1=\Frac{\int_{\Omega}f_0(x)(1+\phi^{(k)}(x))dx}{\int_{\Omega}(1+\phi^{(k)}(x))dx}\\
c^{(k)}_2=\Frac{\int_{\Omega}f_0(x)(1-\phi^{(k)}(x))dx}{\int_{\Omega}(1-\phi^{(k)}(x))dx}
\end{array}$$
     \State {\bf Solve} $\Frac{\phi^{(k+1/3)}-\phi^{(k)}}{\Delta t}=-\lambda \left((1+\phi^{(k+1/3)})(f_0-c^{(k)}_1)^2-(1-\phi^{(k+1/3})(f_0-c^{(k)}_2)^2\right)$
          \State {\bf Solve} $ (Id +\tau \Delta t  B) \delta \phi =-\Delta tA \phi^{(k+/3)}$
                      \State {\bf Set } $ \phi^{(k+2/3)}=\phi^{(k+1/3)}+\delta \phi$
                       \State {\bf Set }
                       $\phi^{(k+1)}=\Frac{\phi^{(k+2/3)}}{\sqrt{e^{-2\frac{\Delta t}{\epsilon^2}} +(\phi^{(k+2/3)})^2(1-e^{-2\frac{\Delta t}{\epsilon^2}}) }}
$                      
            \EndFor
    \end{algorithmic}
    \end{algorithm}
\end{minipage}
\end{center}
In our numerical experiments,  the given image $f$ is normalized with $f_0=\Frac{f-f_{min}}{f_{max}-f_{min}}$, where $f_{max}$ and $f_{min}$ are the maximum
and the minimum values of the given image, respectively, so we have $f_0\in [0,1]$. The initial condition is $\phi=2f_0-1$ and $\Omega=]0,1[^2$.\\
We here apply a post-processing similar to the one described in section 4.2.1 for the inpainting problem to obtain sharp boundaries. The thresholded images are labelled {\it segmented image} while those
computed at the final time $t^*$ (then before thresholding) are labelled {\it segmented image at $t=t*$}.
We observe in Figures \ref{SEG1} and \ref{SEG2} the segmentation process of two classical images: the results are satisfactory, they are in good agreement with those of the literature; the extremal values of the two phases  are respectively very close to $-1$ and $1$ during the time evolution and the convergence to the steady state (segmented image at a sufficient large given time) is linear.\\
\\
We used Lele's compact schemes for the space discretization and used the 2 dimensional cosine FFT for the solution of the linear systems in the RSS schemes; similar results are obtained using CS2 compact scheme. \\
\begin{figure}[htbp!]
\vskip -1.5cm
\begin{center}
\includegraphics[height=6cm]{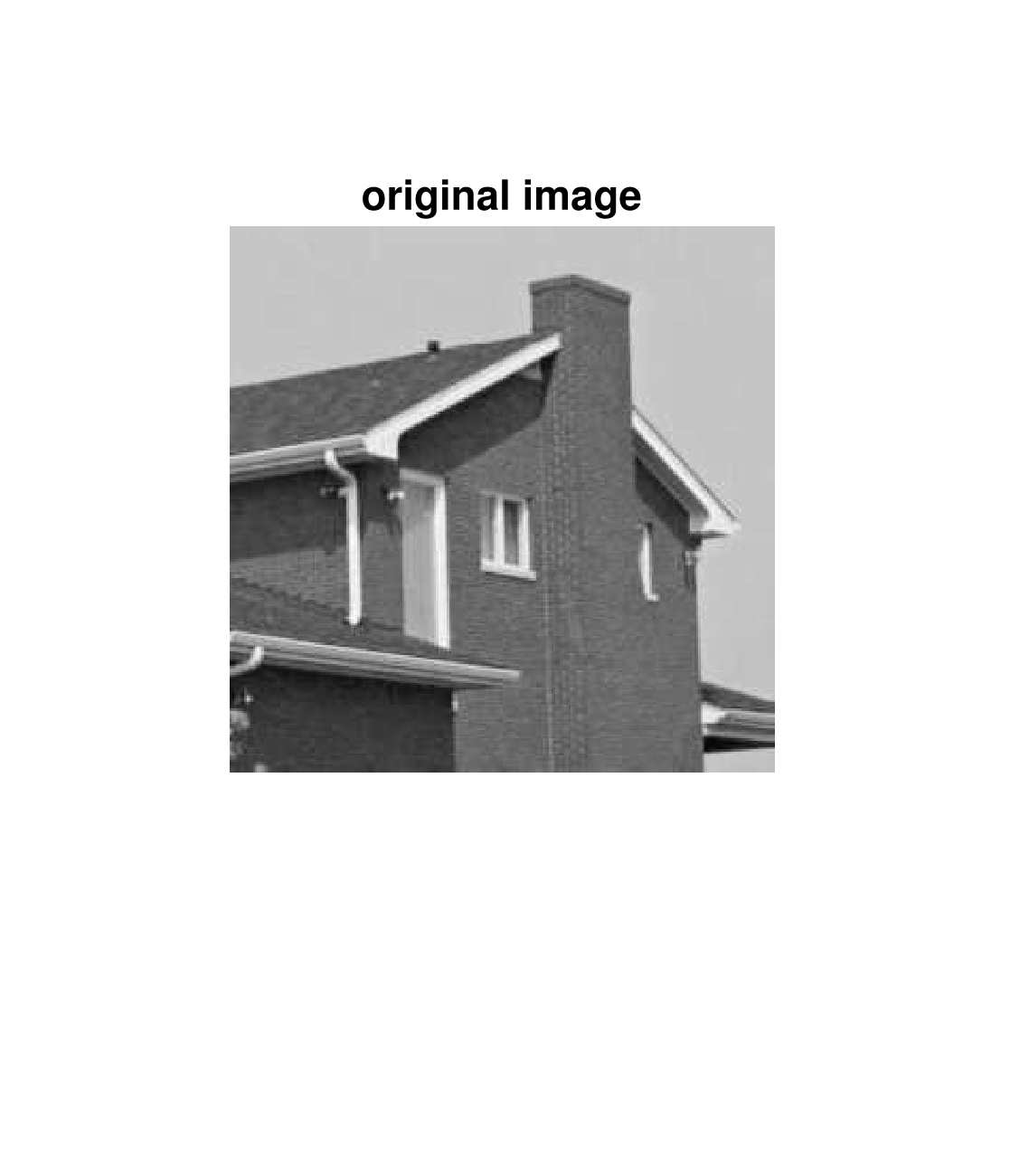}
\includegraphics[height=6cm]{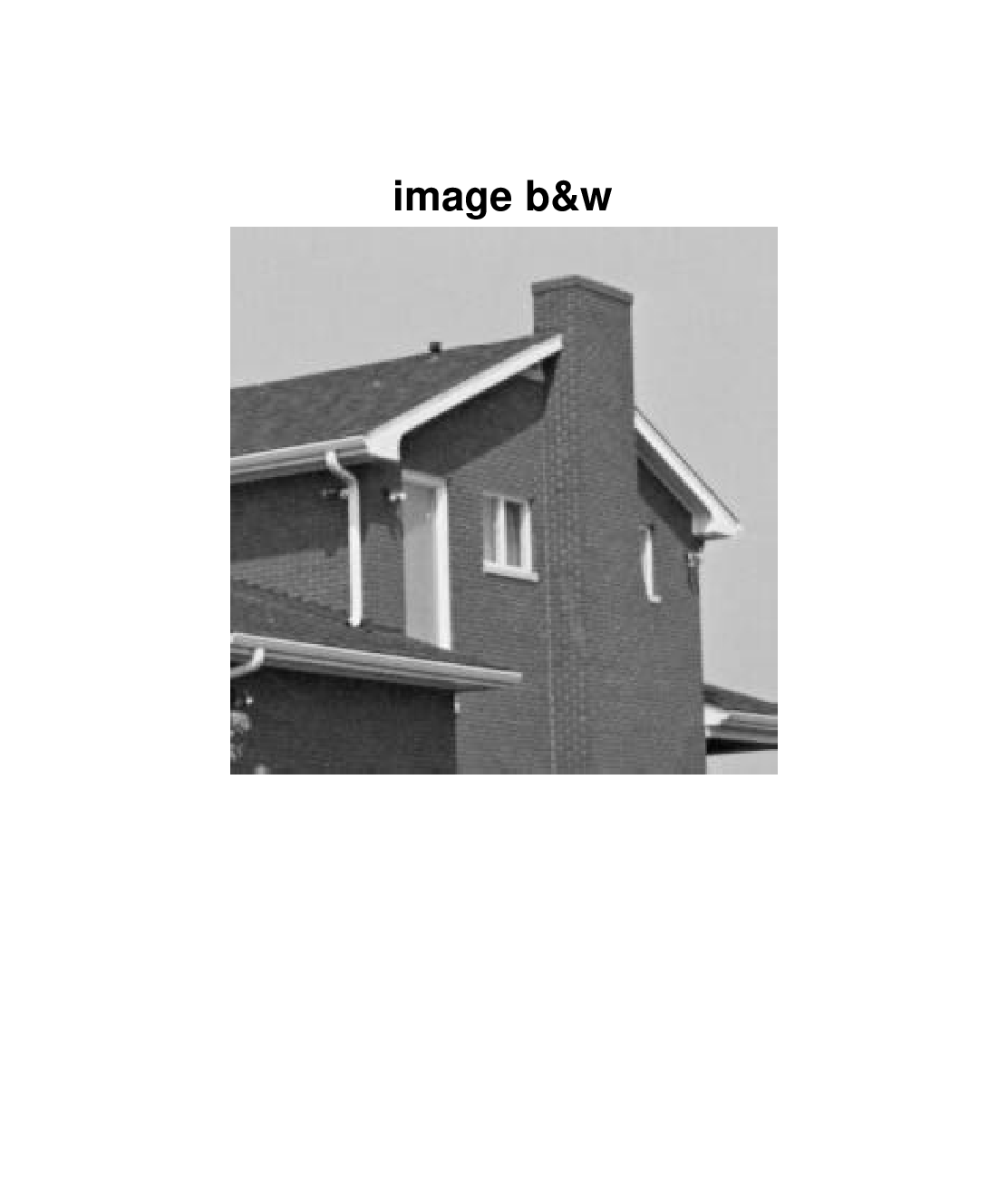}\\
\includegraphics[height=6cm]{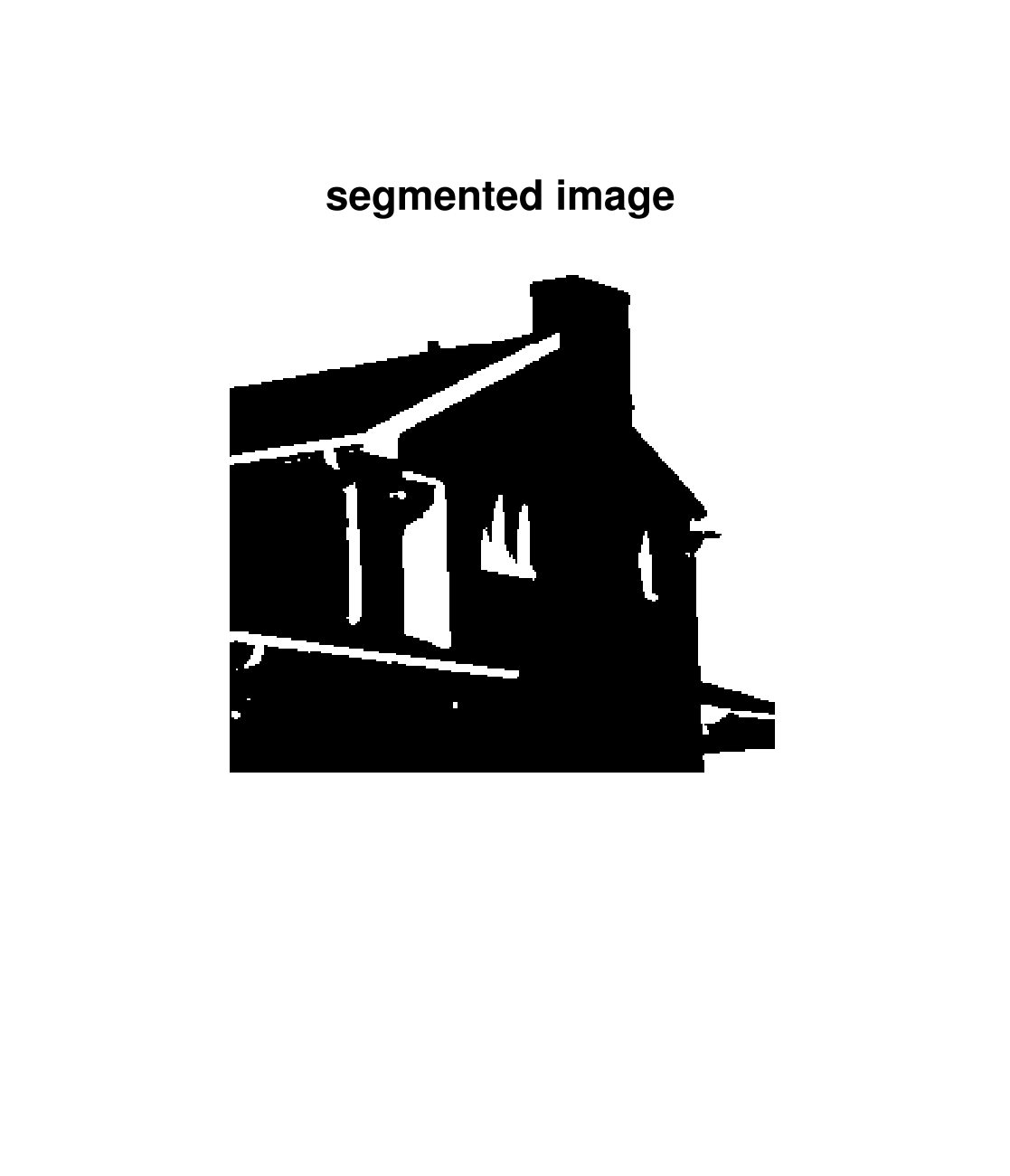}
\includegraphics[height=6cm]{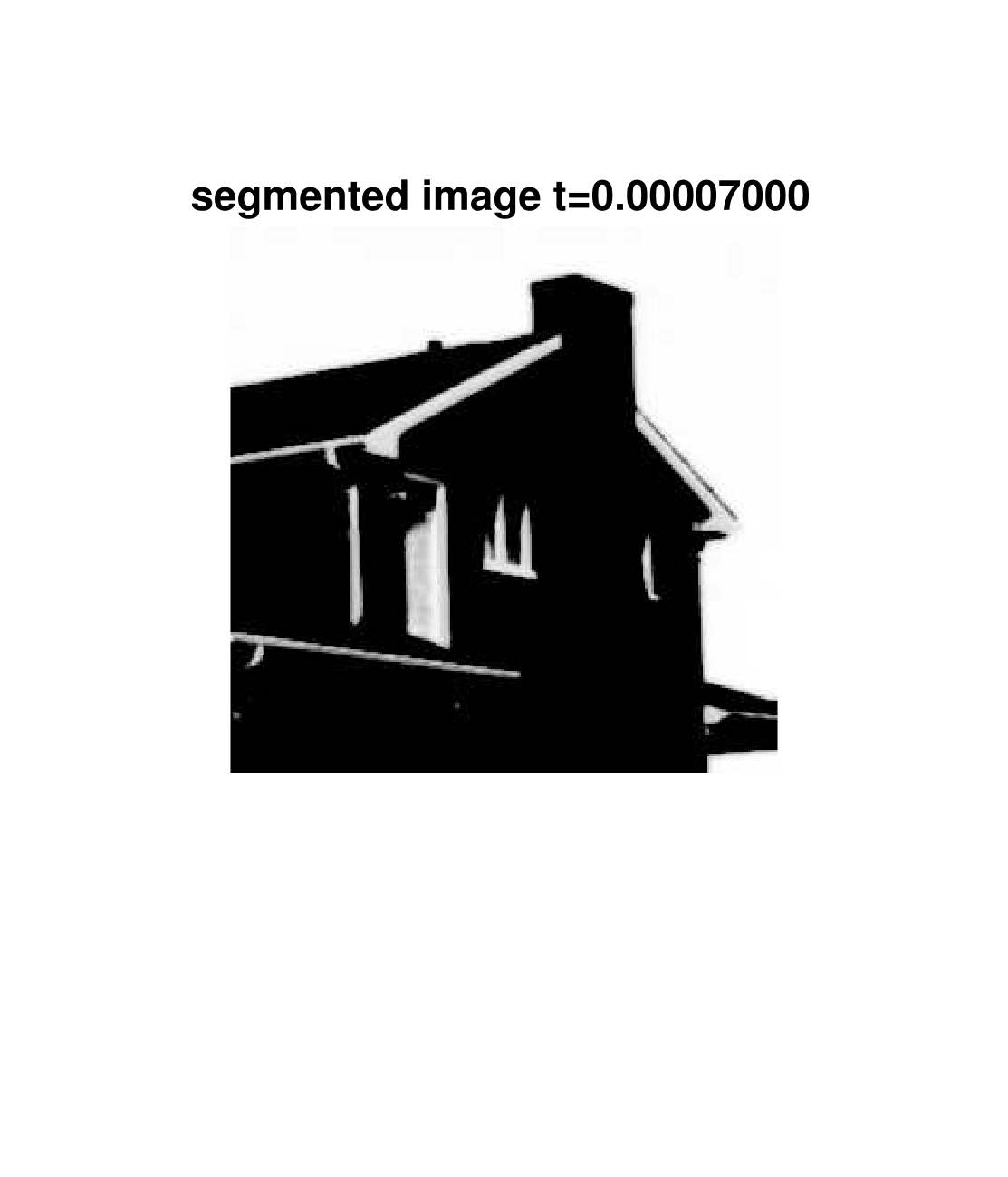}\\
\includegraphics[width=6cm,height=5cm]{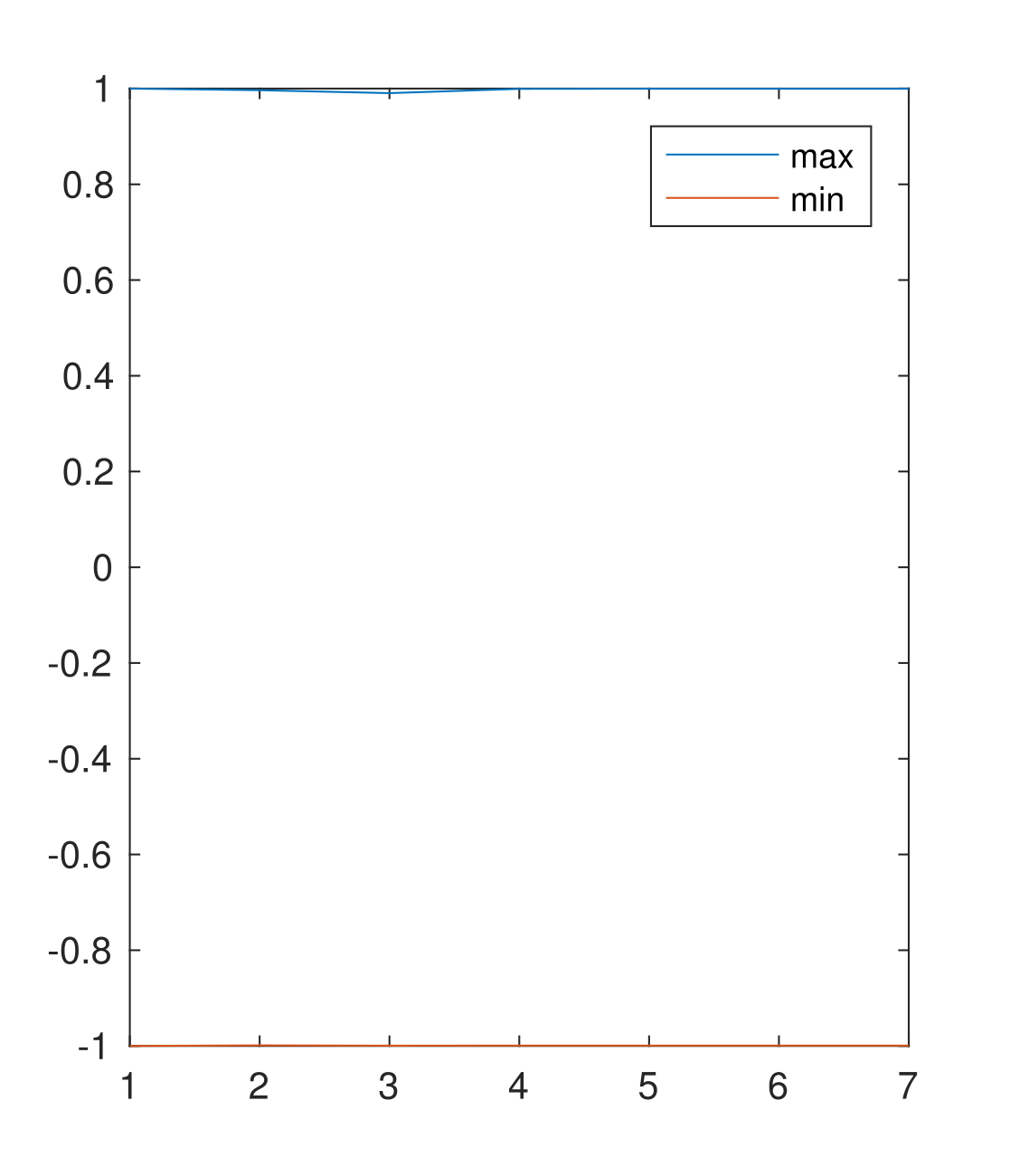}
\includegraphics[width=6cm,height=5cm]{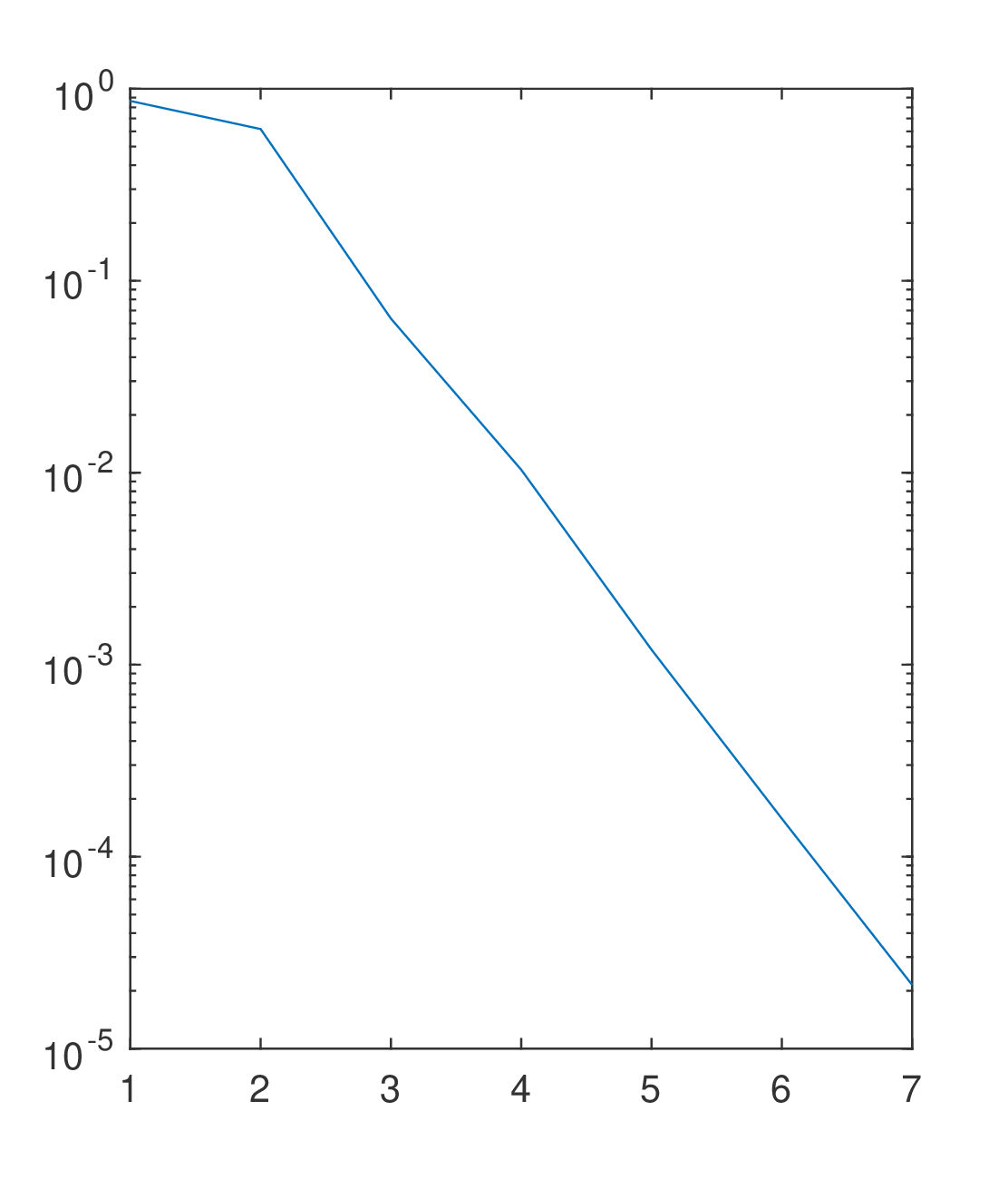}\\
\end{center}
\vskip -.5cm
\caption{House image. On lines 1 and 2: from original image to its segmentation (with RSS scheme) $\Delta t =1.e-5$, $\epsilon=0.01$, $\tau=1$,  $\lambda=10^{10}$.
On line 3: min and max values vs time (left) and $L^2$ norm of the discrete time derivative of the solution vs time.}
\label{SEG1}
\end{figure}
\begin{figure}[htbp!]
\vskip -1.5cm
\begin{center}
\includegraphics[height=6cm]{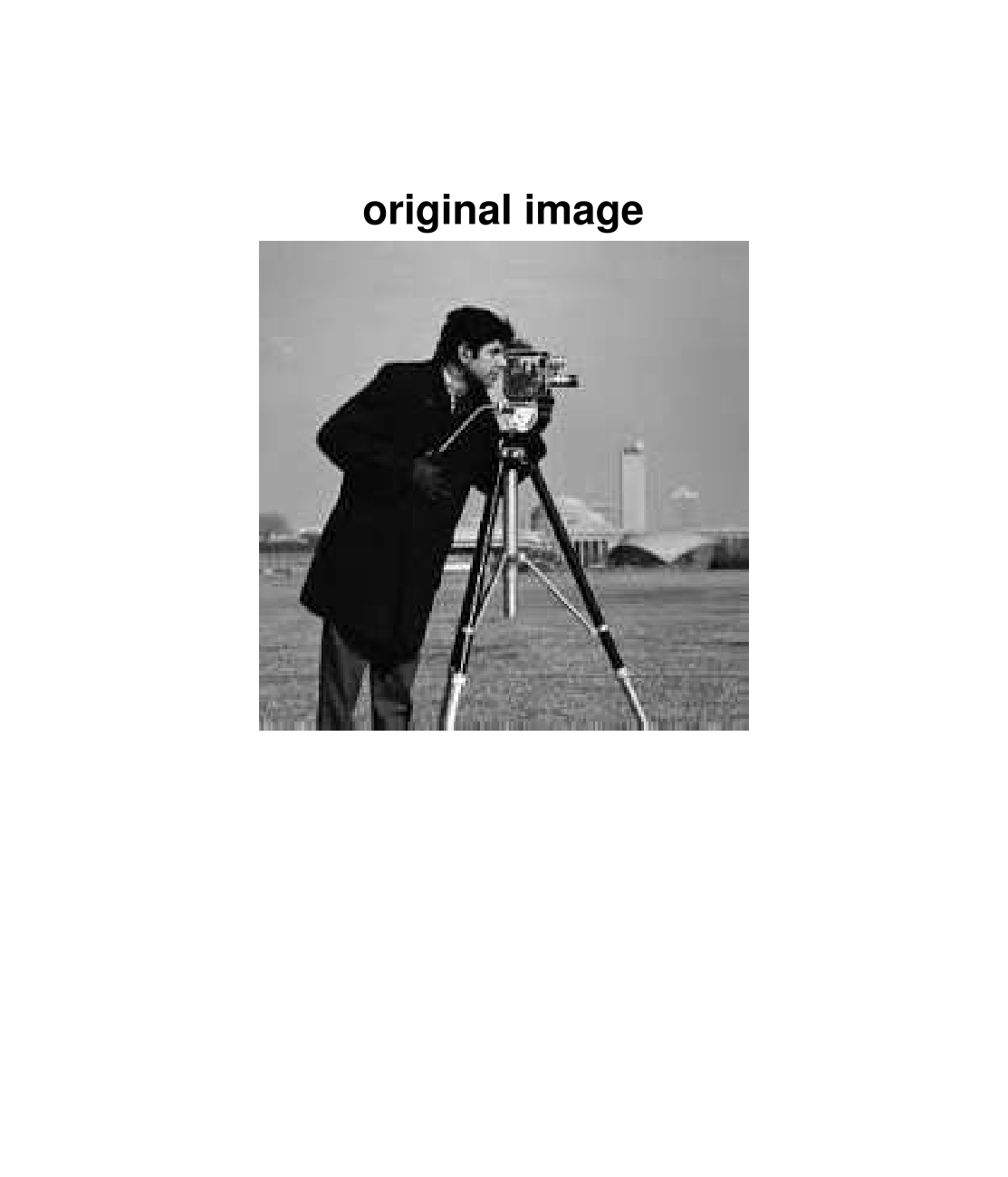}
\includegraphics[height=6cm]{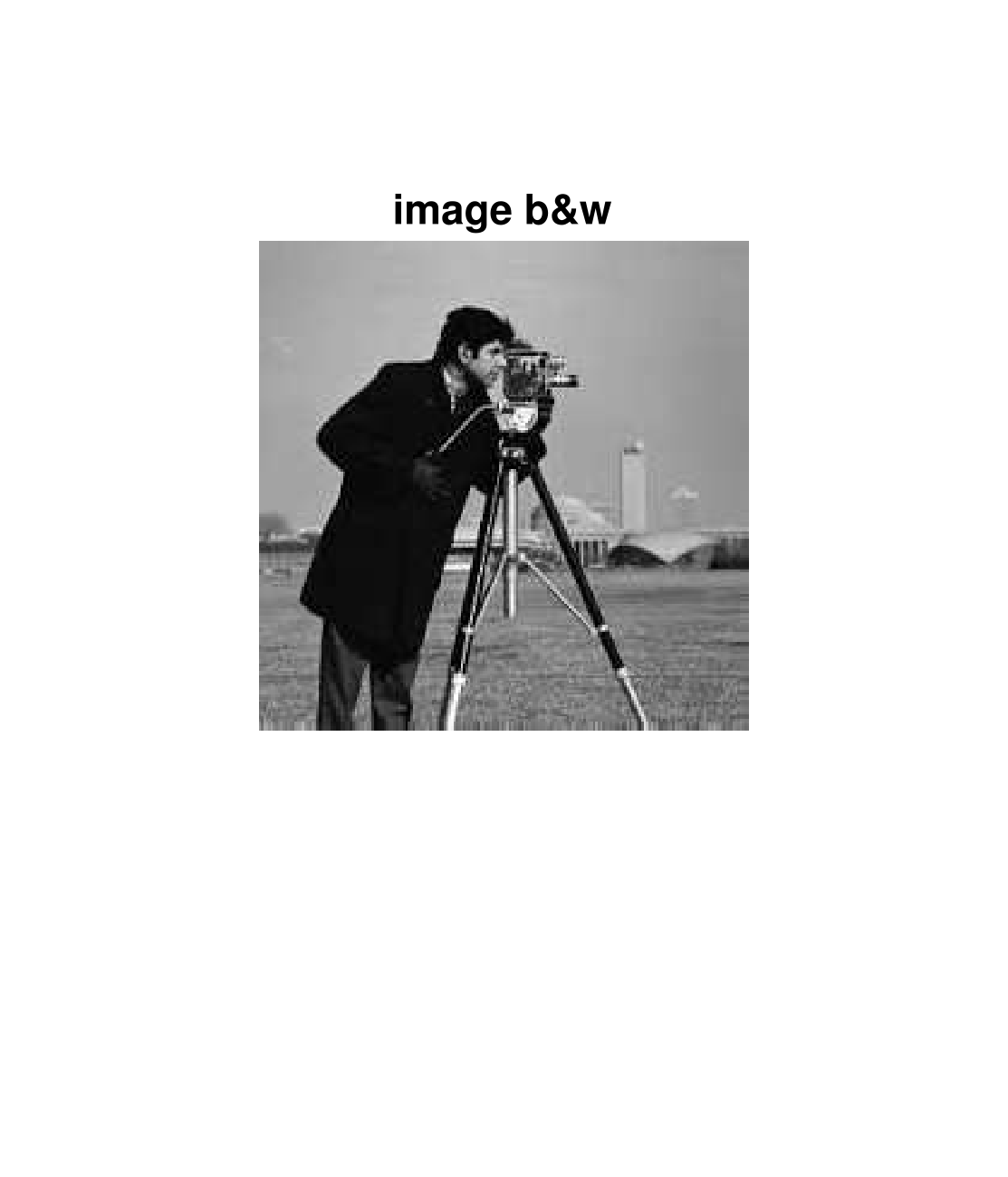}\\
\includegraphics[height=6cm]{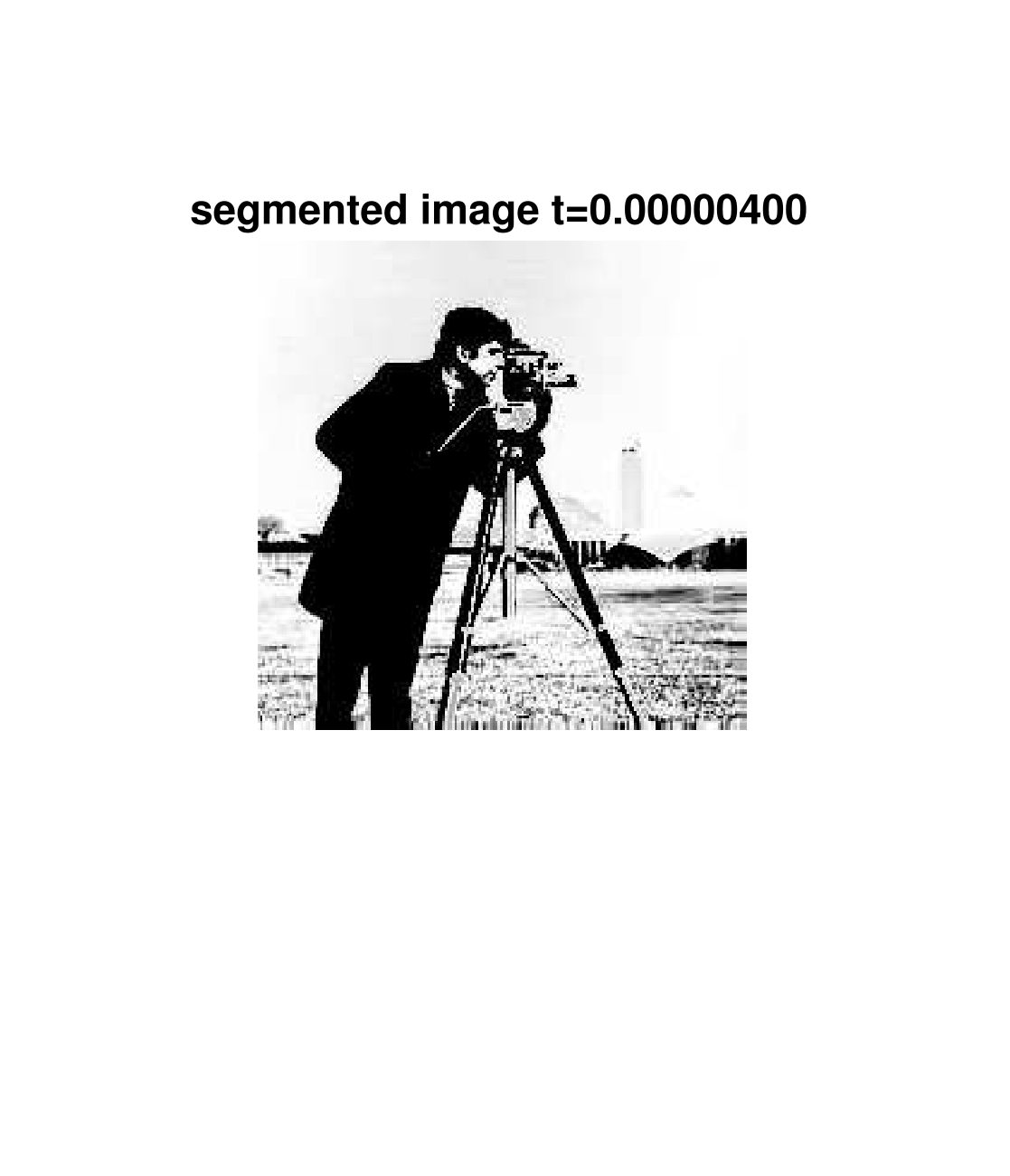}
\includegraphics[height=6cm]{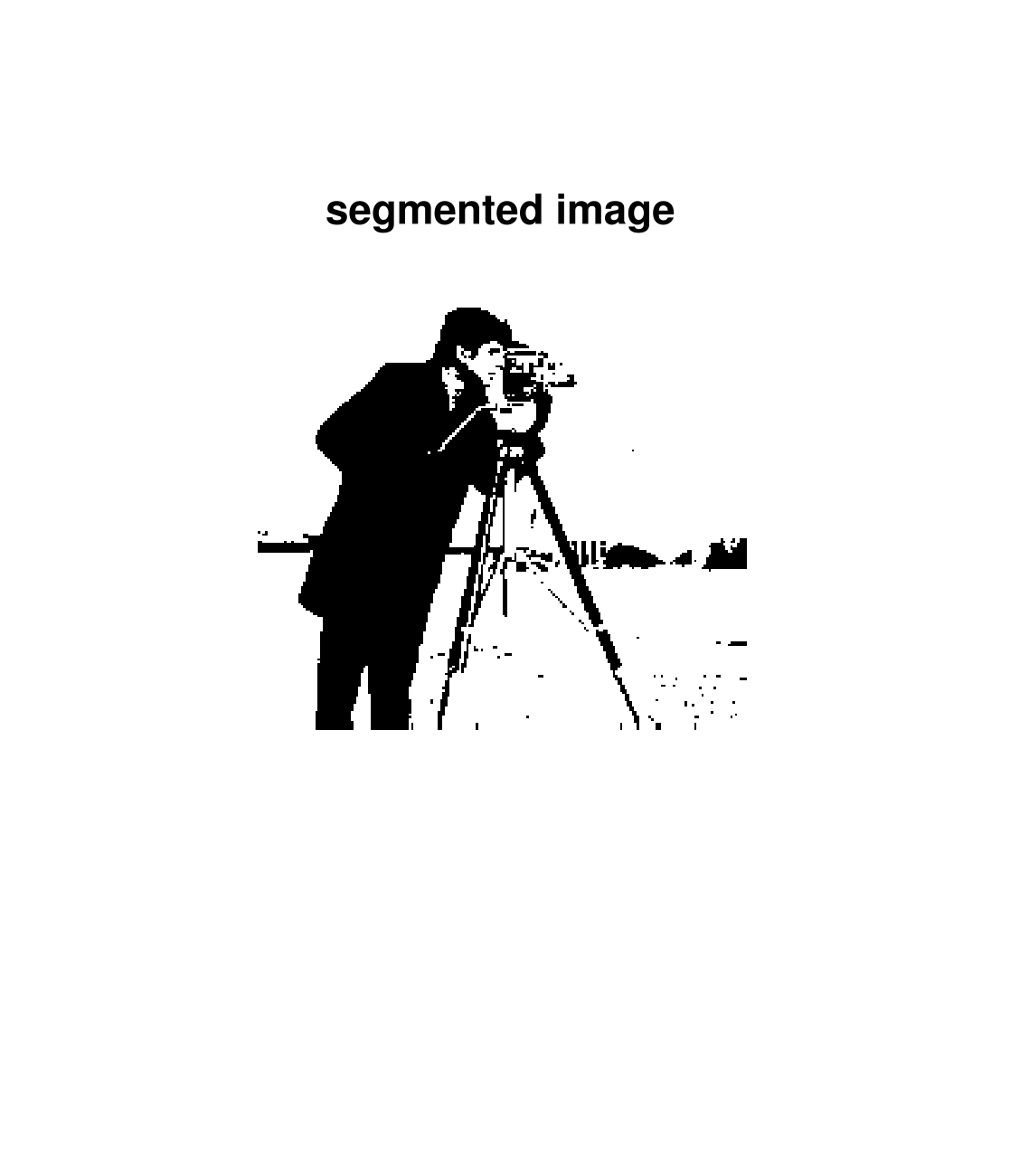}\\
\includegraphics[width=6cm,height=5cm]{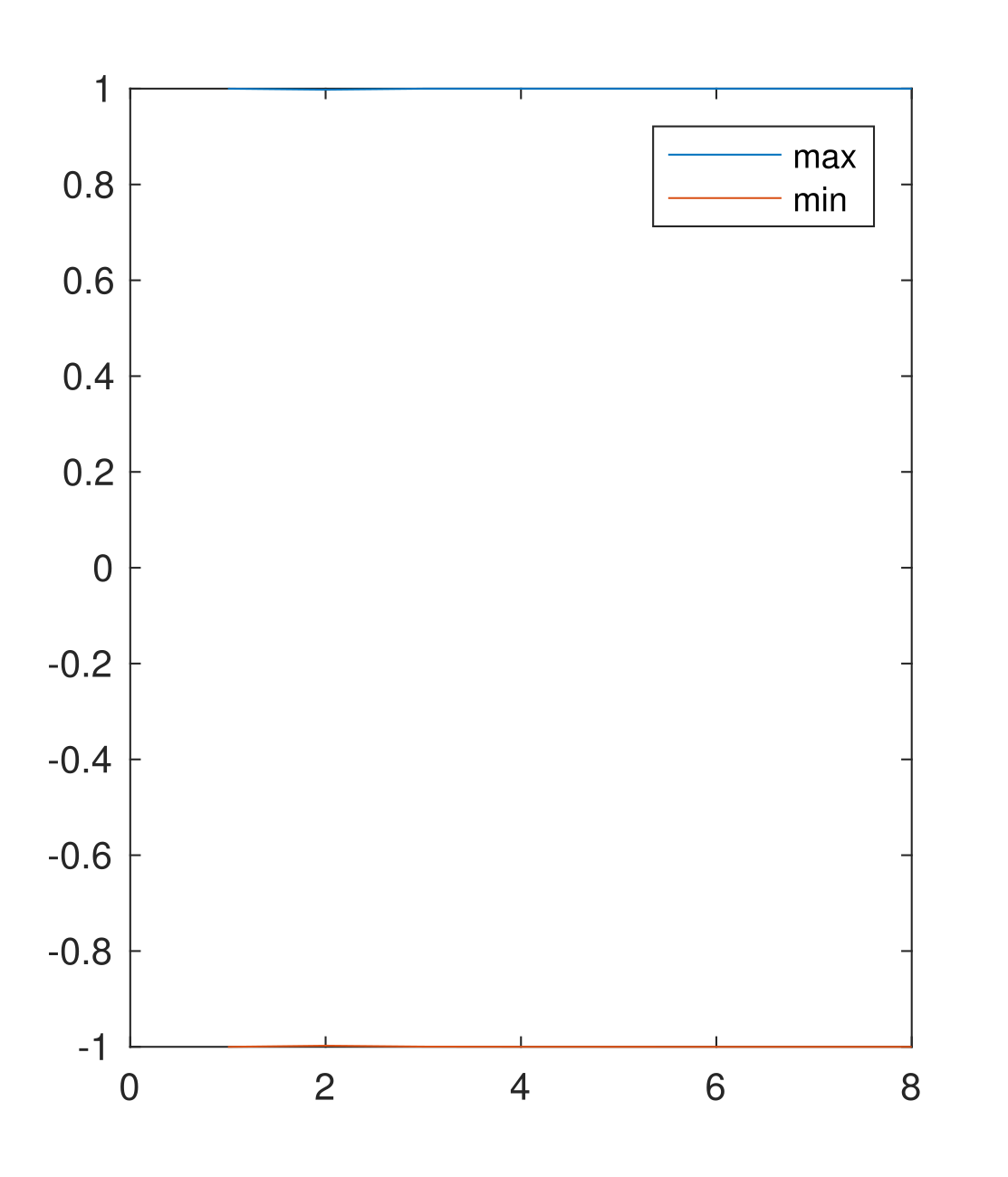}
\includegraphics[width=6cm,height=5cm]{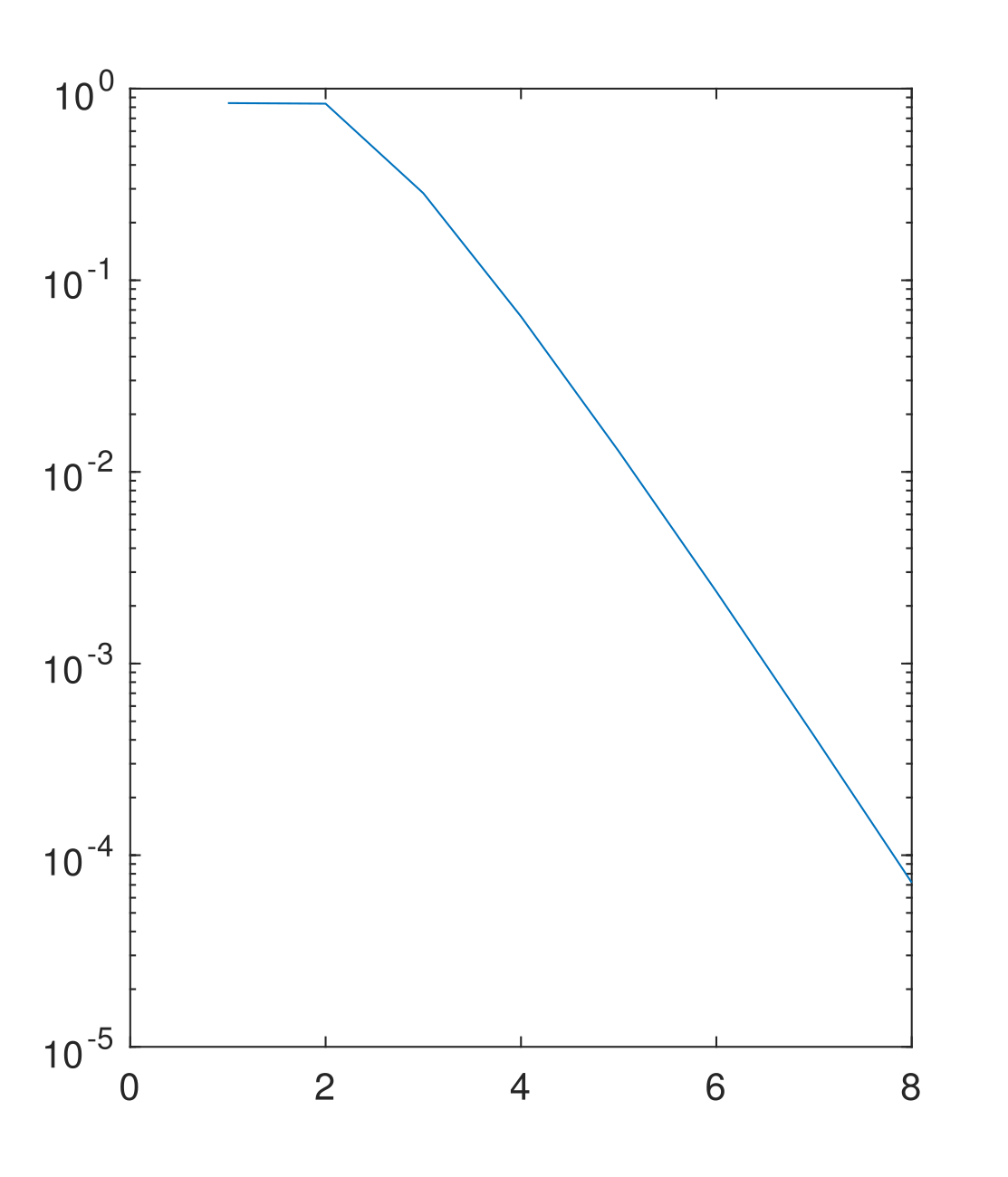}\\
\end{center}
\vskip -.5cm
\caption{Cameraman image. On lines 1 and 2: from original image to its segmentation (with RSS scheme) $\Delta t =5.e-7$, $\epsilon=0.04$, $\tau=1$,  $\lambda=10^{10}$.
On line 3: min and max values vs time (left) and $L^2$ norm of the discrete time derivative of the solution vs time.}
\label{SEG2}
\end{figure}
%
%
\clearpage
\subsection{Cahn-Hilliard equation}
\subsubsection{2D pattern Dynamics}
In the results presented below, we used CS2 Compact Schemes discretization in space and the 2D then the 3D cosine FFT for a fast solution of the linear systems arising in RSS-schemes.\\
\\
\begin{figure}[htbp!]
\label{fig: CH 2D} 
\begin{center}
\includegraphics[height=4cm]{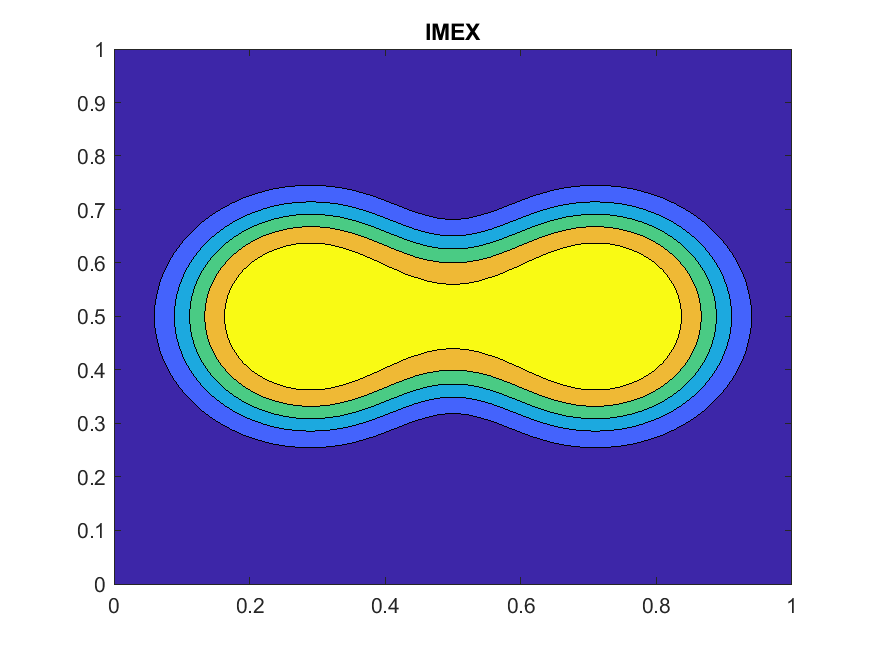}
\includegraphics[height=4cm]{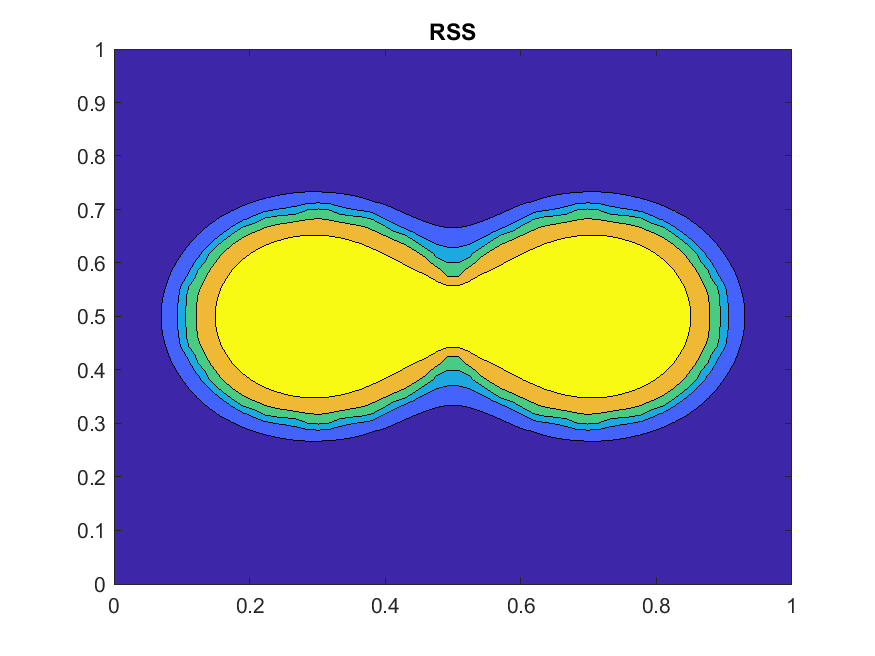}
\includegraphics[height=4cm]{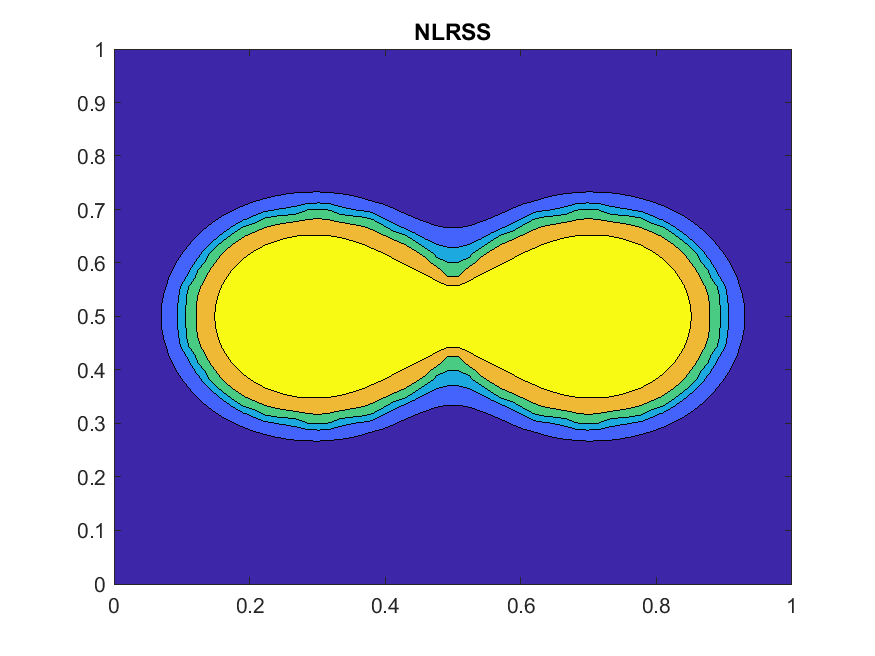}\\
\includegraphics[height=4cm]{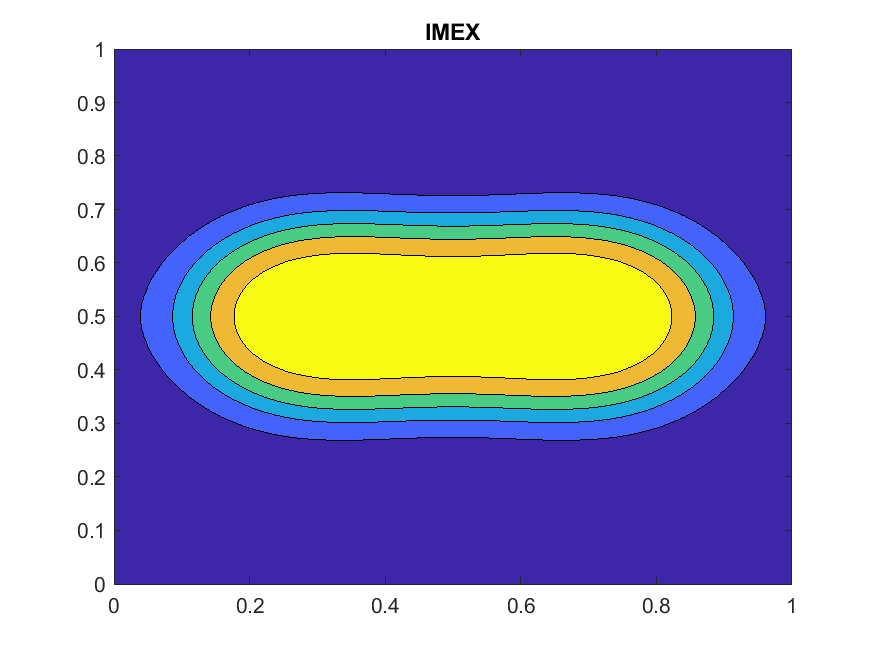}
\includegraphics[height=4cm]{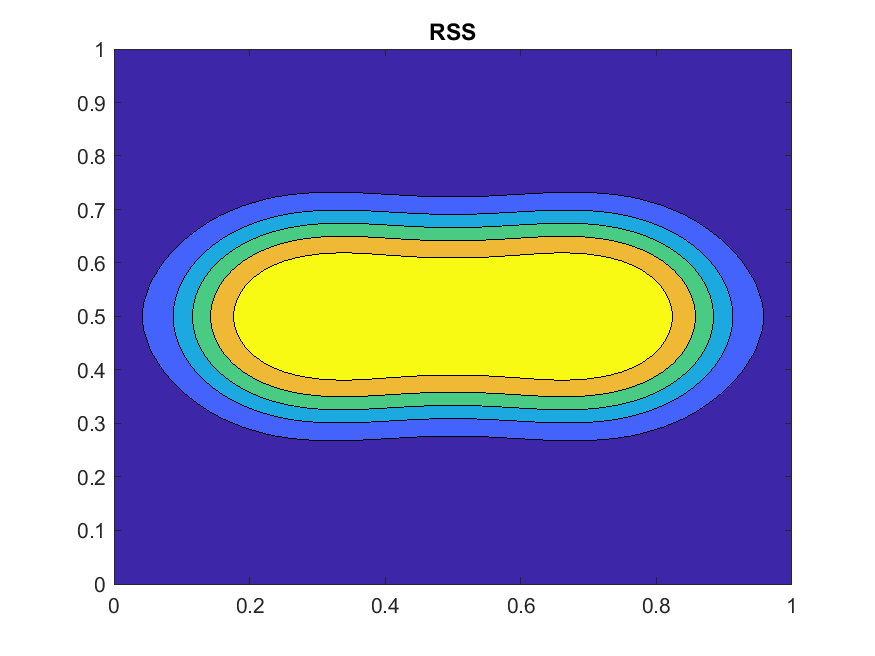}
\includegraphics[height=4cm]{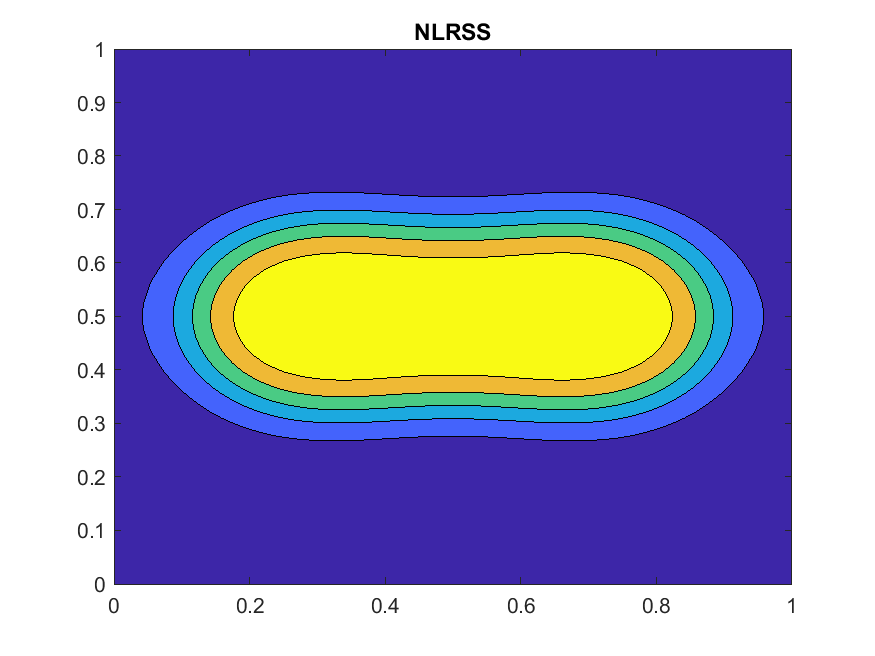}\\
\includegraphics[height=4cm]{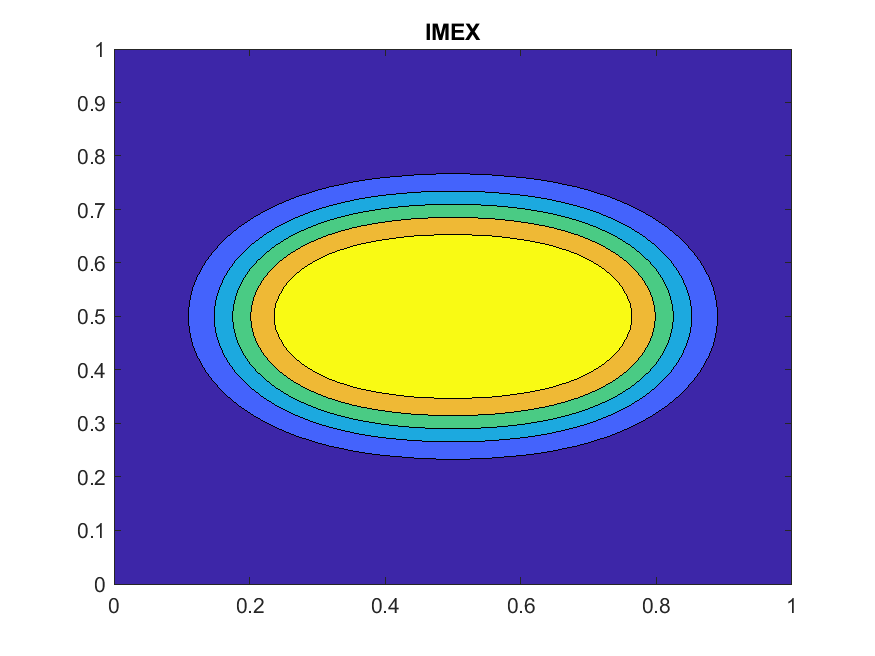}
\includegraphics[height=4cm]{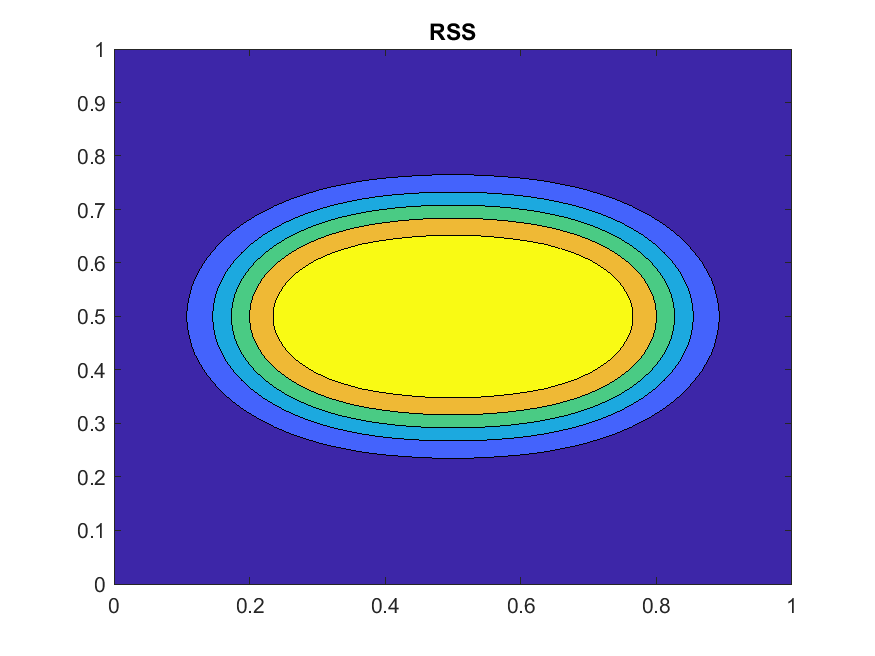}
\includegraphics[height=4cm]{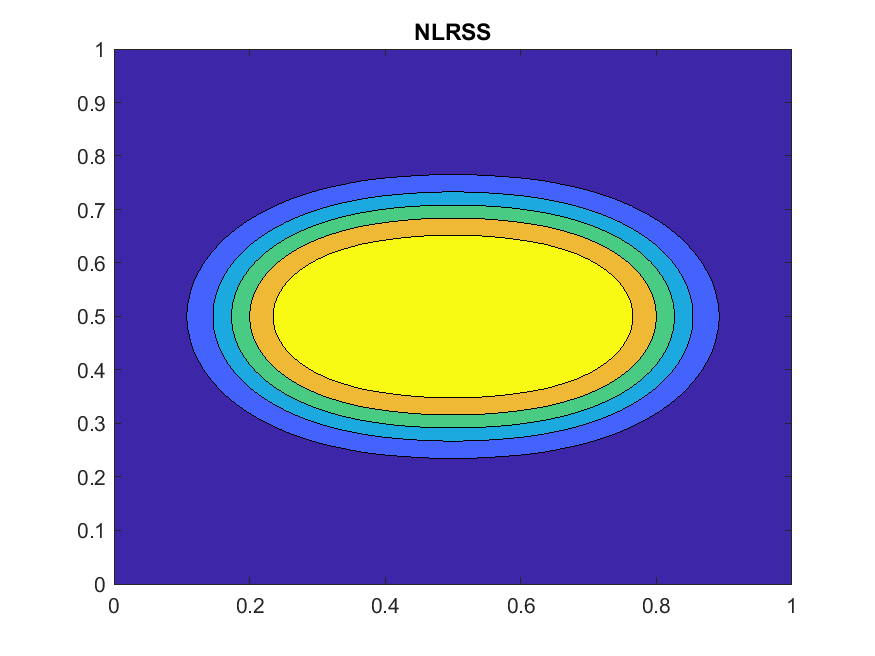}\\
\end{center}
\caption{Solution of the 2D Cahn-Hilliard equation with different time schemes. The initial condition is given by two circles. Line by line, the numerical solution are at time $t=10^{-4}$, $t=10^{-3}$ and $t=5 \cdot 10^{-3}$. The parameters are $\varepsilon = 0.05$, $N=64$, $\Delta t = 10^{-5}$ and $\tau = 4$.}
\end{figure}
\\
\begin{figure}
\label{fig: CH 2D ENER}
\begin{center}
\includegraphics[height=5cm]{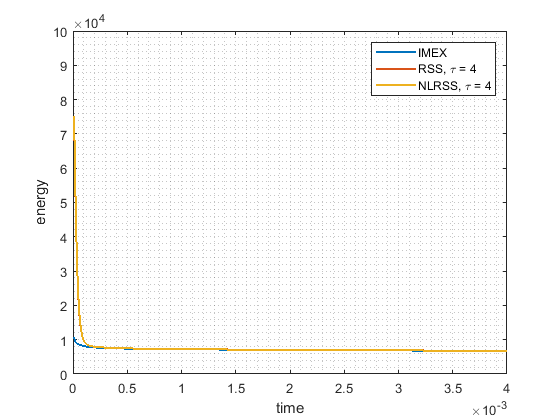}
\includegraphics[height=5cm]{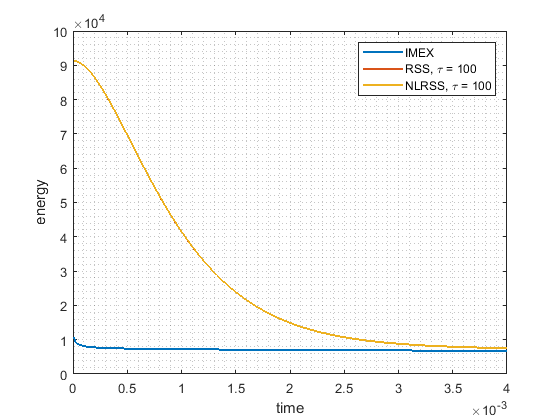}
\end{center}
\caption{History of the numerical energy for the 2D Cahn-Hilliard equation with different time schemes. The initial condition is given by two circles. Line by line, the numerical solution are at time $t=10^{-4}$, $t=10^{-3}$ and $t=5 \cdot 10^{-3}$. The parameters are $\varepsilon = 0.05$, $N=64$, $\Delta t = 10^{-5}$ and $\tau = 4$.}
\end{figure}
\\
\subsubsection{3D Pattern Dynamics}
As an illustration, we consider pattern dynamics problem in 3D:
\begin{eqnarray}
\Frac{\partial u}{\partial t} -\Delta( -\epsilon \Delta u +\Frac{1}{\epsilon}f(u))=0, &x \in \Omega=]0,1[^3,\\
\Frac{\partial u}{\partial n}=0,\Frac{\partial }{\partial n}\left(\Delta u-\Frac{1}{\epsilon^2}f(u)\right)=0,\\
u(0,x)=u_0(x).&
\end{eqnarray}
We now compare the 3 schemes IMEX, NLRSS and IMEX-RSS for a given deterministic initial datum  $u_0(x,y,z) = \cos ( \pi x ) \cos ( 2 \pi y ) \cos ( 3 \pi z)$ : as shown in Figure \ref{fig: CH 3D}, RSS-IMEX capture the pattern dynamics with a decreasing energy and a conserved null mean value of the solution. In Figure \ref{fig: CH 3D ENER} we compare the time evolution of the energy for the 3 schemes when
taking $\tau=4$, then $\tau=100$.
\\
\begin{figure}[htbp!]
\label{fig: CH 3D} 
\begin{center}
\includegraphics[height=4cm]{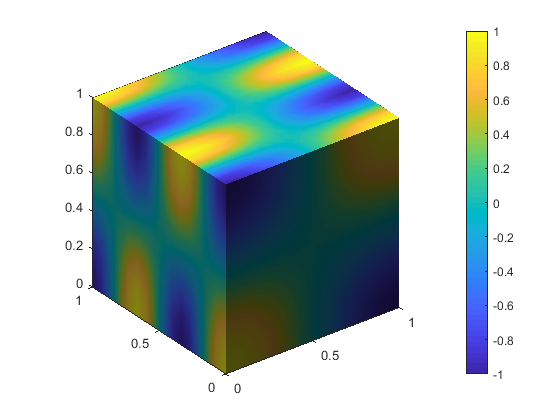}\\
\includegraphics[height=4cm]{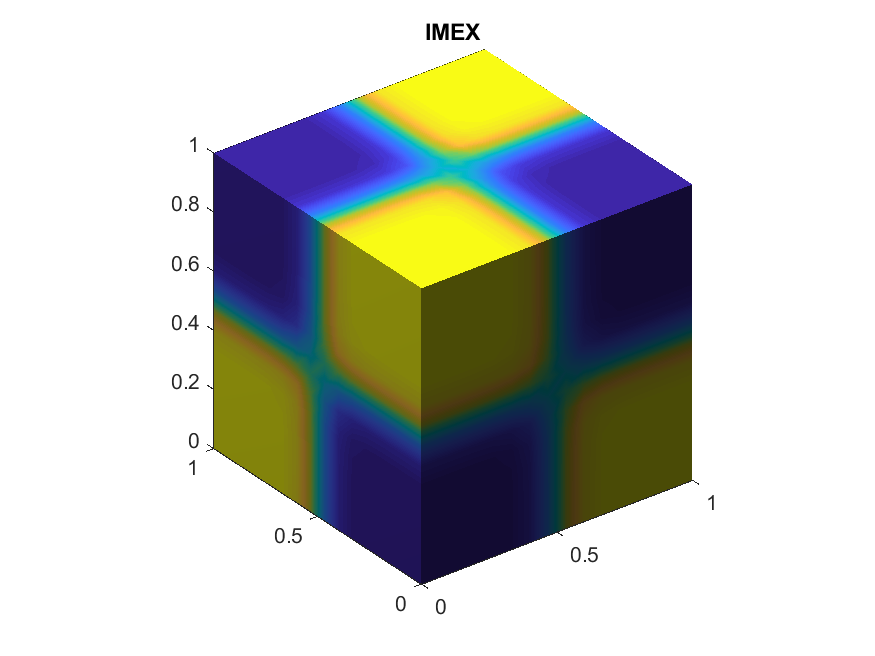}
\includegraphics[height=4cm]{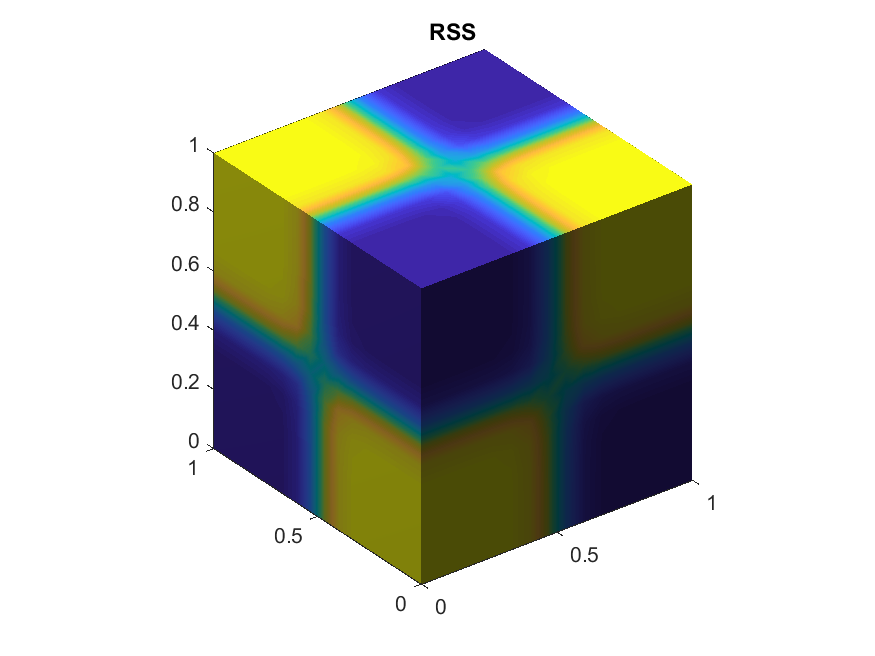}
\includegraphics[height=4cm]{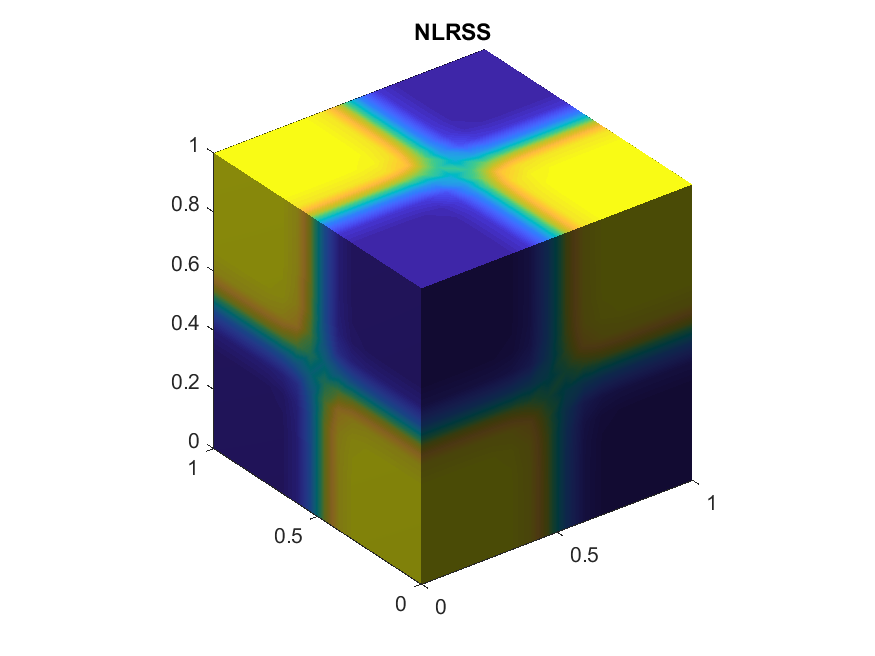}\\
\includegraphics[height=4cm]{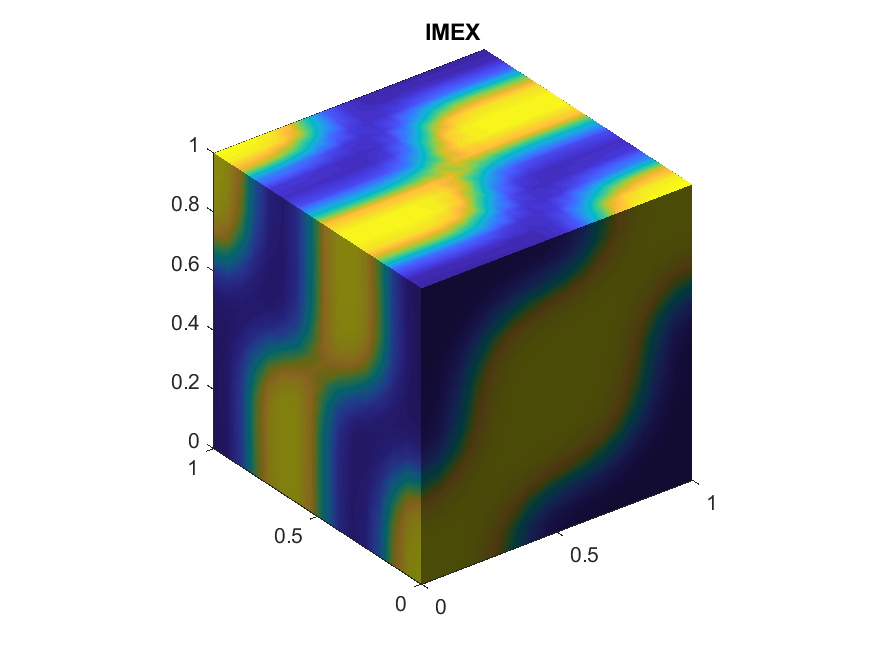}
\includegraphics[height=4cm]{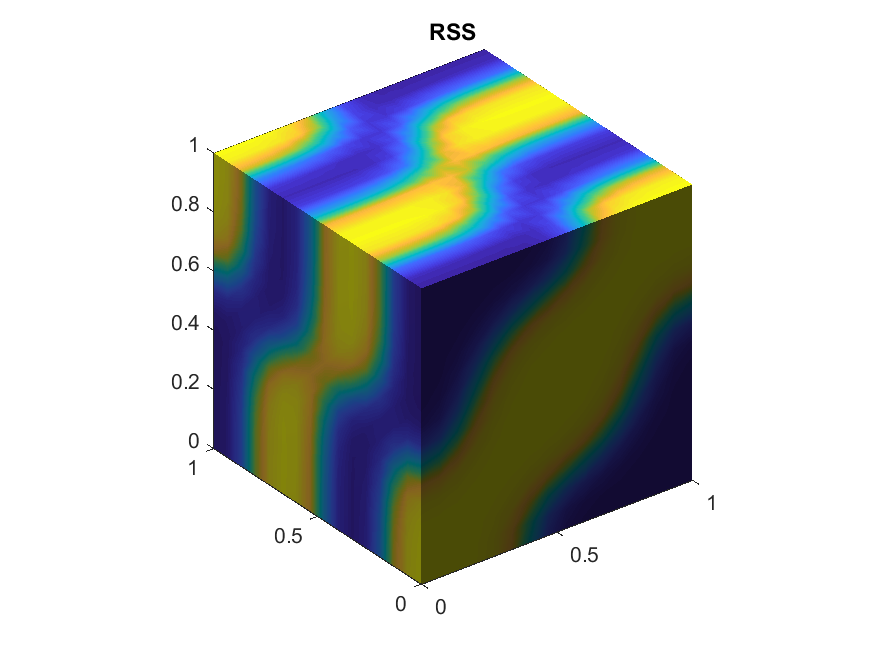}
\includegraphics[height=4cm]{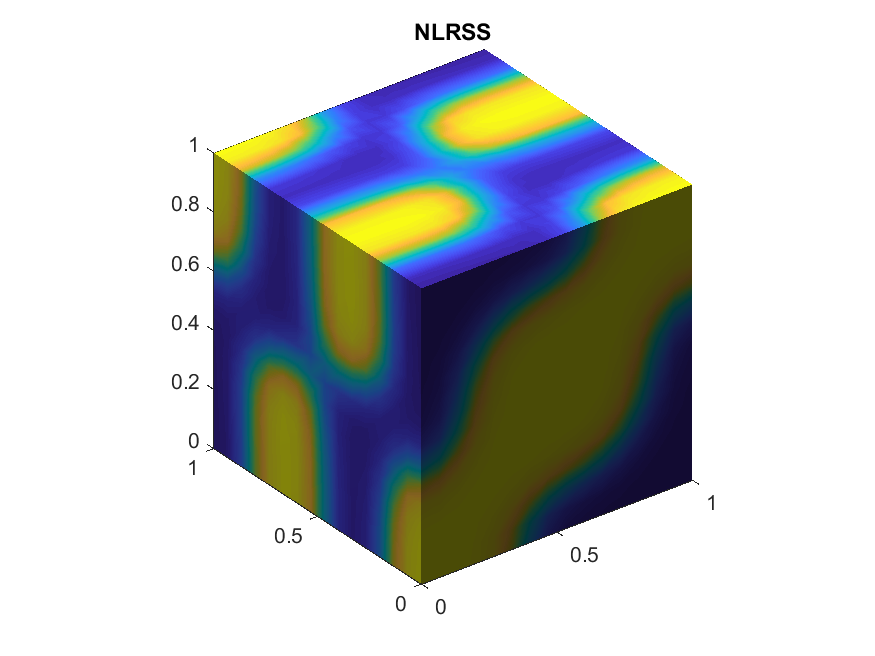}\\
\includegraphics[height=4cm]{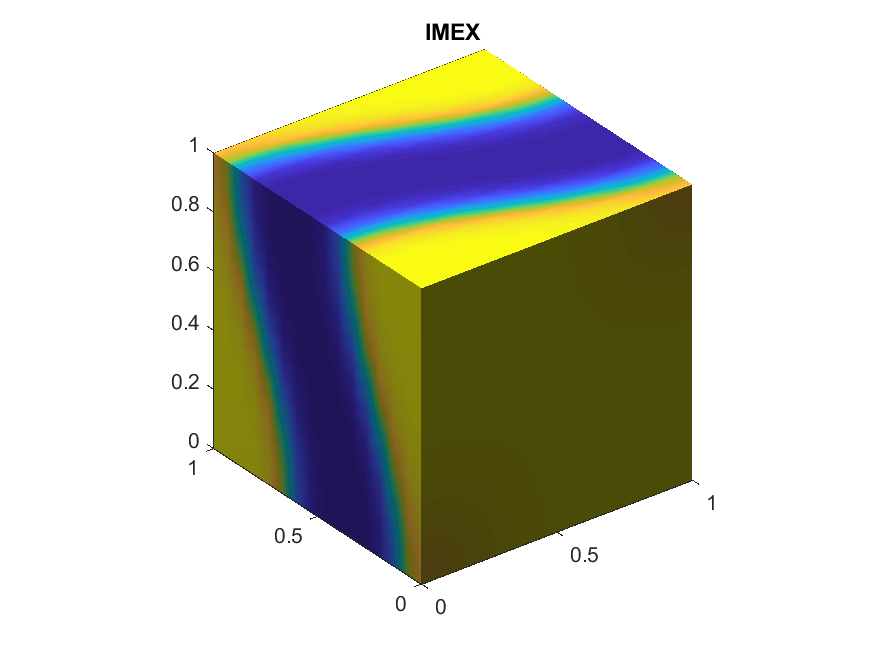}
\includegraphics[height=4cm]{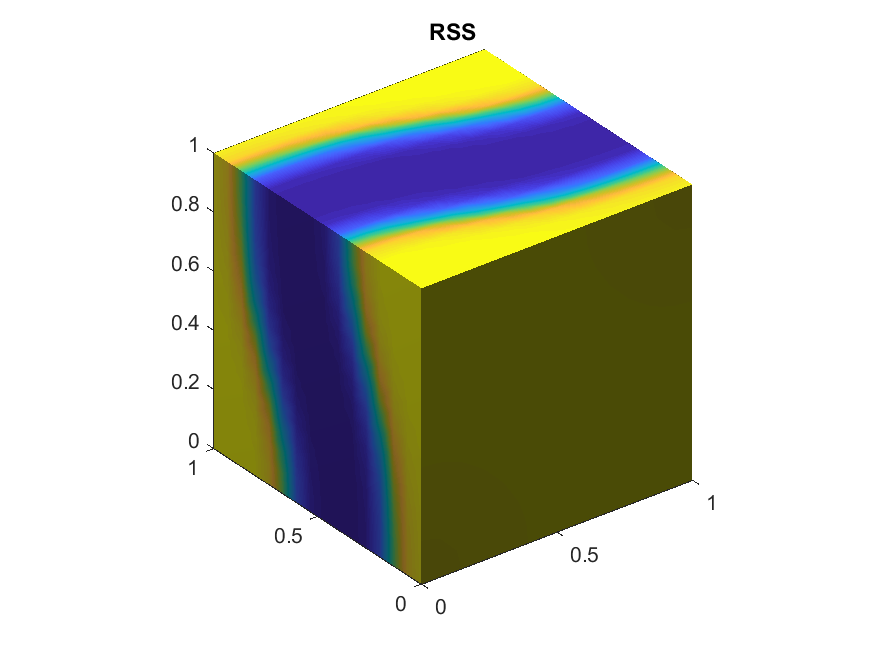}
\includegraphics[height=4cm]{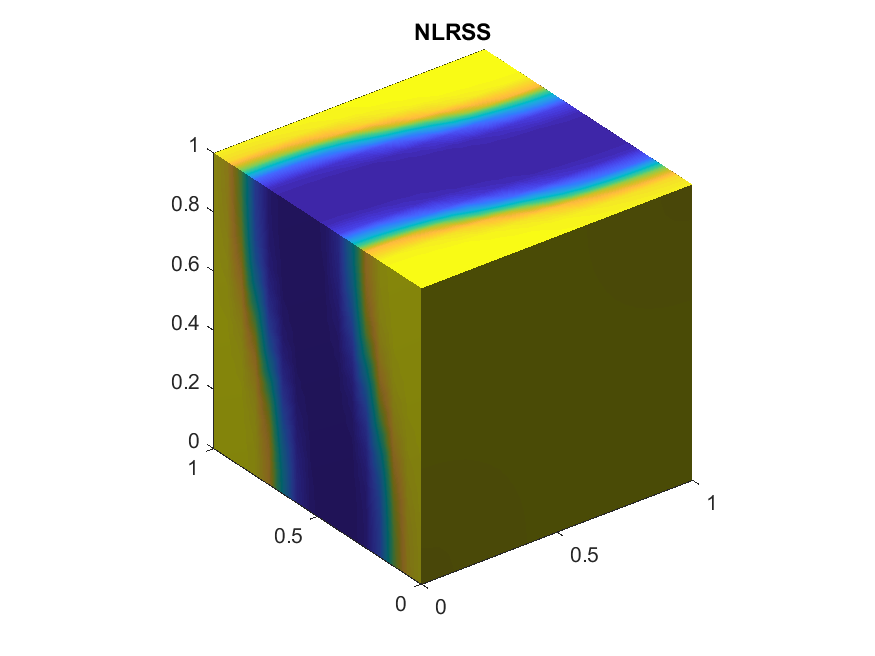}\\
\includegraphics[height=4cm]{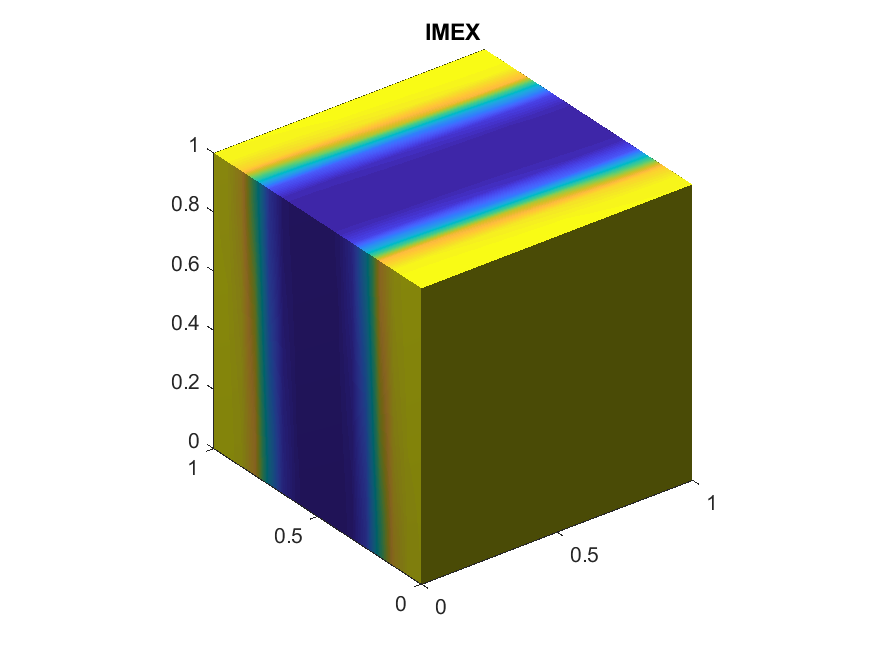}
\includegraphics[height=4cm]{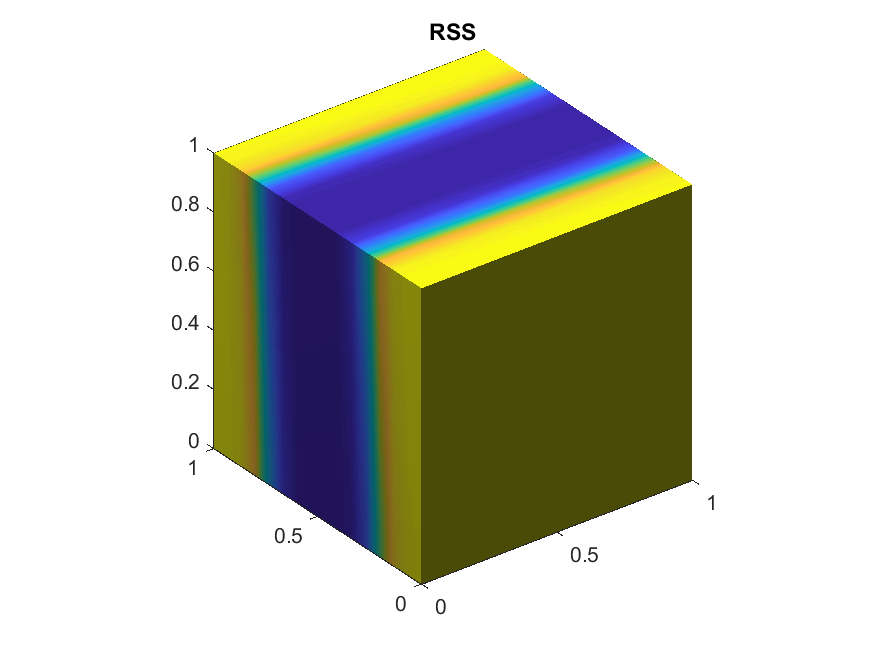}
\includegraphics[height=4cm]{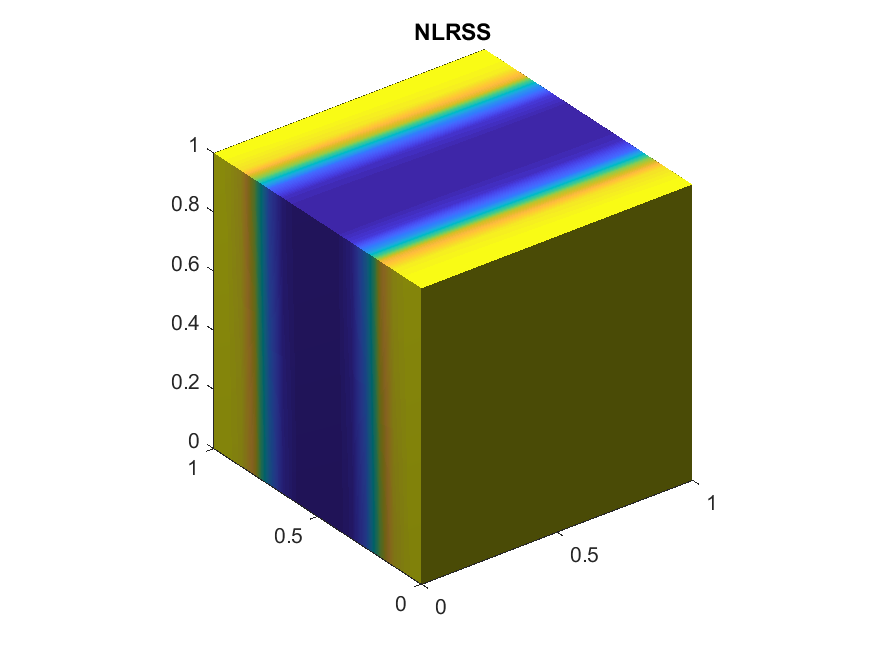}\\
\end{center}
\caption{Solution of the 3D Cahn-Hilliard equation with different time schemes. The initial condition is given by $u_0(x,y) = \cos ( 2\pi x ) \cos ( 2 \pi y )  \cos ( \pi z)$. Line by line, the numerical solution are at time $t=0$, $t=0.03$, $t=0.05$, $t=0.07$ and $t=0.1$. The parameters are $\varepsilon = 0.05$, $N=16$, $\Delta t = 10^{-4}$ and $\tau = 4$.}
\end{figure}
\\
\begin{figure}[htbp!]
\label{fig: CH 3D ENER}
\begin{center}
\includegraphics[height=5cm]{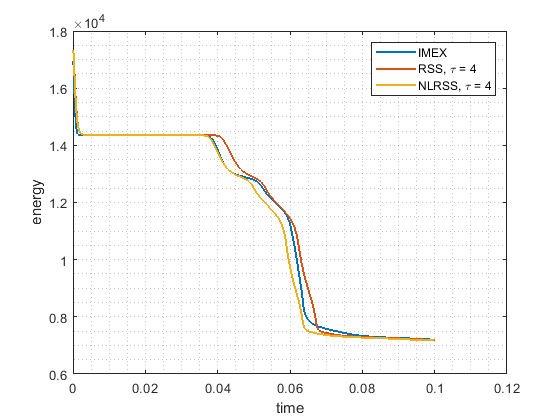}
\includegraphics[height=5cm]{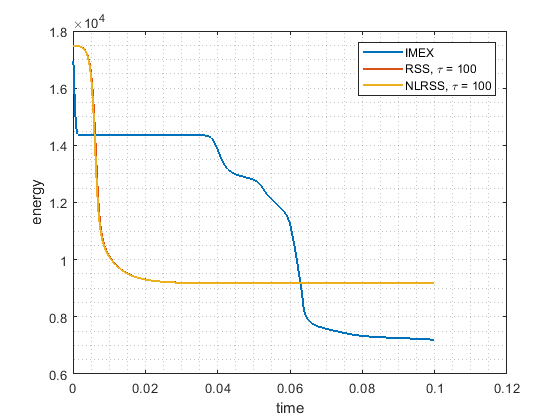}
\end{center}
\caption{History of the numerical energy for the 3D Cahn-Hilliard equation with different time schemes. The initial condition is $u_0(x,y) = \cos ( 2\pi x ) \cos ( 2 \pi y ) \cos ( \pi z )$. The parameters are $\varepsilon = 0.05$, $N=16$ and $\Delta t = 10^{-4}$.}
\end{figure}

\subsubsection{2D inpainting}
We consider here the inpainting problem (see e.g. \cite{Bertozzi1,Bertozzi2,CherfilsFakihMiranville1} )
\begin{eqnarray}
\Frac{\partial u}{\partial t} -\Delta( -\epsilon \Delta u +\Frac{1}{\epsilon}f(u))&+\lambda\chi_{\Omega\setminus D}(x)(u-g)=0,\\
\Frac{\partial u}{\partial n}=0&\Frac{\partial }{\partial n}\left(\Delta u-\Frac{1}{\epsilon^2}f(u)\right)=0,\\
u(0,x)=u_0(x),&
\end{eqnarray}
described in the previous section. Here $\Omega=]0,1[^{2}$, the schemes used is RSS-IMEX for inpainting (Algorithm \ref{CH_INPAINTING_RSS}).
We proceed following the approach described in Section 4.1 and used in \cite{CherfilsFakihMiranville1}:
\begin{itemize}
\item At first, for fixed $\epsilon>0$, we compute the  solution up to a converged time $t^*$ (here $t^*=6 \times 10^{-3}$).
\item Then, we apply a post-processing consisting in a thresholding which replaces the dominant phase by 1 at every point
of $\Omega$ and the other phases (colors) by 0. The final result exhibit sharp contrasts, as we can see in Figures \ref{fig: CH 2D INP CIRC} and \ref{fig: CH 2D INP TRIG}.
\end{itemize}
 We here used CS2 Compact Schemes discretization in space. However the implicit par in the RSS scheme is not particularly adapted to the use of cosine FFT: this is to the presence of the linear fidelity operator. The linear systems are solved using the {\it backslash} command $\backslash$.
\begin{figure}
\begin{center}
\includegraphics[height=5cm]{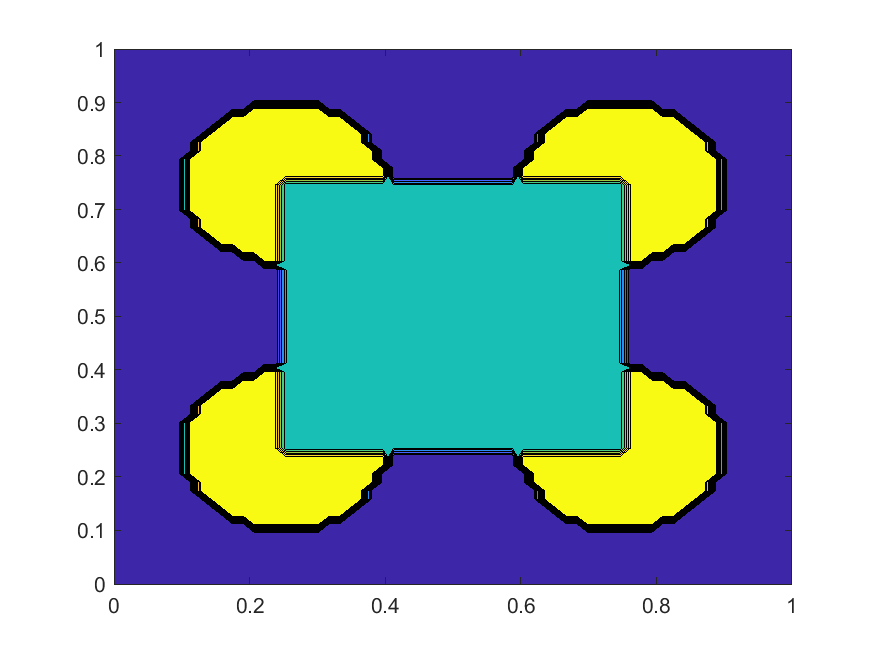}\\
\includegraphics[height=5cm]{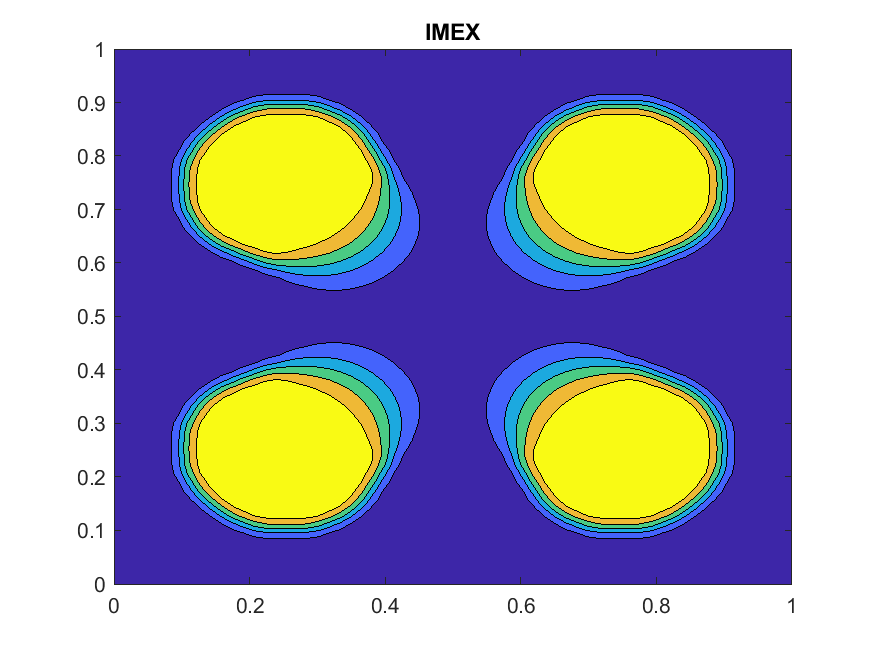}
\includegraphics[height=5cm]{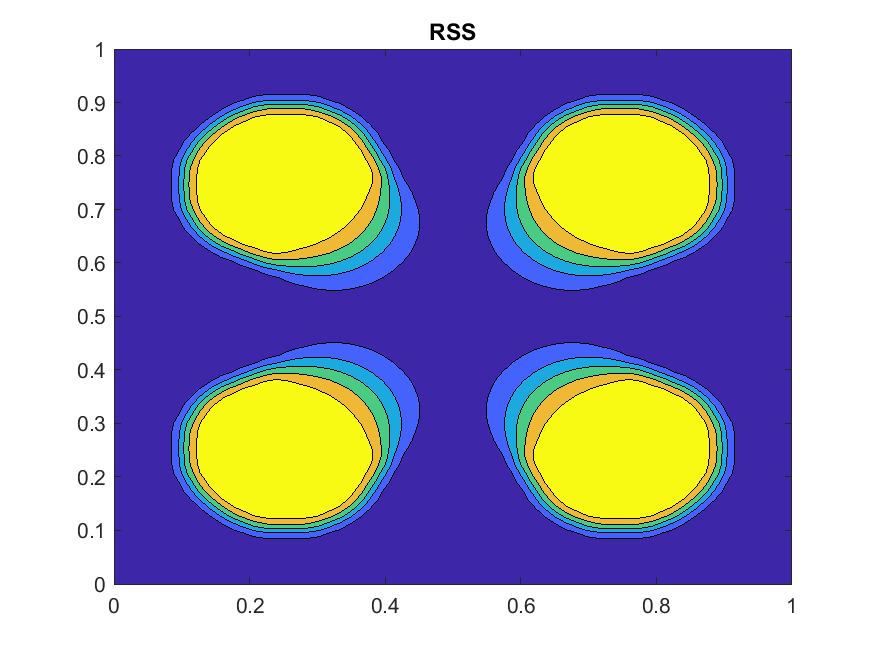}\\
\includegraphics[height=5cm]{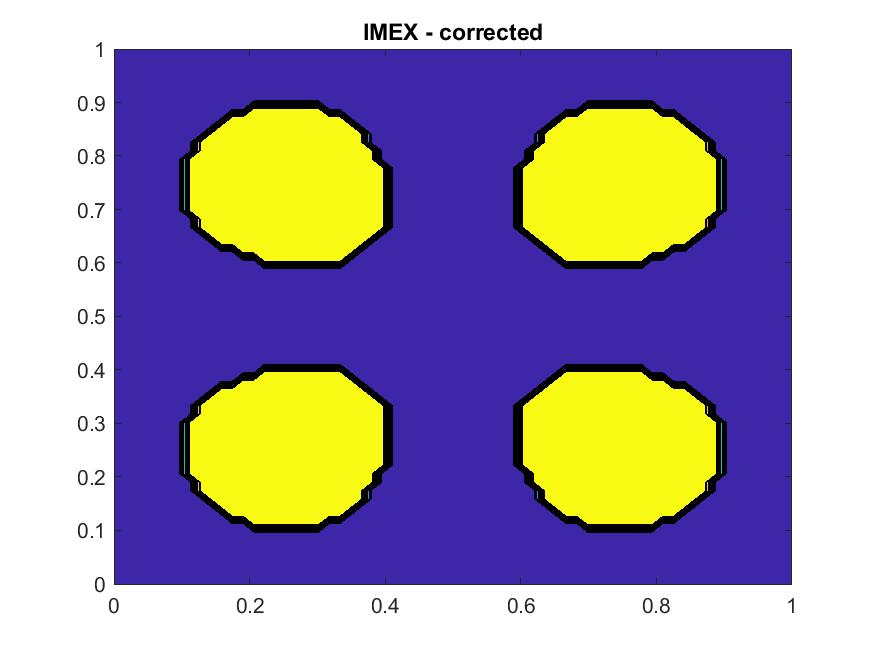}
\includegraphics[height=5cm]{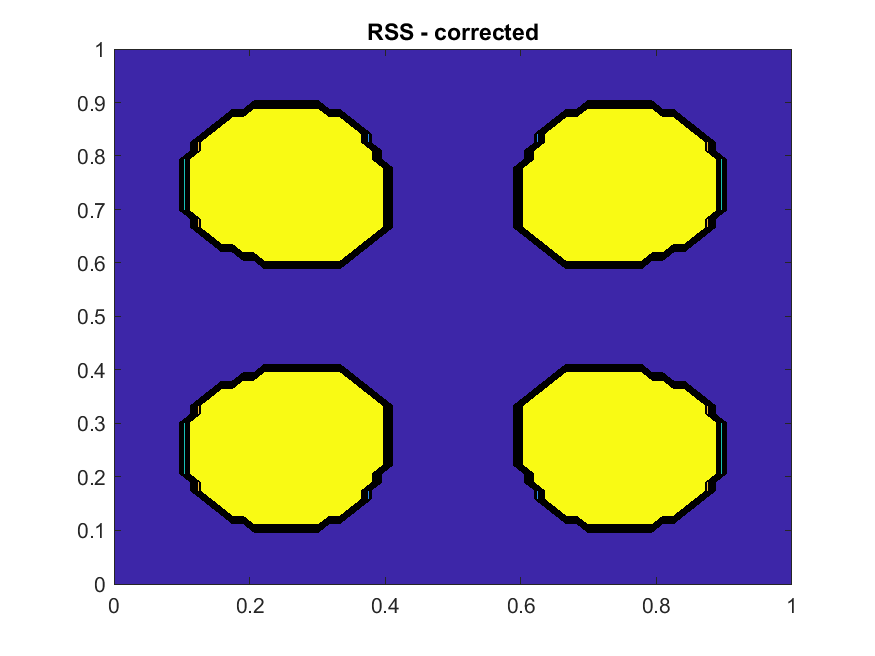}
\end{center}
\caption{Inpainting using the Cahn-Hilliard equation. The PDEs parameter are $\varepsilon = 0.05$ and $\lambda  = 900000$. The numerical parameters are $N=64$, $\Delta t = 10^{-6}$ and $\tau = 4$ (for RSS scheme). The final time is $t = 5 \cdot 10^{-3}$. Top : initial map. The blue square represents the inpainting area. Center line : solution with IMEX (left) and RSS (right) scheme at final time. Bottom line : solution at final time with IMEX (left) and RSS (right) scheme after correction.}
\label{fig: CH 2D INP CIRC}
\end{figure}
\\
\begin{figure}
\begin{center}
\includegraphics[height=5cm]{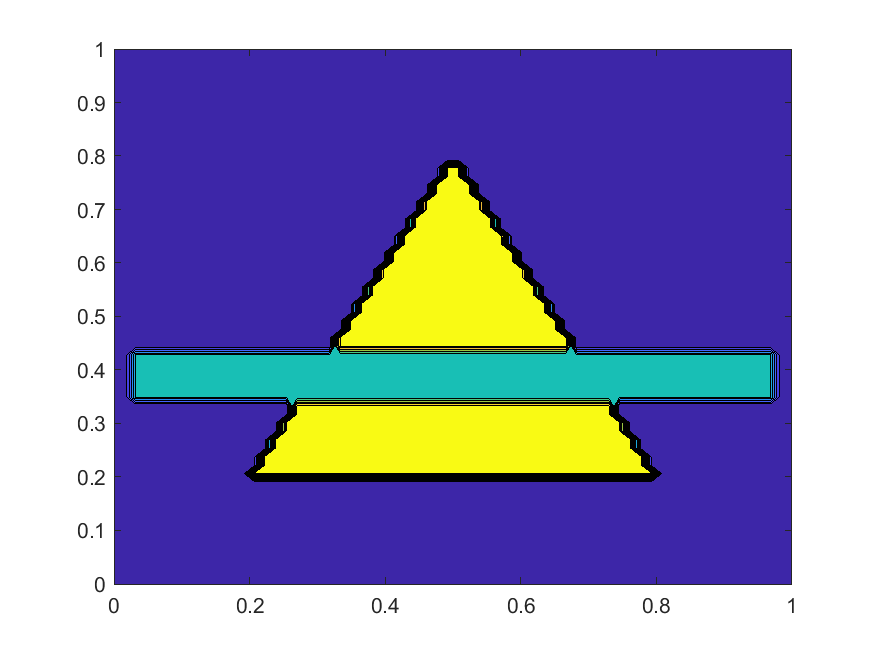}\\
\includegraphics[height=5cm]{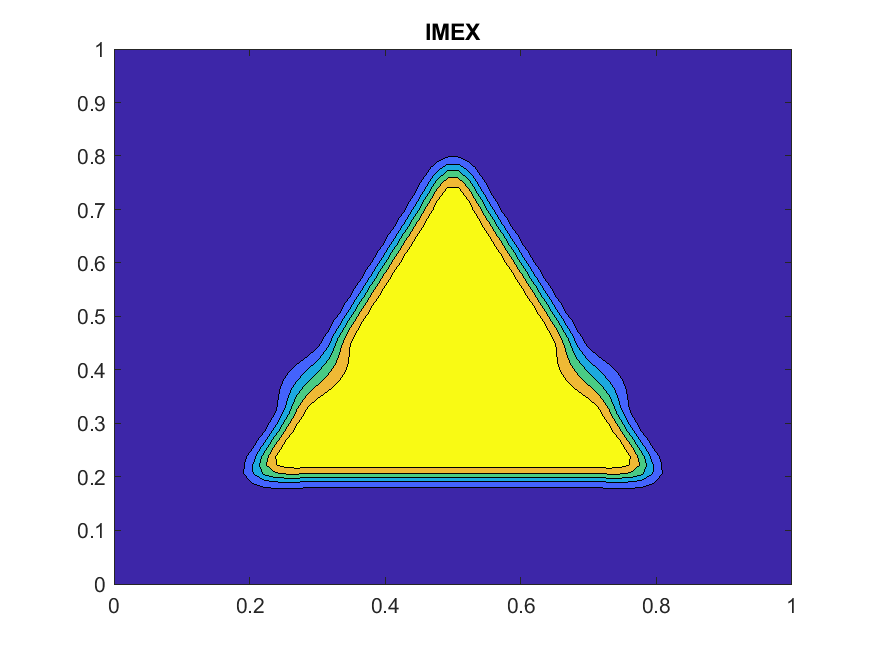}
\includegraphics[height=5cm]{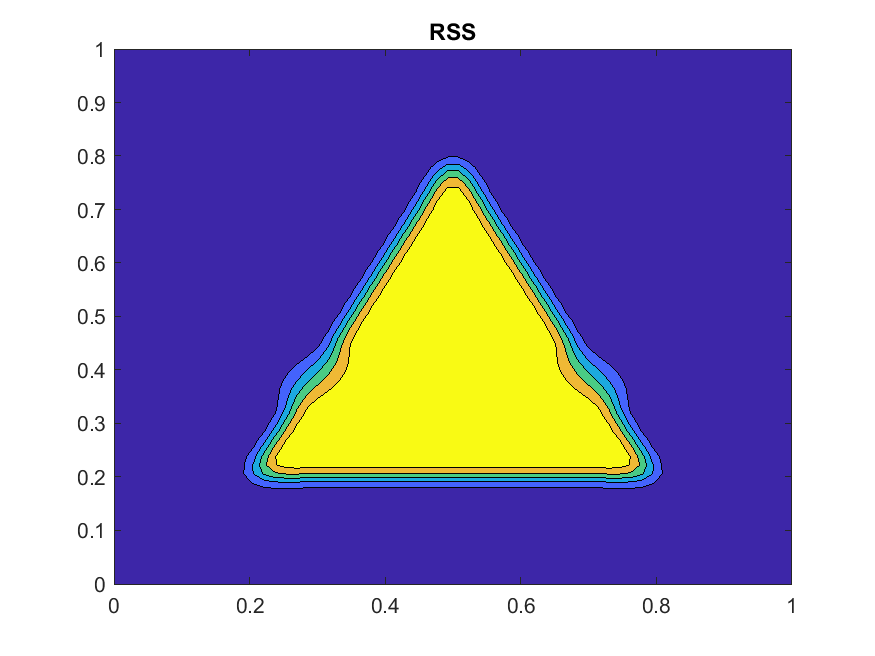}\\
\includegraphics[height=5cm]{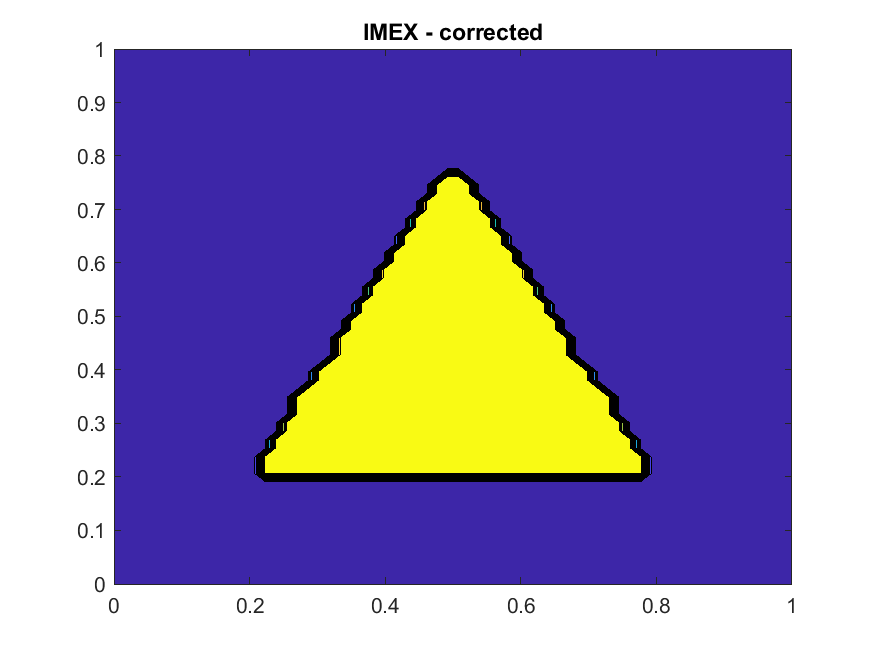}
\includegraphics[height=5cm]{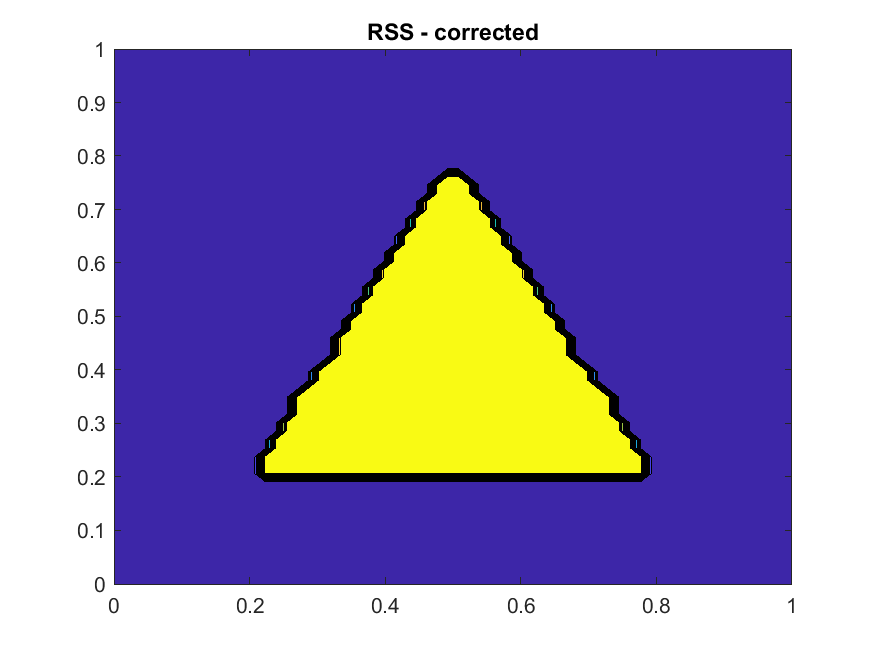}
\caption{Inpainting using the Cahn-Hilliard equation. The PDE€™s parameter are $\varepsilon = 0.05$ and $\lambda  = 900000$. The numerical parameters are $N=64$, $\Delta t = 10^{-6}$ and $\tau = 4$ (for RSS scheme). Results are given after 100 iterations. Top : initial map. The blue rectangle represents the inpainting area. Center line : solution with IMEX (left) and RSS (right) scheme at final time. Bottom line : solution at final time with IMEX (left) and RSS (right) scheme after correction.}
\label{fig: CH 2D INP TRIG}
\end{center}
\end{figure}
\\
The results agree with the ones presented in the inpainting problem (see e.g. \cite{CherfilsFakihMiranville1} ).
%
%
\clearpage
\section{Concluding remarks and perspectives}
We have introduced stabilized finite differences semi-implicit schemes that allow a fast simulation of high accurate solutions of phase fields problems since
the main effort of the computation lies on the efficient solution of sparse linear systems. These new time marching scheme allow to use computational facilities of the sparse linear algebra and also of the  fast solvers, depending on the situation. We have considered only finite differences but the use of two different levels of accuracy can be applied to other discretization techniques, such as finite elements.
Of course several questions have still  to be considered and we address to future work the following possible developments: 
\begin{itemize}
\item The stabilization procedure has been here applied to simple IMEX schemes, but it is versatile and can be considered together with other time marching  schemes.
Indeed, it can apply also to accelerate the diffusion parts of the numerical solution of gradient flows by SAV-like schemes, recently studied, eg, in \cite{JShenXuYang}.
These schemes are obtained by introducing the discretization of the auxiliary variable $s(t)=\sqrt{\int_{\Omega}F(u)dx+C_0}$ and adding the time-derivative of this last expression to avoid the implicitly; here $C_0$ is a positive constant chosen such that $F(u)+C_0>0, \forall u$.\\
The IMEX-scheme applied to the discretized Cahn-Hilliard system reads as
\begin{eqnarray}\label{DIS_SAV_CH}
\Frac{u^{(k+1)}-u^{(k)}}{\Delta t} +A\mu^{(k+1)} =0,\\
\mu^{(k+1)}=Au^{(k+1)} +\Frac{s^{(k+1)}}{\sqrt{Q_h(F(u^{(k)}))+C_0} }f(u^{(k)})=0,\\
\Frac{s^{(k+1)}-s^{(k)}}{\Delta t}=Q_h(\Frac{f(u^{(k)})}{2\sqrt{Q_h(F(u^{(k)}))+C_0}}\Frac{u^{(k+1)}-u^{(k)}}{\Delta t})
\end{eqnarray}
Here, the expression $Q_h(v)$ corresponds to a quadrature formula applied to $v$: $Q_h(v)\simeq \int_{\Omega} vdx$. 
This scheme is unconditionally stable for the modified energy
$$
E_{SAV}(u,s)=\Frac{1}{2}Q_h(<Au,u>)+s^2,
$$
that means that $E_{SAV}(u^{(k+1)},s^{(k+1)})\le E_{SAV}(u^{(k)},s^{(k)}), \forall k\ge0, \ \forall \Delta t >0$, see \cite{JShenXuYang}.
The derivation of the stabilized IMEX-SAV scheme simply writes as
\begin{eqnarray}\label{RSS_SAV_CH1}
\Frac{u^{(k+1)}-u^{(k)}}{\Delta t} +\tau B(\mu^{(k+1)}-\mu^{(k)}) =-A\mu^{(k)} ,\\ \label{RSS_SAV_CH2}
\mu^{(k+1)}-\mu^{(k)}=\tau B(u^{(k+1)}-u^{(k)}) +Au^{(k)} -\mu^{(k)}+\Frac{s^{(k+1)}}{\sqrt{Q_h(F(u^{(k)}))+C_0} }f(u^{(k)}),\\  \label{RSS_SAV_CH3}
\Frac{s^{(k+1)}-s^{(k)}}{\Delta t}=Q_h\left(\Frac{f(u^{(k)})}{2\sqrt{Q_h(F(u^{(k)}))+C_0}},\Frac{u^{(k+1)}-u^{(k)}}{\Delta t}\right).
\end{eqnarray}
We can obtain the following stability result combining the proof of Theorem \ref{theo_CH1} and that of the stability of the SAV scheme as presented in  \cite{JShenXuYang}:
\begin{proposition}\label{theo_CH1_SAV}
Under the hypothesis of Theorem \ref{theo_CH1}, we have the following stability conditions:
If $\tau \ge max(\beta,\Frac{L}{2\epsilon^2 \lambda_{min}(B)}+\Frac{\beta}{2})$,
then the scheme (\ref{RSS_SAV_CH1}) - (\ref{RSS_SAV_CH3})
is unconditionally stable for the modified energy $E_{SAV}$. Here $\lambda_{min}(B)>0$ is the smallest strictly positive eigenvalue of $B$.
In addition, $u^{(k+1)}-u^{(k)} \in W^\perp, \forall k\ge 0$, where $W=Ker(A)=Ker(B)$, in particular, if $W=\{{\bf 1}^T\}$, the mean value of $u^{(k)}$ is conserved.
\end{proposition}
\begin{proof}
We set for simplicity $Q_h(g)=\displaystyle{\sum_{\ell} g_{\ell}h^{d}}$, where $d$ is the dimension of $\Omega$ and $\ell$ the multi-index coordinate of the array $g_{\ell}$ containing the approximation of $g$ at the grid points. We take $Q_h(<\mu^{(k+1)},(\ref{RSS_SAV_CH1})>)$ then   $Q_h(<u^{(k+1)}-u^{(k)},(\ref{RSS_SAV_CH2})>)$. Using both the same majorations as in the proof of Theorem \ref{theo_CH1} and the identity
$$
s^{(k+1)}Q_h\left(\Frac{f(u^{(k)})}{\sqrt{Q_h(F(u^{(k)}))+C_0}},u^{(k+1)}-u^{(k)}\right)=|s^{(k+1)}|^2-|s^{(k)}|^2+|s^{(k+1)}-s^{(k)}|^2,
$$
we obtain the result.
\end{proof}
\item  
The key point in our methods is the choice of the preconditioning matrix $B$ and of the tuning stabilization parameter $\tau$,  the stabilization matrix being here defined as $B_{\tau}=\tau  B$.
In an ideal situation, the stabilization should only act on the high mode components of the solution: indeed,  their speed of propagation determines the time step restriction. Also a strong stabilization of the low mode components of the solution can slow down the dynamics, see \cite{AbboudKosseifiChehab}; this typically arises when taking large values of $\tau$, as also illustrated in the present work for Allen-Chan's or Cahn-Hilliard's pattern dynamics.  In our situation, i.e, when considering two finite differences discretisations $A$ and $B$ of the same operator, typically $-\Delta$, the lower eigenvalues of $B$ are very close to those of $A$ while the high ones are underestimated as respected to the $A$'s ones, see Section 2.4 and more generally \cite{Lele}. So, when $\tau\simeq 1$, $\tau B$ stabilizes the IMEX-scheme without deteriorating the consistency. To enhance much more significantly the stabilization, one faces to the construction of a preconditioner $B_{\tau}$ whose the spectrum is close to the one of $A$ for the small eigenvalues and which controls the high eigenvalues for large values of $\tau$. This question adresses to linear algebra techniques, see e.g. \cite{SaadBook,SaadXi}.
\end{itemize}

\end{document}